\documentclass[intlim,righttag,10pt]{amsart}
\usepackage{enumitem}
\usepackage{amscd}
\usepackage{amssymb}
\usepackage[all]{xy}
\usepackage{fge}
\usepackage{hyperref}
\usepackage{graphics}
\newcommand{\moins}{\mathbin{\fgebackslash}}

\usepackage[textsize=tiny]{todonotes}

\RequirePackage[T1]{fontenc}
\RequirePackage{amsfonts,latexsym,amssymb}

\RequirePackage{mathrsfs}
\let\mathcal\mathscr
\let\cal\mathcal

\usepackage{bbm}

\newtheorem{theorem}[equation]{Theorem}
 \newtheorem{lemma}[equation]{Lemma}
 \newtheorem{proposition}[equation]{Proposition}
 \newtheorem{corollary}[equation]{Corollary}  
\newtheorem{conjecture}[equation]{Conjecture}

\theoremstyle{definition}
\newtheorem{definition}[equation]{Definition}

\newtheorem{remark}[equation]{Remark}
\newtheorem{example}[equation]{Example}
\newtheorem{notation}[equation]{Notation}
\theoremstyle{remark}
\newtheorem*{acknowledgments}{Acknowledgments}

\def\invlim{\mathop{\vtop{\ialign{##\crcr$\hfill{\lim}\hfil$\crcr
\noalign{\kern1pt\nointerlineskip}\leftarrowfill\crcr\noalign
{\kern -3pt}}}}\limits}
\def\dirlim{\mathop{\vtop{\ialign{##\crcr$\hfill{\lim}\hfil$\crcr
\noalign{\kern1pt\nointerlineskip}\rightarrowfill\crcr\noalign
{\kern -3pt}}}}\limits} 
\def\lomapr#1{\smash{\mathop{\relbar\joinrel\longrightarrow}\limits^{#1}}}
 \def\verylomapr#1{\smash{\mathop{\relbar\joinrel\relbar\joinrel\relbar\joinrel\longrightarrow}\limits^{#1}}}

\def\phi{\varphi} 
\def\epsilon{\varepsilon}

\newcommand{\ovk}{\overline{K} }

\newcommand{\dr}{\operatorname{dR} }

 \newcommand{\colim}{\operatorname{colim} }

 \newcommand{\proeet}{\operatorname{pro\acute{e}t} } 
  \newcommand{\qproeet}{\operatorname{qpro\acute{e}t} } 
\newcommand{\eet}{\operatorname{\acute{e}t} }
  
  \newcommand{\cond}{\operatorname{cond} } 

 \newcommand{\Hom}{{\rm{Hom}} }
  
 \newcommand{\Ext}{\operatorname{Ext} }

\newcommand{\Gal}{\operatorname{Gal} }
\newcommand{\can}{ \operatorname{can} }
\newcommand{\synt}{ \operatorname{syn} }
 
\newcommand{\st}{\operatorname{st} }
  
 \newcommand{\coker}{\operatorname{coker} }  
    
 \newcommand{\crr}{\operatorname{cr} }

 \newcommand{\sff}{{\mathcal{F}}}

 \newcommand{\sg}{{\mathcal{G}}}

 \newcommand{\scc}{{\mathcal{C}}}
 
 \newcommand{\sll}{{\mathcal{L}}}
 
 \newcommand{\so}{{\mathcal O}}

\newcommand{\sd}{{\mathcal{D}}} 
\newcommand{\sm}{{\mathcal{M}}}

\newcommand{\wotimes}{\widehat{\otimes}}
 \newcommand{\wt}{\widetilde}
 \newcommand{\wh}{\widehat}
   \numberwithin{equation}{section}

\def\R{{\mathrm R}}
\def\O{{\cal O}}
 \def\A{{\bf A}} \def\B{{\bf B}}
\def\Q{{\bf Q}} \def\Z{{\bf Z}}
\def\C{{\bf C}}
\def\N{{\bf N}}
\def\O{{\cal O}}
\def\G{{\cal G}} 
\def\Qbar{\overline{\bf Q}}

\def\rg{{\rm R}\Gamma}

\def\epsilon{\varepsilon}
\def\dual{{\boldsymbol *}}
\def\bmu{{\boldsymbol\mu}}

\numberwithin{equation}{section}
\begin{document}
\title[Arithmetic duality  for $p$-adic   pro-\'etale cohomology of  analytic curves]
 {Arithmetic duality  for $p$-adic   pro-\'etale cohomology of  analytic curves}
 \author{Pierre Colmez} 
\address{CNRS, IMJ-PRG, Sorbonne Universit\'e, 4 place Jussieu, 75005 Paris, France}
\email{pierre.colmez@imj-prg.fr} 
\author{Sally Gilles}
\address{Max Plack Institute for Mathematics, Vivatsgasse 7,
53111 Bonn, Germany}
\email{gilles.mpim-bonn.mpg.de}
\author{Wies{\l}awa Nizio{\l}}
\address{CNRS, IMJ-PRG, Sorbonne Universit\'e, 4 place Jussieu, 75005 Paris, France}
\email{wieslawa.niziol@imj-prg.fr}
 \date{\today}
\thanks{The authors' research was supported in part by the grant NR-19-CE40-0015-02 COLOSS and by the grant No. DMS-1928930 
of the National Science Foundation  while P.C. and W.N. were  in
residence at the Mathematical Sciences Research Institute in Berkeley, California, during the Spring 2023 semester.}
 \begin{abstract}
 We prove a Poincar\'e duality for arithmetic $p$-adic pro-\'etale cohomology of smooth dagger curves over 
finite extensions of $\Q_p$.
 We deduce it, via the Hochschild-Serre spectral sequence,  from geometric comparison theorems 
combined with Tate and Serre dualities.
 The compatibility of all the products involved is checked via reduction to the ghost circle, for which we also prove a Poincar\'e duality (showing that it behaves like a proper smooth analytic variety of dimension $1/2$).
 Along the way we  study functional analytic properties of arithmetic $p$-adic pro-\'etale cohomology and prove that the usual cohomology is nuclear Fr\'echet and the compactly supported one --  of compact type.
 \end{abstract}
\maketitle
\tableofcontents

\section{Introduction}
This paper is a contribution to the study of dualities in $p$-adic pro-\'etale cohomology of analytic varieties. 
    
  \subsection{Statement of the main theorem}  
Let $p$ be a prime and let $K$ be a finite extension of $\Q_p$. Our main theorem  is  the following duality result. 
\begin{theorem}{\rm (Arithmetic Poincar\'e duality)} \label{main-arithmetic0} Let $X$ be a smooth, geometrically irreducible,  dagger variety\footnote{Dagger varieties, introduced by Grosse-Kl\"onne in \cite{GK}, are rigid analytic varieties with overconvergent structure sheaves.} of dimension $1$ over $K$. Assume that $X$ is proper, Stein, or affinoid. Then: 
\begin{enumerate}
\item
There is a natural trace map isomorphism of solid $\Q_p$-vector spaces\footnote{See Chapter \ref{def-compact} for a definition of compactly supported pro-\'etale cohomology.}
\begin{equation}
\label{trace0}
{\rm Tr}_X: \, H^4_{\proeet,c}(X,\Q_p(2))\stackrel{\sim}{\to} \Q_p. 
\end{equation}
\item
The  pairing
$$
H^i_{\proeet}(X,\Q_p(j))\otimes^{\Box}_{\Q_p} H^{4-i}_{\proeet,c}(X,\Q_p(2-j)){\to } H^4_{\proeet,c}(X,\Q_p(2))\xrightarrow[\sim]{{\rm Tr}_X}\Q_p
$$
is a perfect duality, i.e., we have  induced  isomorphisms of solid $\Q_p$-vector spaces 
\begin{align*}
 &\gamma_{X,i}:  H^i_{\proeet}(X,\Q_p(j))\stackrel{\sim}{\to} H^{4-i}_{\proeet,c}(X,\Q_p(2-j))^*,\\
& \gamma_{X,i}^c: H^i_{\proeet,c}(X,\Q_p(j))\stackrel{\sim}{\to} H^{4-i}_{\proeet}(X,\Q_p(2-j))^*,
\end{align*}
where $(-)^*$ denotes $\underline{\Hom}_{\Q_p}(-,\Q_p)$. 
\end{enumerate}
\end{theorem}
\begin{remark} 
{\rm (i)}
If the curve $X$ is proper then the  cohomology  groups $H^i_{\proeet}(X,\Q_p(j))$ are of finite 
rank over $\Q_p$; 
this is also the case if $j\neq 1$ and
if $X$ is affinoid (or Stein with $H^1_{\rm dR}(X)$ finite dimensional).

{\rm (ii)} If $X$ is Stein, $H^i_{\proeet}(X,\Q_p(j))$ is a 
nuclear\footnote{Nuclear in the classical sense not in the solid sense.} Fr\'echet 
and $H^i_{\proeet,c}(X,\Q_p(j))$ 
is a space of compact type; if $X$ is a dagger affinoid -- it is the other way around. 
\end{remark}
\begin{remark}If $X$ is proper or Stein we have a more general derived duality: the duality map 
    \begin{equation}
    \label{derived11}
     {\gamma_{{X}}}: \R\Gamma_{\proeet}({X},\Q_p(j)) \xrightarrow{} {\mathbb D}(\R\Gamma_{\proeet,c}({X},\Q_p(2-j))[4]),
\end{equation}
where ${\mathbb D}(-)=\R\underline{\Hom}_{\Q_p}(-,\Q_p)$, is a quasi-isomorphism in $\sd(\Q_{p,\Box})$.
 The isomorphisms $\gamma_{X,i}$ are obtained from it because all the higher $\Ext$ groups involved vanish since $H^i_{\proeet,c}(X,\Q_p(j))$ is of compact type.
 The isomorphisms $\gamma^c_{X,i}$ are obtained from $\gamma_{X,i}$ by dualizing using the fact that the spaces $H^i_{\proeet,c}(X,\Q_p(j))$ are (solid) reflexive. 

  \end{remark}
  We venture to conjecture  an  arithmetic Poincar\'e duality in any dimension:
  \begin{conjecture}
Let $X$ be a smooth Stein dagger variety over $K$, geometrically irreducible, of dimension $d$. Then: 
\begin{enumerate}
\item The cohomology groups $H^i_{\proeet}(X,\Q_p(j))$ and $H^i_{\proeet,c}(X,\Q_p(j))$ are nuclear Fr\'echet and of compact type, respectively. 
\item  We have (quasi-)isomorphisms in $\sd(\Q_{p,\Box})$
\begin{align*}
\R\Gamma_{\proeet}(X,\Q_p(j)) & \simeq {\mathbb D}(\R\Gamma_{\proeet,c}(X,\Q_p(d+1-j))[2d+2]),\\
H^i_{\proeet}(X,\Q_p(j)) & \simeq H^{2d+2-i}_{\proeet,c}(X,\Q_p(d+1-j))^*,\\
H^i_{\proeet,c}(X,\Q_p(j)) & \simeq H^{2d+2-i}_{\proeet}(X,\Q_p(d+1-j))^*.
\end{align*}
\end{enumerate}
\end{conjecture}
  \begin{example}
The starting point of our study of dualities for $p$-adic pro-\'etale cohomology of analytic spaces was the  computation for $X=D$, the open unit disc. For example, we computed that:
  \begin{align*}
H^1_{\proeet}(X,\Q_p(1)) & \simeq (\so(D)/K)\oplus H^1(\sg_K,\Q_p(1)),\\
 H^3_{\proeet,c}(X,\Q_p(1)) & \simeq (\so(\partial D)/\so(D))\oplus H^1(\sg_K,\Q_p),
\end{align*}
where $\sg_K=\Gal(\overline{K}/K)$ and $\partial D:=\varprojlim_n D\moins D_n$ 
is the ``ghost circle'', 
boundary of the unit disk
(here, $D_n$ is the closed ball $\{v_p(z)\geq \frac{1}{n}\}$).  
The splitting is not canonical but it is compatible with products. 
This computation  showed   that there is a duality of the form stated in Theorem \ref{main-arithmetic0} and, moreover, that it is induced by Galois duality and coherent duality (i.e., Serre duality):
\begin{align*}
H^i(\sg_K,\Q_p) & \simeq H^{2-i}(\sg_K,\Q_p(1))^*,\\
H^0(D,\Omega^1_D) & \simeq H^1_c(D,\so_D)^*.
\end{align*}
We used here the isomorphisms (the first one is almost the
definition of $H^1_c(D,\so_D)$, taking into account vanishing of
$H^2(D,\so_D)$; the second is induced by $f\mapsto df$: that it is an isomorphism
is equivalent to the fact that $H^1_{\rm dR}(D)=0$):
\begin{equation}
\label{kichasz1}
H^1_c(D,\so_D)\stackrel{\sim}{\leftarrow} \so(\partial D)/\so(D),\quad H^0(D,\Omega^1_D)\stackrel{\sim}{\leftarrow} \so(D)/K.
\end{equation}
Since $[K:\Q_p]<\infty$, the above coherent $K$-duality can be turned into a $\Q_p$-duality by
composing with ${\rm Tr}_{K/\Q_p}$.
\end{example}
\begin{remark} 
The above results for the open disk do not involve the solid formalism, and
we attempted at first to write this paper using classical functional analysis as in \cite{CDN3}.
 We passed to solid formalism  because of two reasons.
 One was a need for a derived dual that would  work with Mayer-Vietoris sequences and which does not seem to exist in the classical world in the generality we wanted.
 The second one was 
a construction of topological Hochschild-Serre spectral sequences which we did not know how to do in the classical set-up. 
\end{remark}
\subsection{The proof of Theorem \ref{main-arithmetic0}} 
An analog of Theorem \ref{main-arithmetic0} holds for schemes and the proof uses Galois descent from   geometric, i.e., over $C:=\wh{\overline{K}}$, Poincar\'e duality and   Galois duality (the Hochschild-Serre spectral sequence does not degenerate at $E_2$ but it does degenerate at $E_3$).
 An analogous argument, starting with geometric Poincar\'e duality due to Zavyalov \cite{Zav21}, Gabber, and Mann \cite{Mann}, yields Theorem \ref{main-arithmetic0} for  proper rigid analytic varieties (in fact, in any dimension). 

  We use here a similar strategy.
 In our setting we do not yet have  the geometric duality so we replace it with comparison isomorphisms.
 They allow us to pass from 
  geometric $p$-adic pro-\'etale cohomology to de Rham data and the terms of the Hochschild-Serre spectral sequence can be identified with coherent cohomology and Galois cohomology of $\Q_p$-vector spaces with some finiteness properties.
 Our Poincar\'e duality is now deduced, via this spectral sequence, from coherent and Galois dualities. 
  
$\bullet$ {\em Geometric comparison results.} 
Recall that,  for smooth Stein rigid analytic varieties over $K$,
 we have  comparison theorems between geometric $p$-adic pro-\'etale cohomology and geometric syntomic cohomology.
 In the case of usual cohomology they are due to  Colmez-Dospinescu-Nizio{\l} \cite{CDN3} and Colmez-Nizio{\l}  \cite{CN5}; in the case of compactly supported cohomology --  to Achinger-Gilles-Nizio{\l} \cite{AGN}.

  In the case where $X$ is a smooth, geometrically irreducible, Stein curve over $K$, they yield:
\begin{enumerate}
\item  Vanishings: 
\begin{align}\label{vanishing1}
 & H^i_{\proeet}(X_C,\Q_p)=0, \quad \mbox{ for } i\neq 0,1,\\
& H^i_{\proeet,c}(X_C,\Q_p)=0, \quad \mbox{ for } i\neq 1,2.\notag
\end{align}
\item Isomorphisms: 
$$H^0_{\proeet}(X_C,\Q_p)\simeq \Q_p,\quad  
H^1_{\proeet,c}(X_C,\Q_p(1)) \xrightarrow{\sim}{\rm HK}^1_{c}(X_C,1),$$
where\footnote{Here  $\breve{F}$ is the completion of the maximal unramified extension of the fraction field $F$ of the Witt vectors of the residue field of $K$ and $\wh{\B}^+_{\st}$ is the semistable period ring of Fontaine in its Banach form.} 
 $ {\rm HK}^j_{*}(X_C,i) : =(H^j_{{\rm HK},*}(X_C){\otimes}^{\Box}_{\breve{F}}\wh{\B}^+_{\st})^{N=0,\phi=p^i}$ and $H^i_{{\rm HK},*}(X_C)$, for $*=[\phantom{x}],c$, is the Hyodo-Kato cohomology. 
\item  Exact sequences: 
\begin{align}\label{comp1}
   0\to  \so(X_C)/C\to   & H_{\proeet}^1(X_C,\Q_p(1))\to {\rm HK}^1(X_C,1) \to 0 \\
 {\rm HK}^1_{c}(X_C,2) \to H^{1}{\rm DR}_{c}(X_C,2)\to  & H^2_{\proeet,c}(X_C,\Q_p(2))\to\Q_p(1)\to 0,\notag
\end{align}
where  we set 
\begin{align*}
& {\rm DR}_{c}(X_C,i) : =  (H^1_{c}(X,\so_X){\otimes}^{\Box}_{K}(\B^+_{\dr}/F^i)\to H^1_{c}(X,\Omega^1_X){\otimes}^{\Box}_{K}(\B^+_{\dr}/F^{i-1}))[-1],
\end{align*}

\end{enumerate}

The last sequence yields the definition of the geometric trace map $${\rm Tr}_{X_C}: H^2_{\proeet,c}(X_C,\Q_p(1))\to  \Q_p.$$  
The arithmetic trace map \eqref{trace0} 
is defined  as the composition
$$ {\rm Tr}_{X}: H^{4}_{\proeet,c}(X,\Q_p(2))\simeq H^2(\sg_K,H^{2}_{\proeet,c}(X,\Q_p(2)))\lomapr{ {\rm Tr}_{X_C}(1)} H^2(\sg_K,\Q_p(1))\lomapr{{\rm Tr}_K}\Q_p.
$$

{$\bullet$} {\em Galois descent.} For $s\in\Z$, we have Hochschild-Serre spectral sequences
\begin{equation}\label{HS0}
E^{i,j}_2=H^i(\sg_K,H^j_{\proeet,*}(X_C,\Q_p(s)))\Rightarrow H^{i+j}_{\proeet,*}(X,\Q_p(s)).
\end{equation}
The  computations  done above show that to understand the terms of the spectral sequences \eqref{HS0} we need to understand Galois cohomology of 
\begin{align}
\label{niedziela2}
H^0(X,\Omega^1_{X})\otimes^{\Box}_KC(s),\quad H^1_{c}(X,\so_X){\otimes}^{\Box}_{K}C(s),\quad  {\rm HK}^1_{c}(X_C,1)(s).
\end{align}
We used here   the isomorphisms \eqref{kichasz1}. 
We claim that,  in the case $H^1_{{\rm HK}}(X_C)$ is of finite rank over $\breve{F}$,   this  cohomology  is or trivial or  isomorphic to
\begin{equation}\label{terms1}
H^0(X,\Omega^1_{X}),\quad H^1_{c}(X,\so_X),\quad H^i(\sg_K, V),
\end{equation}
where $V$ is a finite rank  $\Q_p$-Galois representation. Indeed, note that the last group in \eqref{niedziela2}  is an almost $C$-representation~\cite{Fo-Cp}
  (a finite rank $C$-vector space plus/minus a finite rank $\Q_p$-vector space). Our claim then follows 
 easily from a generalization of  Tate's computations:
$$
H^i(\sg_K, C(j))\simeq \begin{cases} K & \mbox{ if } j=0, i=0,1;\\
0 & \mbox{ otherwise.}
\end{cases}
$$

 By the vanishing results \eqref{vanishing1}, the spectral sequences \eqref{HS0} have only two nontrivial columns,  hence degenerate at $E_2$.
 Looking at the terms of the spectral sequences in \eqref{terms1}, we see that  if we knew that the arithmetic pro-\'etale product is compatible with the coherent and Galois products we would have our duality (at least in the case the Hyodo-Kato cohomology is of finite rank but we can always reduce to that case).
 However, we have found  it very difficult to check this compatibility  for a general curve $X$ as above: the main difficulty arises from the exact sequence \eqref{comp1} that does not behave well with respect to products.
 We decided thus to pass from $X$ to simpler curves. 
 
 Nevertheless,  the above computations imply the following functional analytic result:
 \begin{theorem}
 Let $X$ be a smooth Stein dagger curve over $K$, geometrically irreducible. Then
the cohomology groups $H^i_{\proeet}(X,\Q_p(j))$ and $H^i_{\proeet,c}(X,\Q_p(j))$ are nuclear Fr\'echet and of compact type, respectively. 
 \end{theorem}
 This is because we know these properties for $H^0(X,\Omega^1_{X})$ and  $H^1_{c}(X,\so_X)$, respectively, hence for the terms of the Hochschild-Serre spectral sequences, and the extension problems arising from the spectral sequences can be solved. 

{$\bullet$} {\em Reductions.} By a simple limit argument, to prove the derived Poincar\'e duality \eqref{derived11}, we can pass from a general Stein curve to a wide open curve, i.e., to a complement of a finite number of closed discs in a proper curve (we allow for a finite field extension here). For wide open curves we use a Mayer-Vietoris argument to reduce to proper curves, open discs, and open annuli. For proper curves we know the result, for open discs and annuli we reduce the computation to the one of their boundaries. 

{$\bullet$} {\em Arithmetic duality for ghost circle.} We show that, in duality theory, the boundary of an open disc, a ghost circle,  behaves like a proper rigid analytic variety over $K$ of dimension $1/2$.
\begin{theorem}{\rm (Arithmetic duality for ghost circle)}\label{main-arithmeticY0} Let $D$ be an     open disc over $K$.  Let $Y:=\partial D$ be the boundary of $D$. Then: 
\begin{enumerate}
\item  There is a natural trace map isomorphism of  solid $\Q_p$-vector spaces
$${\rm Tr}_Y: \, H^3_{\proeet}(Y,\Q_p(2))\stackrel{\sim}{\to} \Q_p. 
$$
\item The  pairing
\begin{equation}
\label{pairing11}
H^i_{\proeet}(Y,\Q_p(j))\otimes^{\Box}_{\Q_p} H^{3-i}_{\proeet}(Y,\Q_p(2-j))\to H^3_{\proeet}(Y,\Q_p(2))\stackrel{\sim}{\to}\Q_p
\end{equation}
is a perfect duality, i.e., we have the  induced  isomorphism of solid $\Q_p$-vector spaces
\begin{align*}
 & \gamma_Y: \quad H^i_{\proeet}(Y,\Q_p(j))\stackrel{\sim}{\to} H^{3-i}_{\proeet}(Y,\Q_p(2-j))^*.
\end{align*}
\end{enumerate}
\end{theorem}
We will now sketch the proof of this theorem. We define an ascending filtration on $H^{i}_{\proeet}(Y,\Q_p(j))$:
$$
F^2_{i,j}=H^{i}_{\proeet}(Y,\Q_p(j))\supset F^1_{i,j}\supset F^0_{i,j}\supset F^{-1}_{i,j}=0,
$$
such that 
\begin{align*}
 & F^2_{i,j}/F^1_{i,j}\simeq H^{i-1}(\sg_K, \Q_p(j-1)),\\
& F^1_{i,j}/F^0_{i,j}\simeq H^{i-1}(\sg_K, \so(Y_C)/C(j-1)),\\
  & F^0_{i,j}\simeq H^i(\sg_K, \Q_p(j)).
\end{align*}
It is helpful to visualize this filtration in the following way (with $H^i(-\Q_p(j)):=H^i_{\proeet}(-\Q_p(j))$):
$$
\xymatrix@R=6mm@C=15pt{ & &0\ar[d]  & 0\ar[d]\\
0\ar[r] &  F^0_{i,j}:=H^{i}(\sg_K, \Q_p(j))\ar[r] \ar@{=}[d] & F^1_{i,j} \ar[r] \ar[d] & H^{i-1}(\sg_K, \frac{\so(Y_C)}{C}(j-1))\ar[d] \ar[r] & 0\\
0\ar[r] &  H^{i}(\sg_K, \Q_p(j))\ar[r] & F^2_{i,j}:=H^i(Y,\Q_p(j))\ar[r] \ar[d] &  H^{i-1}(\sg_K, H^{1}(Y_C,\Q_p(j)))\ar[r] \ar[d]& 0 \\
& &   H^{i-1}(\sg_K,\Q_p(j-1) )\ar@{=}[r] \ar[d] &  H^{i-1}(\sg_K,\Q_p(j-1))\ar[d]\\
& & 0 & 0
}
$$
 The middle exact row comes from the filtration induced by the Hochschild-Serre spectral sequence; the right exact column is induced by the syntomic filtration from the analog\footnote{Take this sequence for an annuli and go to the limit towards the boundary.} of the exact sequence \eqref{comp1}.  The term $F^1_{i,j}$ is defined as a pullback of the top right square. 

 Now, the key computation is the following:
\begin{theorem}{\rm (Explicit Reciprocity Law)}
We have: 
\begin{enumerate}
\item  The pairing (\ref{pairing11}) is compatible with the above filtration.
\item  On the associated grading the pairing  (\ref{pairing11}) yields a pairing induced by the Galois cohomology pairing and coherent pairing. 
\end{enumerate}
\end{theorem}
The proof of this Theorem interpolates between syntomic and $(\phi,\Gamma)$-techniques. 

  This finishes the proof of Theorem \ref{main-arithmeticY0} and hence of Theorem \ref{main-arithmetic0}, the main theorem of this paper.
\subsection{Conjectural geometric duality} 
We finish the introduction with a discussion of what we think  happens over $C$. 
Let us start with the example that guided our study. 
\begin{example}
Let $D$ be the open unit disc over $C$.  The nontrivial cohomology groups are:
\begin{align*}
H^0_{\proeet}(D,\Q_p(j))  \simeq \Q_p(j),\quad &
H^1_{\proeet}(D,\Q_p(j))  \simeq (\so(D)/C)(j-1),\\
H^2_{\proeet,c}(D,\Q_p(j)) & \simeq \Q_p(j-1)\oplus(\so(\partial D)/\so(D))(j-1).
\end{align*}
It looks like we have the right groups for a duality (for appropriate choices of $j$):
$\so(D)/C$ and $\so(\partial D)/\so(D)$ are $C$-vector spaces
in duality via coherent duality and $\Q_p(j)$ and $\Q_p(j-1)$ can be put in duality.
But the degrees are wrong for a Poincar\'e-type duality to work: 
coherent duality adds up to degree $3$ but $\Q_p$-vector space duality adds up to degree $2$. 
Also, since $[C:\Q_p]=\infty$, it is impossible to turn a $C$-duality into a $\Q_p$-duality.
\end{example}
To find a set-up in which a duality could be restored,
 we turned to the category of Vector Spaces, 
i.e., $v$-sheaves of topological $\Q_p$-vector spaces on the category ${\rm Perf}_C$
of perfectoid spaces over $C$ (that it is not unreasonnable to do so
is due to the fact that there is a natural way~\cite{CN4}
 to turn pro-\'etale geometric cohomology groups into VS's). 
There we have the following computation~\cite[Prop.\,10.16]{Esp}
 of $\Ext$-groups:
\begin{align*}
\Hom_{\rm VS}(\Q_p,\Q_p(1))\simeq \Q_p(1),\quad \Ext^1_{\rm VS}(\Q_p,\Q_p(1))=0,\\
\Hom_{\rm VS}({\mathbb G}_a,\Q_p(1))=0,\quad \Ext^1_{\rm VS}({\mathbb G}_a,\Q_p(1))\simeq C.
\end{align*}
Moreover, $\Ext^i$, for $i\geq 2$, vanish by ~\cite[Th. 3.8]{ALB}. 
The nontrivial $\Ext$-group in the second row is generated 
by the fundamental exact sequence of Banach-Colmez spaces
$$
0\to \Q_p(1)\to {\mathbb B}_{\crr}^{+,\phi=p}\to {\mathbb G}_a\to 0
$$

  The above computations, ignoring functional analytic questions, yield a Verdier duality isomorphisms
(with $H^i_{*}(-\Q_p(j)):=H^i_{{\proeet},*}(-\Q_p(j))$): 
\begin{align*}
 \Ext^1_{\rm VS}(H^2_{c}(D,\Q_p(2-j)),\Q_p(1))\stackrel{\sim}{\to}& H^1(D,\Q_p(j)),\\
0\to \Ext^1_{\rm VS}(H^1(D,\Q_p(2-j)),\Q_p(1))\to  & H^2_{c}(D,\Q_p(j))\to \Hom_{\rm VS}( H^0(D,\Q_p(2-j)),\Q_p(1))\to 0 
\end{align*}
Guided by this, we venture a conjectural statement of geometric duality for $p$-adic pro-\'etale cohomology in any dimension (again, ignoring topology):
\begin{conjecture}{\rm (Geometric Verdier duality)}
Let $X$ be a smooth Stein rigid analytic  variety over $C$, connected, of dimension $d$. There is a natural quasi-isomorphism
$$
\R\Gamma_{\proeet}(X,\Q_p(j))\simeq \R\Hom_{\rm VS}(\R\Gamma_{\proeet,c}(X,\Q_p(d+1-j))[2d],\Q_p(1)).
$$
\end{conjecture}
The interested reader can find more details in \cite{Ober}.
\begin{acknowledgments} \ We would like to thank 
Guido Bosco, Juan Esteban Rodriguez Camargo, Matthew Emerton,  Akhil Mathew, and Peter Scholze for help concerning solid mathematics. Special thanks go to Gabriel Dospinescu for many helpful discussions involving functional analysis in general. 
We also thank Guido Bosco and Zhenghui Li for helpful comments on an earlier version of this paper. 

  This material is based upon work supported by the National Science Foundation under Grant No. 1440140, while the first and the third authors were in residence at the Mathematical Sciences Research Institute in Berkeley, California, during the Spring semester of 2023. Parts of this paper were written during the first and  third author's stay at the Hausdorff Research Institute for Mathematics in Bonn, in the Summer of 2023. We would like to thank the Institute for support and hospitality. 
  
  The second author was partially supported by the ERC under the European Union's Horizon 2020 research and innovation programme (grant agreement No. 804176) and through the fellowship from the MPIM in Bonn. 
\end{acknowledgments}
 \subsubsection*{Notation and conventions.}\label{Notation}
Let $K$ be a finite extension of $\Q_p$, with ring of integers
$\so_K$ and residue field $k$.
Let $\ovk$ be an algebraic closure of $K$ and let $\so_{\ovk}$ denote the integral closure of $\so_K$ in $\ovk$.   Set $\sg_K=\Gal(\overline {K}/K)$. Let $C=\wh{\ovk}$ be the $p$-adic completion of $\ovk$.  Let
$W(k)$ be the ring of Witt vectors of $k$ with 
 fraction field $F$ (i.e., $W(k)=\so_F$); let $e=e_K$ be the ramification index of $K$ over $F$. 
Let $\breve{F}:=W(\overline k)[\frac{1}{p}]$ 
denote the completion of the maximal unramified extension of $F$ and 
let $\phi$ be the absolute
Frobenius on $W(\overline {k})$.

  We will denote by $\A_{\rm inf}, \A_{\crr}, \B_{\crr}, \wh{\B}_{\st},\B_{\dr}$ the Witt, crystalline, semistable, and  de Rham period rings of Fontaine, respectively. 


  All rigid analytic spaces and dagger spaces considered will be over $K$ or $C$.  We assume that they are separated, taut, and countable at infinity. 

Since the only cohomology with $\Q_p(r)$-coefficients that we consider is the pro-\'etale
cohomology, we will remove ``${\proeet}$'' from the notations, i.e.~write:
$$H^i(-,\Q_p(r)):=H^i_{\proeet}(-,\Q_p(r)),\quad H^i_c(-,\Q_p(r)):=H^i_{{\proeet},c}(-,\Q_p(r)).$$

\section{Functional analysis}
\subsection{Classical functional analysis} 
We gather here some basic facts from classical $p$-adic functional analysis that we use in the paper. 
 \subsubsection{Locally convex vector spaces} Our cohomology groups will be equipped with a canonical topology. To talk about it in a systematic way,  we will work  in the category $C_K$ of locally convex $K$-vector spaces. For details  the reader may consult \cite[Sec.\,2.1, 2.3]{CDN3}. To summarize quickly:
    $C_K$ 
   is a quasi-abelian category.
   We will denote the category of left-bounded complexes of $C_K$ by $C(C_K)$ and the associated derived $\infty$-category  by 
$\sd(C_K)$.
  A morphism of complexes that is a quasi-isomorphism in $\sd(C_K)$, i.e., its mapping cofiber is strictly exact,  will be called a {\em strict quasi-isomorphism}.
 The associated cohomology objects are denoted by\footnote{$LH$ stands for ``left heart''.} $\wt{H}^n(E)\in {LH}(C_K)$:
\begin{equation}
\label{quasiabelian}
 \wt{H}^n(E):=\tau_{\leq n}\tau_{\geq n}(E)=({\rm coim}(d_{n-1})\to \ker (d_n)),\quad E\in C(C_K);
\end{equation} they are called {\em classical} if the canonical map $\wt{H}^n(E)\to {H}^n(E)$ is an isomorphism\footnote{In our situations this is usually equivalent to $H^n(E)$ being separated.}. Classical objects are closed under extensions in $
{\rm LH}(C_K)$.
\subsubsection{Hausdorff locally convex vector space} Our main reference here is \cite[Sec. 3.1]{pros}. We will denote by $C^H_K$ the category of Hausdorff locally convex $K$-vector spaces. It is stable under direct sums and direct products and if $V\subset W$ are two locally convex $K$-vector spaces then $W/V$ is Hausdorff if and only if $V$ is closed in $W$. 
The category $C^H_K$ is quasi-abelian: kernels and coimages are defined as in $C_K$; cokernels and images are defined using closures of images in $C_K$. A sequence in $C^H_K$ is strictly exact if and only if it is strictly exact in $C_K$. 
 \subsubsection{Duality} If $V, W\in C_{K}$, we   denote by $\sll(V,W)$ the space of continuous linear maps from $V$ to $W$. We write $\sll_s(V,W)$ and $\sll_b(V,W)$ for the vector space $\sll(V,W)$ equipped with the weak and strong topologies (i.e., the topology of pointwise convergence and the topology of bounded convergence), respectively. 
  For $V\in C_{K}$, we set $V^{\prime}_s:=\sll_s(V,K)$ and $V^*:=V^{\prime}_b:=\sll_b(V,K)$. Both of these dual spaces are Hausdorff. 
   If $V$ is a Hausdorff  space it is called {\em reflexive} if the duality map $V\to (V^{\prime}_b)^{\prime}_b$ is a topological isomorphism.
   
   We will also use  {\em stereotype dual}  $V^{\star}$ of $V\in C_K$. It is defined as $\sll(V,K)$ equipped with the topology of  compactoid  convergence\footnote{See Section \ref{compact1} for a definition of compactoids.} 
   of $V$.       We have continuous maps
  $$
  V^{\prime}_b\to V^{\star}\to V^{\prime}_s.
  $$
  If $V$ is a Banach space then they are not topological isomorphisms unless $V$ has finite dimension. 
   For any Banach space,  its stereotype dual is a Smith space\footnote{A {\em Smith space} is a complete compactly generated locally convex  vector space $V$  having a universal compact set, i.e., a compact set $K$, which absorbs every other compact set $T\subseteq V$ (i.e., $T\subseteq \lambda \cdot K$, for some $ \lambda >0$). }, and vice versa, for any Smith space,  its stereotype dual  is a Banach space,
and $(V^\star)^\star=V$ if $V$ is Banach or Smith. 
If we write a Banach space $V\in C_K$ as $V\simeq (\wh{\oplus_I\so_K})[1/p]$ then its stereotype dual is the Smith space $V^{\star}\simeq(\prod_I\so_K)[1/p]$. A Banach space is a Smith space if and only if it is finite dimensional. 
  

\subsubsection{Hausdorff compactly generated locally convex vector spaces} Recall that a topological space $T$ is  {\em compactly generated} if a map $f:T\to T^{\prime}$ to another topological space is continuous as soon as the composite $S\to T\to T^{\prime}$ is continuous for all compact Hausdorff spaces $S$ mapping to $T$. The inclusion of compactly generated spaces into all topological spaces admits a right adjoint $T\mapsto T^{\rm cg}$ that sends a topological space $T$ to its underlying set equipped with the quotient topology for the map $\coprod_{S\to T}S\to T$, where the disjoint union runs over all compact Hausdorff spaces $S$ (alternatively, profinite sets $S$) mapping to $T$.
  
  Any first-countable space  (in particular, any metrizable topological space) is compactly generated (see \cite[Remark 1.6]{Sch19} for a proof). So, for example, Fr\'echet spaces are compactly generated. The category of compactly generated spaces is closed under taking coproducts, closed subspaces and quotients by closed subspaces.  Hence a colimit of Fr\'echet spaces is compactly generated if it is Hausdorff; this applies, in particular, to locally convex vector spaces of compact type\footnote{See Section \ref{compact1} below for the definition of spaces of compact type.} since they are Hausdorff and can be written as a countable colimit of Banach spaces.
  
  The category $C^{\rm Hcg}_K$ of Hausdorff compactly generated  locally convex vector spaces over $K$  is quasi-abelian: kernels and coimages are defined as in $C_K$; cokernels and images are defined using closures of images in $C_K$. A sequence in $C^{\rm Hcg}_K$ is strictly exact if and only if it is strictly exact in $C_K$.  
  
  \subsubsection{Spaces of compact type and nuclear spaces} \label{compact1} Our references for this section are \cite[IV.19]{Schn}, \cite[Ch. 1]{ST}, \cite[Ch. 8]{PS}. 
  \begin{definition}Let $V,W\in C_{K}$ be Hausdorff. 
  \begin{enumerate}
  \item A subset $B\subset V$ is called {\em compactoid} if for any open lattice $L\subset V$ there are finitely many vectors $v_1, \ldots, v_m\in V$ such that $B\subset L + \so_K v_1 +\cdots + \so_K v_m$. Compactoids are preserved by maps in $C_K$. 
  \item A continuous linear map $f: V\to W$  is called {\em compact} if there is an open lattice $L$ in $V$ such that the closure of $f(L)$ in $W$ is compactoid and complete. If $W$ is quasi-complete (in particular, if  it is complete) then this is equivalent to $f(L)$ being compactoid.
  \item $V$ is called of {\em compact type} if  it is the inductive limit of a sequence 
  $$
  V_1\stackrel{\iota_1}{\to} V_2\stackrel{\iota_2}{\to} V_3 \to \cdots
  $$
  of  Hausdorff  spaces $V_n\in C_K$, for $n\in\N$, with injective compact linear maps. By the proof of \cite[Prop. 16.10]{Schn}, we may assume the $V_n$'s to be Banach spaces. 
  \end{enumerate}
  \end{definition}
   We will use often the following facts (see  \cite[Rem. 16.7]{Schn}). 
   \begin{lemma} 
 Let $g: V\to W$ be a compact map.
Then
\begin{enumerate}
\item If $h: V_1\to V$ and $f: W\to W_1$ are arbitrary continuous linear maps then the map $fgh: V_1\to W_1$ is compact. 
\item  If the image of $g$ is contained in a closed subpace $W_0\subset W$ then the induced map $g:V\to W_0$ is compact. 
\end{enumerate}
\end{lemma}

   We will denote by  $C_{c,K}$  the full subcategory of $C_K$ consisting of spaces of compact type. 
Spaces of compact type are Hausdorff, complete, 
   and reflexive. Their strong duals are Fr\'echet  and satisfy $V^{\prime}_b=\invlim_n(V^{\prime}_n)_b$. Moreover, 
      a closed subspace of a space of compact type is also of compact type and so is the relevant quotient; $C_{c,K}$ is closed under countable direct sums. 
   
   \begin{definition}\label{nuclear1}
  The space $V\in C_K$ is called {\em nuclear} if for any open lattice $L\subset V$ there exists another open lattice $M\subset L$ such that the canonical map between completions $\wh{V}_M\to \wh{V}_L$ is compact
(i.e., the image of $M\to L/p^L$ is, for any $n$, contained in a module of finite type over $\O_K$).
     \end{definition}
     Nuclear Banach spaces are finite dimensional. Nuclear Fr\'echet spaces are reflexive.  Moreover, a subspace of a nuclear space is nuclear  and if this subspace is closed the relevant quotient  is nuclear. A countable inductive limit of nuclear spaces is nuclear (see \cite[Th. 8.5.7]{PS}). By loc. cit., projective limits 
   of nuclear spaces are nuclear.  Any 
   compact projective or inductive limit\footnote{In the sense of \cite[Cor. 16.6, Prop. 16.10]{Schn} .} of   locally convex $K$-vector spaces is nuclear.  A Fr\'echet space  is the strong dual of a space of compact type if and only if it is nuclear. 
   
   The following  lemma will be essential for us. 
    \begin{lemma}\label{ext10}
     \begin{enumerate}
     \item Every strict exact sequence of spaces from $C_K$ 
   $$
     0\to V\to W\to W'\to 0
$$
     such that $V$ is of finite rank and $W'$ is Hausdorff   splits (and $W\simeq V\oplus W'$). 
      \item In a strict extension of spaces from $C_K$ 
     \begin{equation}
\label{bonn12}
\xymatrix{
0\ar[r]&  V\ar[r] &  W\ar[r] &  W'\ar[r] &0,
}
\end{equation}
where  $V$ is of finite rank over $K$ and $W'$ is nuclear Fr\'echet, $W$ is also nuclear Fr\'echet. Same holds for spaces of compact type. 
\end{enumerate}
\end{lemma}
\begin{proof} 

  In the first claim, we may assume that $V$ is of rank $1$.  Note that $W$ is Hausdorff. Take   a continuous
linear form $\lambda:W\to K$,  not identically $0$ on $V$ (such a form exists by \cite[Cor. 9.3]{Schn}). It suffices to show that 
the canonical map $V\oplus {\rm Ker}\,\lambda\to W$ is a topological  isomorphism. 

  It is an algebraic isomorphism and it is continuous. Hence we just need to check
that it is open.  So, let $U_\lambda$ be an open lattice in ${\rm Ker}\,\lambda$ and let $V_0$
be an open lattice in $V$. We need to  show  that $U_\lambda+V_0$ is open
in $W$ and it is enough to do it 
for open sub-lattices of $U_\lambda$ and $V_0$.

  Any
neighborhood of $0$ in ${\rm Ker}\,\lambda$ contains an open of the
form $U\cap({\rm Ker}\,\lambda)$, where $U$ is an open lattice
in $W$. Hence   we can assume $U_\lambda$ to be of this form. By construction, 
we have an exact sequence $0\to U_\lambda\to U\to\lambda(U)\to 0$.

   Now, $U$ contains $p^NV_0$ for $N$ big enough, since it is a lattice.
Then $U_\lambda+p^NV_0$ is the inverse image of $p^N\lambda(V_0)$
in $U$ by $\lambda$, hence it is open as $\lambda$ is continuous
and $p^N\lambda(V_0)$ is open in $K$. We proved what we wanted, up to
replacing $V_0$ by its sub-lattice $p^NV_0$.

 The second claim follows immediately from the first one. 
\end{proof}
 
\begin{lemma}\label{ext10.1}
 If we have a map of strict exact sequences of  complete Hausdorff spaces from $C_K$ 
$$
\xymatrix@R=6mm{
0\ar[r]&  V\ar[r] \ar[d]^{f_V}&  W\ar[r] \ar[d]^{f_W}&  W'\ar[r] \ar[d]^{f_{W'}}&0\\
0\ar[r]&  V_1\ar[r] &  W_1\ar[r] &  W_1'\ar[r] &0
}
$$
such that  $V,V_1$ are of finite rank over $K$, the bottom sequence splits, and the map  $f_{W'}$ is compact then
 the map  $f_W$ is compact as well. 
\end{lemma}
\begin{proof}
Take an open lattice $L_{W'}$ in $W'$ such that $f_{W'}(L_{W'})$ is compactoid.  Let $L_W$ be the  preimage of $L_{W'}$ in $W$. 
 Choose a section $s: W_1'\to W_1$ of the projection $\pi: W_1\to W_1'$ and consider the continuous map 
 $$
 g: W\to V_1,\quad g(w):=f_W(w)-s\pi(f_W(w)).
 $$
 Take a compact lattice $L_{V_1}$ in $V_1$,  its preimage (via $g$) in $W$, and then change  $L_W$ to its intersection with that preimage. 
 
 Now, we see that  $f_W(L_W)$ is compactoid: we have 
 $$
 f_W(L_W)\subset g(L_W)+sf_{W'}(L_{W'})
 $$
 and both $g(L_W)$ and $sf_{W'}(L_{W'})$ are  compactoid in $W_1$. This concludes the proof of our lemma.
\end{proof}

  Let $C_{nF,K}$ denote  the full subcategory of $C_K$ consisting of nuclear Fr\'echet spaces. The functor
  $$
  C_{c,K}\to C_{nF,K},\quad V\mapsto V^{\prime}_b,
  $$
  is an anti-equivalence of categories. For any two $V,W\in C_{c,K}$ the natural linear map $\sll_b(V,W)\to \sll_b(W^{\prime}_b, V^{\prime}_b)$ is a topological isomorphism. 

\begin{lemma} 
If $X\in C_K$ is nuclear Fr\'echet or of compact type then the canonical map $X^{\prime}_b\to X^{\star}$ is a topological isomorphism.
\end{lemma}
\begin{proof}In both cases the space $X$  is reflexive,   hence Montel (see \cite[Cor. 8.4.22]{PS}). But in Montel spaces, by \cite[Th. 8.4.5]{PS}, 
every bounded subset is compactoid (and vice versa, of course) thus the strong
topology on the algebraic dual of $X$ coincides with the topology of compactoid convergence, as wanted.
\end{proof}
\subsection{Solid functional analysis} We will  review here briefly results from   solid functional analysis that we will need. Our main references are  \cite{Bosc}, \cite{Cam}, \cite{Sch19}, \cite{Sch20}.
\subsubsection{Basic properties of condensed sets} \label{colibre1}
Let  {\rm Cond} denote  the category of {\em condensed sets}, i.e., sheaves of sets on the site of pro-finite sets with coverings given
by finite families of jointly surjective maps\footnote{We refer the reader to \cite[Lecture III]{Sch19} for a discussion of set theoretical issues involved in this definition.} or, equivalently, on the pro-\'etale site $*_{\proeet}$ of a geometric point.  We define similarly condensed groups, rings, etc. 

  We will denote by ${\rm CondAb}$ the category of condensed abelian groups.  It  is an abelian category \cite[Thm. 2.2]{Sch19}. It has all limits and colimits. Arbitrary products, arbitrary direct sums and filtered colimits are exact. It is generated by compact projective objects.
  For a condensed commutative 
 ring $A$, we will write  ${\rm Mod}^{\rm cond}_A$ for  the category of $A$-modules in ${\rm CondAb} $ and 
$\underline{\Hom}_A(-,-) $ for its internal $\Hom$ (in the case $A = \Z$, we will often omit the
subscript $\Z$). 
\begin{remark} Because of set theoretical issues this definition of the category ${\rm Cond}$ is not sensu stricto correct. The correct  definition of ${\rm Cond}$  is given in \cite[Lecture II]{Sch19}. In this paper we will use the latter though we find it helpful to keep in mind  its simplified version given above. 
\end{remark}

 For a condensed set $X$, we think of $X(*)$ as the underlying set of $X$ and about $X(S)$ as the continuous maps from $S$ to $X$. A quasi-separated\footnote{A condensed set $X$ is called {\em quasi-compact} if there is a profinite set $S$ and a surjective map ${S}\to X$; it is called {\em quasi-separated} if, for any pair of profinite sets $S$ and $S^{\prime}$ over $X$ the fiber product ${S}\times_X{S}^{\prime}$ is quasi-compact.}
 condensed set $X$ is trivial as soon as $X(*)$ is trivial (a fact that is false for a general $X$, see
  \cite[Lecture I]{Sch20}). 

Let ${\rm Top}$ denote the category of  T1  topological spaces\footnote{A topological space is called T1 if all its  points are closed.}.
We have a functor $$(\underline{\phantom{x}}): {\rm Top} \to {\rm Cond},\quad  T \mapsto \underline{T}=\scc(S, T),$$
where $\scc(S, T)$ denotes  the
set of  continuous functions from $S$ to $T$. Condensed sets that come from topological spaces we will call {\em classical}. 
We quote: 

\begin{proposition}
 {\rm (Clausen-Scholze, \cite[Prop.  1.2]{Sch20})}. The functor $(\underline{\phantom{x}})$:
\begin{enumerate}
\item   has a left adjoint $X\mapsto  X(*)_{\rm top}$ sending any condensed set $X$ to the set $X(*) $ equipped
with the quotient topology arising from the map
$$\coprod\nolimits_{S,a\in X(S)}
S \overset{a}{\to}  X(*).$$
\item restricted  to compactly generated topological spaces  is fully faithful.
\item   induces an equivalence between the category of compact Hausdorff spaces and qcqs (i.e., quasi-compact quasi-separated) condensed sets.
\item  induces a fully faithful functor from the category of compactly generated weak Hausdorff\footnote{A topological space is called {\em weak Hausdorff}  if  the image of every continuous map from a compact Hausdorff space into the space is closed. In particular, every Hausdorff space is weak Hausdorff.  Every weak Hausdorff space is a T1 space.}
spaces, to quasi-separated condensed sets. The category of quasi-separated condensed sets is equivalent
to the category of ind-compact Hausdorff spaces ``$\colim_n$''$X_n$, where all transition maps  are
closed immersions. If $\{X_n\}_{n\in\N}$ is a ind-system of compact Hausdorff spaces with closed
immersions and $X = \colim_n X_n$  as topological spaces, then the canonical map
$$\colim_n \underline{X_n}\to  \underline{X}$$
is an isomorphism of condensed sets. In particular, $\colim_n \underline{X_n}$ is classical, i.e., it comes from a topological space. 
\end{enumerate}
\end{proposition}

\begin{remark}
(1) For $T\in {\rm Top}$, the counit $\underline{T}(*)_{\rm top}\to T$ of the adjunction agrees with the counit $T^{\rm cg}\to T$ of the adjunction between compactly generated spaces and all   topological spaces by \cite[Prop. 1.7]{Sch19}. In particular, $\underline{T}(*)_{\rm top}\simeq T^{\rm cg}$.

(2) For a finite extension $K$ of $\Q_p$, we will abbreviate the notation ${\rm Mod}^{\rm cond}_{\underline{K}}$ to ${\rm Mod}^{\rm cond}_K$; we will call the elements of this category {\em condensed $K$-vector spaces}. 
 \end{remark}
\subsubsection{Solid modules} \label{colibre2}We will briefly review basic facts concerning solid modules.

 ($\bullet$) {\em  Analytic rings.} We start with analytic rings.
 
 ({\em i}) {\rm (Clausen-Scholze, \cite[Th. 5.8]{Sch19})}  The analytic ring $\Z_{\Box}=(\Z,\sm_{\Z})$ is defined as  the ring $\Z$ equipped with the  functor of measures $\sm_{\Z}$ sending an extremally disconnected set $S=\lim_i S_i$, where each $S_i$ is a finite set, to the condensed abelian group 
$\Z_{\Box}:=\lim_i\Z[S_i]$. 

({\em ii}) {\rm (Clausen-Scholze, \cite[Prop. 7.9]{Sch19})} Let  $K$  be  a finite extension of $\Q_p$. There is an analytic structure on the condensed rings $\so_K$ and  $K$ given by sending an extremally disconnected set $S=\lim_i S_i$, 
where each $S_i$ is finite,  to $$\so_{K,\Box}[S]:=\lim_i\so_K[S_i],\quad K_{\Box}[S]:=K\otimes_{\so_K}\so_{K,\Box}[S].$$
The first analytic ring structure is induced from  the analytic ring structure of $\Z_{\Box}$ by   base change   to $\so_K$.

 ($\bullet$) {\em Solid modules.} Now we pass to solid modules.

\begin{proposition}{\rm (Clausen-Scholze \cite[Prop. 7.5]{Sch19})} \label{nyear1} Let $A=(A,\sm_A)$ be one of the  analytic rings above.
\begin{enumerate}
 \item  The full subcategory of solid $ A$-modules
\begin{equation}
 \label{form1}{\rm Mod}^{\rm solid}_{A}\subset  {\rm Mod}^{\rm cond}_A 
 \end{equation}
 consists of all $A$-modules $M$ such that, for all extremally disconnected sets $S$, the maps 
 $$
 \Hom_A(\sm_A[S],M)\to  \Hom_A(A[S],M)
 $$ are isomorphisms. It 
is an  abelian subcategory,   stable under all limits, colimits, and extensions.  The
inclusion \eqref{form1} admits a left adjoint
\begin{equation}
 \label{form11}
 {\rm Mod}^{\rm cond}_A\to  {\rm Mod}^{\rm solid}_{A} : \quad M \mapsto M \otimes_A (A,\sm_A), 
 \end{equation}
which   preserves all colimits and is  symmetric monoidal. 
\item  The functor
\begin{equation}
\label{form2}
\sd({\rm Mod}^{\rm solid}_{A})
\to  \sd({\rm Mod}^{\rm cond}_{A}) 
\end{equation}
is fully faithful. Its essential image is stable under all limits and colimits.  It is given by  complexes $M \in \sd({\rm Mod}^{\rm cond}_{A}) $ such that the map 
$$
\R \Hom_A(\sm_A[S],M)\to \R\Hom_A(A[S],M)
$$ is a quasi-isomorphism for all extremally disconnected sets $S$. 

A complex  $M \in \sd({\rm Mod}^{\rm cond}_A)$ is in $\sd({\rm Mod}^{\rm solid}_{A})$ if and only if $H^i(M)$
is in ${\rm Mod}^{\rm solid}_{A}$, for all $i$. The functor \eqref{form2}  admits a left adjoint
\begin{equation}
\label{form3}
\sd({\rm Mod}^{\rm cond}_A) \to \sd({\rm Mod}^{\rm solid}_{A}) : M\mapsto M\otimes^{\rm L}_A(A,\sm_A), 
\end{equation}
which  is the left derived functor of \eqref{form11}. It is symmetric
monoidal.
\item For  $M, N \in \sd({\rm Mod}^{\rm solid}_{A})$, we have  the derived internal Hom $$\R\underline{\Hom}_{A}(M,N)\in\sd({\rm Mod}^{\rm solid}_{A}).$$
The natural map
$\R\underline{\Hom}_{(A,\sm_A)}(M,N) \to  \R\underline{\Hom}_A(M,N)$
is a quasi-isomorphism.

\end{enumerate}
\end{proposition}

 \begin{notation}  (1) We write   ${\rm Solid} := {\rm Mod}^{\rm solid}_{(\Z,\sm_\Z) }$ for the category of solid
abelian groups;  we write
${\rm CondAb}\to  {\rm Solid} : M \mapsto  M^{\Box}$
for the functor \eqref{form} of Proposition \ref{nyear1}, and call it {\em solidification}, and we denote by 
$\otimes^{\Box}_\Z$ the unique
symmetric monoidal tensor product making the solidification functor symmetric monoidal.

 (2) For a finite extension $K$ of $\Q_p$, we write ${\rm Mod}^{\rm solid}_{\so_K}$ and ${\rm Mod}^{\rm solid}_K$ for the categories of solid $\so_K$-modules and $K$-vector spaces, respectively. We will denote 
 by $\sd(\so_{K,\Box})$ and $\sd(K_{\Box})$ the corresponding derived $\infty$-categories. 
 
 (3) For a finite extension $K$ of $\Q_p$ and a commutative solid $K$-algebra $A$, we write $\otimes^{\Box}_A$ for the symmetric monoidal tensor product $\otimes_{(A,\sm_{A})}$. 
\end{notation}
\subsubsection {Locally convex  and condensed vector spaces} \label{res21}
 Consider the functor
$${\rm CD}:=(\underline{\phantom{x}}): \quad C^{\rm Hcg}_K\to {\rm Mod}^{\rm cond}_{{K}}, \quad V\mapsto \underline{V}.$$ 
We will denote in the same way its extension to the category of complexes. By Lemma \ref{res2} below,  the functor ${\rm CD}$ preserves (strict) quasi-isomorphisms 
of  complexes of Fr\'echet spaces or spaces of compact type and if $V$ is such a complex then
$$
{\rm CD}(\wt{H}^i(V))\simeq H^i({\rm CD}(V)).
$$
\begin{lemma} \label{res2} 
 The functor $(\underline{\phantom{x}})$ maps strict exact sequences of Fr\'echet spaces over $K$ to exact sequences of condensed ${K}$-vector spaces. Similarly, for strict exact sequences of spaces of compact type over $K$. 
\end{lemma}
\begin{proof}  Since the functor $V\mapsto \underline{V}$ is left exact, it suffices to show that a strict surjection $V\to W$ of Fr\'echet spaces or of spaces of compact type is carried to a surjection $ \underline{V}\to  \underline{W}$. For that, it suffices  to show that, for an extremally disconnected set $S$, we 
have $\scc(S,V )\twoheadrightarrow \scc(S,W)$, i.e.,
given   $g\in \scc(S,W)$, there exists ${g}^{\prime}\in \scc(S, V )$ making the following diagram commute
$$
\xymatrix@R=6mm{& V\ar[d]^f\\
S \ar[r]^{g} \ar@{-->}[ru]^{{g}^{\prime}}& W
 }
$$

  Assume first that both $V$ and $W$ are Fr\'echet. In that case we have the following argument of Guido Bosco \cite[Lemma 1.A.33]{Bosc}:  Since  $g(S)$ is compact in $W$,  by \cite[Lemma 45.1]{Tre67},  it is the image $ f(H)$ of a compact subset $H$
of $V$.  We conclude by recalling that the extremally disconnected sets are the projective objects of
the category of compact Hausdorff topological spaces.

 Assume now that both $V$ and $W$ are of compact type. Since  $g(S)$ is compact in $W$, if we write $W=\colim_n W_n$, for an ind-system $\{W_n\}_{n\in\N}$ of Banach spaces over $K$ with compact injective transition maps, then $g(S)\subset W_m$, for some $m\in\N$ (see \cite[Lemma 16.9]{Schn}). By \cite[Lemma 3.39]{Cam}, we have a commutative diagram of solid arrows
 $$
 \xymatrix@R=6mm{
 V_m\ar@{->>}[r]^{f_m} \ar[d]& W_m\ar[d]\\
 V\ar@{->>}[r]^{f} & W & S\ar[l]^{g}\ar[lu]_{g_{W_m}}\ar@{-->}[llu]_-{{g}^{\prime}},
 }
 $$
 for a $K$- Banach space  $V_m$. As above, it follows that the dashed arrow exists, which concludes our proof of  the lemma.
 \end{proof}
  \subsubsection{Solid Fr\'echet spaces and solid spaces of compact type}\label{solid-frechet} We will define solid Fr\'echet spaces and solid spaces of compact type as the images in the category of solid $K$-vector spaces of their classical analogs. See \cite[Ch. 3]{Cam} for alternative definitions. 
  
  (i) {\em Solid Fr\'echet spaces.} A {\em solid Banach space}  is a solid $K$-vector space of the form $\underline{V}$, for a classical Banach space $V$.
  A {\em solid Fr\'echet  space}  is a solid $K$-vector space of the form  $\underline{V}$, for a classical Fr\'echet space $V$. It is called {\em nuclear} if $V$ is nuclear. 
The functor $(\underline{\quad})$ identifies the categories of classical and solid Fr\'echet spaces (identifying also exact sequences by Lemma \ref{res2}). If this does no confusion we will use the term "Fr\'echet spaces" for elements of both of these categories. 

{\bf Warning}. The definition of nuclear used here is stronger than the definition of nuclear in the solid formalism. To distinguish, we will call the latter "solid nuclear". 
  
   Let $V$ be a solid Fr\'echet space written  as a limit $V=\lim_n V_n$ of  Banach spaces $V_n$ with dense transition maps. Then: 
 \begin{enumerate}
\item (Topological Mittag-Leffler) Then
$$
\R^j\lim_n V_n=0,\quad j\geq 1.
$$
In particular, $V\simeq \R\underline{\Hom}_K(V^*,K)$. 
\item If  $W$ is  a Banach space, then $\underline{\Hom}_K(V,W)=\colim_n\underline{\Hom}_K(V_n,W)$.
\end{enumerate}

 (ii) {\em Solid spaces of compact type.} 
  A {\em solid space of compact type}  is a solid $K$-vector space of the form  $\underline{V}$, for a classical space of compact type $V$. 
The functor $(\underline{\quad})$ identifies the categories of classical and solid spaces of compact type (identifying also exact sequences by Lemma \ref{res2}). If this does no confusion we will use the term "spaces of compact type"  for elements of both of these categories. 

The following two lemmas will be needed later. 
  \begin{lemma}\label{mine1}
  Let $V\in C_K$ be of compact type. Write it as  $V\simeq \colim_n V_n$ with $V_n$, $n\in\N$,  Hausdorff and  injective compact transition maps. Then
  we have a canonical isomorphism of condensed $K$-modules
  $$
  \colim_n \underline{V_n}\stackrel{\sim}{\to} \underline{V}.
  $$
  In particular, the condensed $K$-module $  \colim_n \underline{V_n}$ is classical.
  \end{lemma}
  \begin{proof}
  Since compact sets are bounded, this follows immediately from \cite[Lemma 16.9]{Schn}, which states that any bounded subset of $V$ comes from some $V_n$. 
  \end{proof}
  \begin{lemma}
If a sequence of solid Fr\'echet spaces over $K$ is exact then the classical sequence is strictly exact. Similarly, for a  sequence of solid  spaces of compact type over $K$. 
 \end{lemma}
  \begin{proof}
 We start with an exact sequence of solid Fr\'echet spaces over $K$ 
 $$
 0\to \underline{V}_1\to \underline{V}_2\to \underline{V}_3\to 0.
 $$
 We need to show that the sequence of classical Fr\'echet spaces
  \begin{equation}
  \label{fr1}
  0\to {V}_1\stackrel{f_1}{\to} {V}_2\to V_3\to 0
  \end{equation}
  is strictly exact.  
 But the   sequence
 \begin{equation}
 \label{fr2}
  0\to \underline{V}_1(*)_{\rm top}\to \underline{V}_2(*)_{\rm top}\to \underline{V}_3(*)_{\rm top}\to 0
 \end{equation}
 maps, via topological isomorphisms $\underline{V}_i(*)_{\rm top}\stackrel{\sim}{\to} V_i$ (since $V_i$ is compactly generated), to the sequence \eqref{fr1}.  Since $\underline{V_3}$ is quasi-separated, the map $f_1$ is a closed immersion. In particular we have a strict exact sequence of Fr\'echet spaces
 $$
 0\to {V}_1\stackrel{f_1}{\to} {V}_2\to V'_3\to 0
 $$
 and a continuous injection $f_2: V'_3\to V_3$. We need to prove that the map $f_2$ is a topological isomorphism. Since $f_2$  is a map between two Fr\'echet spaces, by the Open Mapping Theorem, it suffices to show that $f_2$ is an algebraic isomorphism. But, by Lemma \ref{res2}, we have the exact sequence
 $$
 0\to \underline{V}_1\to \underline{V}_2\to \underline{V}'_3\to 0.
$$
Hence the canonical map $\underline{f}_2:  \underline{V}'_3 \to \underline{V}_3$ is an isomorphism, which yields that so is the map $f_3$ (by the faithfullness of the $(\underline{\phantom{x}})$ functor. 
 
  The argument for spaces of compact type is similar (the Open Mapping Theorem is valid for LB spaces). 
  \end{proof}

  \subsubsection{Solid tensor product}\label{solid-product}
  We list properties of the solid tensor product that we will often use. 
  \begin{enumerate}
    \item (\cite[Prop. A.68]{Bosc}) Let $V, W$ be Fr\'echet spaces over $K$. Then we have a  natural isomorphism  of solid  $K$-vector spaces
  $$
  \underline{V}\otimes_{K}^{\Box}\underline{W}\stackrel{\sim}{\to}\underline{V\wotimes_KW},
  $$
  where $V\wotimes_KW$ denotes the projective tensor product in the category $C_K$.   
    \item (\cite[Cor. A.65]{Bosc}) Any Fr\'echet space  over $K$  is acyclic for the tensor product $\otimes^{\Box}_{K}$. That is, if $V$ is a Fr\'echet space  over $K$ then $(-)\otimes^{{\rm L}\Box}_KV\simeq (-)\otimes^{\Box}_KV$. 
  \item (\cite[Cor. A.67]{Bosc})\begin{enumerate}
  \item Let $\{V_n\}_{n\in\N}$ be a  pro-system of solid nuclear $K$-vector spaces and let $W$ be a Fr\'echet vector space over $K$. Then we have an isomorphism
  $$
   \lim_n (V_n\otimes_{K}^{\Box}W)\stackrel{\sim}{\leftarrow}(\lim_n V_n)\otimes_{K}^{\Box}W.
  $$
  \item Let $\{V_n\}_{n\in\N}$ be a  pro-system  in $\sd(K_{\Box})$ of complexes of solid nuclear $K$-vector spaces. Let $W$ be a  complex of $K$-Fr\'echet spaces. Then we have a quasi-isomorphism
  $$
    \R \lim_n (V_n\otimes_{K}^{{\rm L}\Box}W)\stackrel{\sim}{\leftarrow}(\R\lim_n V_n)\otimes_{K}^{{\rm L}\Box}W.
  $$
  \end{enumerate}
  \end{enumerate}
 The above  properties  also hold  if we replace $K$ with $\breve{F}$ (with the same references). 
\section{Galois cohomology of $K$} This chapter gathers together a number of properties of Galois cohomology that we will need later. 
\subsection{Preliminaries} We record here few basic facts about Galois cohomology seen via condensed formalism. 
\subsubsection{Condensed Galois cohomology} Let  $G$ be a condensed group. A condense $G$-module is a condensed abelian group endowed with a $\Z[G]$-module structure. The condensed group cohomology of $G$ with values in a condense $G$-module $V$ is defined as 

$$
\R\Gamma(G,V):=\R\underline{\Hom}_{\Z[G]}(\Z,V)\in\sd({\rm CondAb}).
$$
\begin{notation}
 Let $G$ be a profinite group. 
 
 (a) The {\em Iwasawa} algebra of $G$ is  the solid ring
 \begin{align*}
 \so_{K,\Box}[G] & :=\lim_{H\subset G}\so_K[G/H]\in {\rm Mod}^{\rm solid}_{\so_K},\\
 K_{\Box}[G]= \so_{K,\Box}[G][1/p] & :=\big (\lim_{H\subset G}\so_K[G/H]\big )[1/p]\in {\rm Mod}^{\rm solid}_{K},
 \end{align*}
 where $H$ runs over all open and normal subgroups of $G$.
 
 (b) A solid $G$-module over $\so_K$ (or a solid $\so_{K,\Box}[G]$-module) is a solid abelian group  endowed with an $\so_{K,\Box}[G]$-module structure. The category of solid 
 $\so_{K,\Box}[G]$-modules will be denoted by ${\rm Mod}^{\rm solid}_{\so_{K,\Box}[G]}$ and its derived $\infty$-category by $\sd(\so_{K,\Box}[G])$.
 Similarly, for solid  $K_{\Box}[G]$-modules.
\end{notation}
We list the following properties of $\R\Gamma(G,-)$:
\begin{enumerate}
\item (\cite[Prop. B.2]{Bosc})  Let $G$ be a profinite group and let $V$ be a ${G}$-module in solid abelian groups. Then
\begin{enumerate}
\item The complex $\R\Gamma({G},V)$ is quasi-isomorphic to the complex of solid abelian groups
\begin{equation}
\label{cond11}
V\to \underline{\Hom}(\Z[{G}^1],V)\to \underline{\Hom}(\Z[{G}^2],V)\to  \underline{\Hom}(\Z[{G}^3],V)\cdots.
\end{equation}
\item  If $V=\underline{V_{\rm top}}$, with $V_{\rm top}$ a T1 topological $G$-module over $\Z$, then, for all $i\geq 0$, we have a natural isomorphism of abelian groups\footnote{The second cohomology group is the continuous group cohomology.}
$$
\R\Gamma({G},V)(*)\simeq \R\Gamma({G},V_{\rm top}).
$$
\end{enumerate}
\item (\cite[Prop. B.3]{Bosc}) For $n\in\N$, let $\Gamma:=\Z_p^n$, and let $\gamma_1,\ldots,\gamma_n$ denote the generators of $\Gamma$. Let $V$ be a $\Gamma$-module in ${\rm Mod}^{\rm solid}_{\Z_p}$. Then we have a quasi-isomorphism
$$
\R\Gamma(\Gamma,V)\simeq  {\rm Kos}_{\gamma}(V):={\rm Kos}_V(\gamma_1-1,\ldots,\gamma_n-1),
$$
the Koszul complex of $V$ with respect to the elements $\gamma_1-1,\ldots,\gamma_n-1$. 
\item (\cite[Lemma 5.2]{Cam}) Let $G$ be a profinite group. There is a solid projective resolution of the trivial representation 
$$
\cdots \to K_{\Box}[{G}^{n+1}]\to K_{\Box}[{G}^{n}]\to \cdots\to K_{\Box}[{G}^{1}]\to K\to 0.
$$
In particular,  if  $V$  is  a ${G}$-module in solid modules  over $K$, one has that 
$$
\R \underline{\Hom}_{K_{\Box}[G]}(K,V)\simeq \R\Gamma({G},V).
$$
\end{enumerate}

\begin{lemma} \label{niedziela1}
Let $G$ be a profinite group and let $V$ be a finite rank $\Q_p$-vector space equipped with a continuous action of $G$. Then 
\begin{enumerate}
\item we have  
  a quasi-isomorphism and isomorphisms
$$
 {\rm CD}(\R\Gamma({G},V))\simeq  \R\Gamma({G},\underline{V}),\quad   {\rm CD}(\wt{H}^i({G},V))\simeq H^i({G},\underline{V}),\, i\geq 0.
$$
\item we have a  quasi-isomorphism\footnote{Topology on $\R\Gamma({G},V)$ is defined using continuous cochains. See the proof of the lemma for details.} in $\sd(C_{\Q_p})$
$$
 \R\Gamma({G},\underline{V})(*)_{\rm top}\simeq  \R\Gamma({G},V).
$$
\end{enumerate}
\end{lemma}
\begin{proof} For claim (1) we compute
\begin{align*}
 & \R\Gamma({G},V)\simeq C(G,V),\quad n\mapsto \scc(G^{n-1},V);\\
&  {\rm CD}(\R\Gamma({G},V))(S): n\mapsto \scc(S,\scc(G^{n-1},V))\simeq  \scc(S\times G^{n-1},V),
\end{align*}
where $S$ is a profinite set and $\scc(-,-)$ denotes the space of continuous maps equipped with compact open topology (note that $ \scc(S\times G^{n-1},V)$, since $S\times G^{n-1}$ is compact,  is a $\Q_p$-Banach space).
We also have 
\begin{align*}
& \R\Gamma({G},\underline{V}): n\mapsto \underline{\Hom}(\Z[{G}^{n-1}],\underline{V});\\
&  \underline{\Hom}(\Z[{G}^{n-1}],\underline{V})(S)={\Hom}(\Z[{S}]\otimes \Z[{G}^{n-1}],\underline{V})\simeq {\Hom}(\Z[{S\times G^{n-1}}],\underline{V}). 
\end{align*}
Since ${\Hom}(\Z[{S\times G^{n-1}}],\underline{V})\simeq \scc(S\times G^{n-1},{V})$, we get claim  (1) of the lemma. 

Claim (2) follows from  the fact that the complex of continuous cochains representing $ \R\Gamma({G},V)$ is a complex of Banach spaces and \cite[Prop. 3.5]{Cam}.
\end{proof}
\subsubsection{Poitou-Tate duality} Let $K$ be a finite extension of $\Q_p$ and let $V$ be a continuous, finite rank $\Q_p$-representation of $\sg_K$.
 Recall that  the Galois pairing 
$$H^i(\sg_K,V)\otimes^{\Box}_{\Q_p} H^{2-i}(\sg_K,V^*(1))\stackrel{\cup}{\to} H^2(\sg_K,\Q_p(1))\xrightarrow[\sim]{{\rm Tr}_{K}}\Q_p$$
is  a perfect pairing (by Poitou-Tate duality). Hence, for $i\in\N$, 
$H^i(\sg_K,V)$ and $H^{2-i}(\sg_K,V^*(1))$ 
are natural duals (via the above pairing). In particular, we have 
\begin{align*}
H^0(\sg_K,\Q_p(j)) & \simeq \begin{cases} \Q_p&{\text{if  $j=0$,}}\\
0 &{\text{otherwise;}}\end{cases}
\quad
H^2(\sg_K,\Q_p(j))\simeq \begin{cases} \Q_p&{\text{if  $j=1$,}}\\
0 &{\text{otherwise.}}\end{cases}
\end{align*}

\subsection{$(\varphi,\Gamma)$-modules and Galois cohomology}\label{prelim} In the next two sections, 
 we will briefly recall and  refine the relationship between $(\varphi,\Gamma)$-modules and Galois cohomology. 
\subsubsection{Notations}\label{tau1}
If $n\geq 1$, let $F_n=\Q_p(\bmu_{p^n})$ and
let $F_\infty:=\cup_nF_n$ be the cyclotomic extension of $\Q_p$.
Let $\chi:\G_{\Q_p}\to\Z_p^\dual$ be the cyclotomic character. 
Then $\chi$ factors through $\Gamma:={\rm Gal}(F_\infty/\Q_p)$ and induces
an isomorphism $\chi:\Gamma\overset{\sim}{\to}\Z_p^\dual$.

If $\Delta$ is the torsion subgroup of $\Gamma$,
then $\chi^{|\Delta|}$ takes
values in $1+p\Z_p$ (resp.~$1+8\Z_2$) if $p\neq 2$ (resp.~$p=2$).
Let $\tau:\G_{\Q_p}\to\Z_p$ be defined by
$$\tau=\begin{cases}\tfrac{1}{p\,|\Delta|}\log \chi^{|\Delta|}&{\text{if $p\neq 2$,}}\\
\tfrac{1}{4\,|\Delta|}\log \chi^{|\Delta|}&{\text{if $p= 2$.}}\end{cases}
\quad{\text{i.e., }}\tau=\tfrac{1}{p^{c(p)}}\log\chi,\ {\rm with}\  
{c(p)}=\begin{cases} 1 &{\text{if $p\neq 2$,}}\\ 2 &{\text{if $p= 2$.}}\end{cases}$$
Let $F_\infty':=F_\infty^\Delta$ be the cylotomic $\Z_p$-extension of $\Q_p$.
Then $\tau$ factors through $\Gamma':={\rm Gal}(F_\infty'/\Q_p)$ and induces
an isomorphism $\tau:\Gamma'\overset{\sim}{\to}\Z_p$.

Let $K$ be a finite extension of $\Q_p$. If $n\in\N$, let $K_n=K(\bmu_{p^n})$ and
let $K_\infty:=\cup_nK_n$ be the cyclotomic extension of $K$.
Let $\Gamma_K:={\rm Gal}(K_\infty/K)$; then $\chi$ induces an isomorphism
from $\Gamma_K$ to an open subgroup of $\Z_p^\dual$. Let $\Delta_K$ be
the torsion subgroup of $\Gamma_K$, let $K'_\infty:=K_\infty^{\Delta_K}$
be the cylotomic $\Z_p$-extension of $K$. Then $\Gamma_K':={\rm Gal}(K_\infty'/K)=
\Gamma_K/\Delta_K$ and $\tau$ induces
an isomorphism $\tau:\Gamma_K'\overset{\sim}{\to}p^{n(K)}\Z_p$ for some $n(K) \in\N$.

Let $\gamma_K\in\Gamma'_K$ be the element verifying $\tau(\gamma_K)=p^{n(K)}$.
Then $\gamma_K$ has a unique lifting in
$\Gamma_K$ whose image by $\chi$ belongs to $1+p^{n(K)+c(p)}\Z_p$; we denote this lifting
also by $\gamma_K$. Then $\Gamma_K=\gamma_K^{\Z_p}\times\Delta_K$.

Let $L=K(\bmu_{p^{c(p)}})$. Then ${\rm Gal}(L/K)=\Delta_K$ and $\Gamma_L=\gamma_K^{\Z_p}$.
\begin{remark}\label{tau2}

(i)
Let $F:=K\cap F_\infty$.  Then $[K_\infty:F_\infty]=[K:F]$
and $\Gamma_F=\Gamma_K=\Delta_K\times p^{n(K)}\Z_p$;
hence $[F:\Q_p]=p^{n(K)}\frac{|\Delta|}{|\Delta_K|}$ and
$[K:\Q_p]=[K_\infty:F_\infty]\cdot p^{n(K)}\frac{|\Delta|}{|\Delta_K|}$.

(ii) We have $\tau(\gamma_K)=p^{n(K)}$, 
hence $\log\chi(\gamma_K)=p^{n(K)+c(p)}$.
\end{remark}

For  $0<u\leq v\in v(K_\infty^\flat)$, let
$$
\B_{K_\infty}^{[u,v]}
:=(W(\O_{K_\infty}^\flat)[\tfrac{p}{[\alpha]},\tfrac{[\beta]}{p}]^{\wedge_p})[\tfrac{1}{p}],
\quad v(\alpha)=\tfrac{1}{v},\ v(\beta)=\tfrac{1}{u}.
$$
It is a Banach space over $\Q_p$. Let $\varphi$ denote the Frobenius morphism acting on $\B_{K_\infty}^{[u,v]}$ and $\psi$ its left inverse. Let  $U^{[u,v]}$ be the corresponding open set of the  Fargues-Fontaine curve over
$K_\infty$.
Take  $u=\frac{p-1}{p}$, $v=p-1$ if  $p\neq 2$, and  $u=\frac{2}{3}$, $v=\frac{4}{3}$ if $p=2$.
Then $\B^{[u,v]}_{K_\infty}/t=\widehat K_\infty$ and  $t$ is a unit 
in \footnote{ 
That is, the intersection of  $U^{[u,v]}$ with the orbit under
$\varphi$ of the point  $\infty$ of the 
 Fargues-Fontaine curve is  reduced to $\{\infty\}$ and  $U^{[u,v/p]}$ does not intersect this orbit.}
$\B^{[u,v/p]}_{K_\infty}$. We write $\theta$ the canonical map $\B_{K_\infty}^{[u,v]} \to \widehat K_\infty$ with ${\rm Ker}(\theta)=t\B^{[u,v]}_{K_\infty}$.

\subsubsection{Galois cohomology}
 For $j\in\Z$, 
the theory of  $(\varphi,\Gamma)$-modules yields quasi-isomorphisms
\begin{align}
\label{Galois1}
&\alpha_j: 
{\rm Kos}_{\varphi,\gamma}(\B^{[u,v]}_{K_\infty}(j))\simeq  \rg(\sg_L,\Q_p(j)),\\
&\alpha_j:
{\rm Kos}_{\varphi,\gamma}(\B^{[u,v]}_{K_\infty}(j))^{\Delta_K}\simeq  \rg(\sg_K,\Q_p(j)),
\notag
\end{align}
where the  $(\varphi,\gamma)$-Koszul complexes (of Banach spaces over $\Q_p$) are defined by:
\begin{align*}
{\rm Kos}_{\varphi,\gamma}(\B^{[u,v]}_{K_\infty}(j))&:= \xymatrix@R=6mm@C=15mm{
\big((\B^{[u,v]}_{K_\infty}(j))\ar[r]^-{(\varphi-1,\gamma_K-1)}
&(\B^{[u,v/p]}_{K_\infty}(j))\oplus (\B^{[u,v]}_{K_\infty}(j))\ar[r]^-{-(\gamma_K-1)+(\varphi-1)}
&(\B^{[u,v/p]}_{K_\infty}(j))\big)}
\end{align*}
and the complex 
${\rm Kos}_{\varphi,\gamma}(\B^{[u,v]}_{K_\infty}(j))^{\Delta_K}$
is obtained by taking fixed points under $\Delta_K$ of each of the terms of the complex.

Set $[\Delta_K]:=\sum_{\sigma\in\Delta_K}\sigma$.
The following commutative diagram allows to deduce results for $K$ from results for $L$,
i.e.~we can often assume that $\Delta_K=1$ in the proofs,
$$
\xymatrix@C=15mm@R=6mm{
{\rm Kos}_{\varphi,\gamma}(\B^{[u,v]}_{K_\infty}(j))
\ar[r]^{\sim}\ar@<3mm>[d]^{[\Delta_K]}
&\rg(\sg_L,\Q_p(j))\ar[d]^{{\rm cor}_L^K}\\
{\rm Kos}_{\varphi,\gamma}(\B^{[u,v]}_{K_\infty}(j))^{\Delta_K}\ar[r]^{\sim}\ar@<3mm>[u]^{\rm id}
&\rg(\sg_K,\Q_p(j))\ar@<5mm>[u]^{{\rm res}_K^L}}
$$

\begin{remark}
If $n$ is big enough so that $K_n$ has enough roots of unity in the sense of~\cite{CN1},
there exist normalized trace maps ${\rm Res}_{p^{-n}\Z_p}:\B^{[u,v]}_{K_\infty}\to
\B_{K_n}^{[u,v]}$ where $\B_{K_n}^{[u,v]}:=\varphi^{-n}(\B_K^{[u/p^n,v/p^n]})$
which commute with $\Gamma_K$ and verify $\varphi\circ {\rm Res}_{p^{-n-1}\Z_p}=
{\rm Res}_{p^{-n}\Z_p}\circ\varphi$.
These decompletion maps play a big role in the proofs because
$\B_{K_n}^{[u,v]}$ is a much nicer ring than $\B^{[u,v]}_{K_\infty}$ (it is 
a ring of analytic functions on an annulus, whereas the later is a ring of analytic functions 
on a perfectoid annulus).

Applying ${\rm Res}_{p^{-n}\Z_p}$ to each of the terms of the above complex produces
a quasi-isomorphic complex with rings related to $K_n$ instead of $K_\infty$,
which is closer to the complexes used in the theory of $(\varphi,\Gamma)$-modules,
like in Herr's thesis~\cite{Herr} or in~\cite{CCjams}.
Going from these rings to the usual rings uses
the standard techniques of  $(\varphi,\Gamma)$-modules, i.e., 
normalized trace ${\rm Res}_{p^{-n}\Z_p}$,
bijectivity of $\gamma_K-1$ on $(\B^{[u/p^n,v/p^n]}_{K}(j))^{\psi=0}$, etc., as in~\cite{CN1}.
\end{remark}

\subsubsection{Examples}\label{fanto1001}
Denote  by 
\begin{equation}
\label{explicit}
h^i_K:Z^i({\rm Kos}_{\varphi,\gamma}(\B^{[u,v]}_{K_\infty}(j))^{\Delta_K})\to H^i(\sg_K,\Q_p(j))
\end{equation}
 the map
induced by  quasi-isomorphism \eqref{Galois1}; it factorizes as
$$h^i_K:Z^i({\rm Kos}_{\varphi,\gamma}(\B^{[u,v]}_{K_\infty}(j))^{\Delta_K})\to H^i({\rm Kos}_{\varphi,\gamma}(\B^{[u,v]}_{K_\infty}(j))^{\Delta_K})
\simeq  H^i(\sg_K,\Q_p(j)).$$
If  $(a,b)\in Z^1({\rm Kos}_{\varphi,\gamma}(\B_{K_\infty}^{[u,v]}(j))^{\Delta_K})$
and   $u\in\B_{\overline K}^{[u,v]}(j)$
satisfy  $(\varphi-1)u=a$, then
\begin{equation}\label{cocy1}
h^1_K(a,b)={\rm cl}\big(\sigma\mapsto\tfrac{\sigma-1}{\gamma_K-1}b-(\sigma-1)u\big).
\end{equation}
(In this  expression, $\sigma$ acts  through its image in  $\Gamma'_K$ on $b$, 
and we think of $\tfrac{\sigma-1}{\gamma_K-1}$ as an element of $\Z_p[[\Gamma'_K]]$; 
the expression between parenthesis
is a $1$-cocycle on  $\sg_K$ with values in $\Q_p(j)$, 
and  ${\rm cl}$ denotes its image in $H^1(\sg_K,\Q_p(j))$.)

\vskip1mm
$\bullet$ {\it The case $j=0$}.
 We can apply the formula~(\ref{cocy1}) to  $j=0$ and cocycle $(1,0)$. Then
the corresponding $\sg_K$-cocycle factors through
${\rm Gal}(\overline{\bf F}_p/k_K)$ since the solution
of $(\varphi-1)u=1$ belongs to $W(\overline{\bf F}_p)$,
and sends the relative Frobenius to  $f:=f(K/\Q_p)$
(because $(\sigma-1)u$ is equal to $(\varphi^f-1)u=f$ since $\varphi(u)=u+1$).
Since the relative Frobenius corresponds to a uniformizer of  $K^\dual$, it follows that 
the element 
$\lambda$ of  ${\rm Hom}(K^\dual,\Q_p)$ which corresponds to $h^1_K(1,0)$
is $f(K/\Q_p) v_K=v_p\circ{\rm N}_{K/\Q_p}$ 
(where  $v_K$ is the valuation on $K$ with image $\Z$).

Starting with $(0,p^{n(K)+c(p)})$, formula (\ref{cocy1}) with $u=0$ yields a group
morphism $\G_K\to\Z_p$ which factors through $\Gamma'_K$ and has value $\log\chi(\gamma_K)$
at $\gamma_K$; it follows that this morphism is $\log\chi$.

\vskip1mm
$\bullet$ {\it The case $j=1$}.
Let $n=n(K)+c(p)$.
We have $(\varphi-1)\tfrac{1}{\pi}\in (\B_{\Q_p}^{(0,v]})^{\psi=0}$ 
and there exists  $a\in
(\B_{\Q_p}^{(0,v/p^n]})^{\psi=0}\otimes\tfrac{d\pi}{1+\pi}$ such that 
$(\gamma_K-1)a=((\varphi-1)\tfrac{1}{\pi})\otimes\tfrac{d\pi}{1+\pi}$.
Then, 
$$(\varphi^{-n}(a), \varphi^{-n}(\tfrac{1}{\pi}\otimes\tfrac{d\pi}{1+\pi}))\in Z^1(
{\rm Kos}_{\varphi,\gamma}(\B_{F_n}^{[u,v]}(1)))
\quad{\rm and}\quad
h^1_{F_n}(\varphi^{-n}(a), \varphi^{-n}(\tfrac{1}{\pi}\otimes\tfrac{d\pi}{1+\pi}))
={\rm cl}(\zeta_{p^n}-1).$$
(This follows from point iii) of~\cite[Prop.\,V.3.2]{CCjams}, using the constructions
leading to~\cite[th.\,II.1.3]{CCjams} and the (obvious) fact that Coleman power series
attached to $(\zeta_{p^n}-1)_{n\geq 1}$ is just $T$.)
\subsubsection{Cup-products}
 We define compatible cup products ($\alpha\in\Q_p$)
 \begin{align*}
 \cup_{\alpha} :  \quad {\rm Kos}_{\varphi,\gamma}(\B^{[u,v]}_{K_\infty}(j_1))^{\Delta_K}
 & \otimes^{\Box}_{\Q_p}  {\rm Kos}_{\varphi,\gamma}(\B^{[u,v]}_{K_\infty}(j_2))^{\Delta_K}
\to {\rm Kos}_{\varphi,\gamma}(\B^{[u,v]}_{K_\infty}(j_1+j_2))^{\Delta_K}
 \end{align*}
 using  (twice) the formulas from Section \ref{prod}. 
\begin{lemma}\label{compatibility-Galois-product}
On the level of cohomology, the quasi-isomorphism \eqref{Galois1} is compatible with products.
\end{lemma}
\begin{proof} This is easy to check for the pairing of $0$- and $2$-cocycles. 
For the pairing of two  $1$-cocycles, 
this amounts to checking that
\begin{equation}\label{cocy2}
h^1_K(a,b)\cup h^1_K(a',b')=h^2_K(b\otimes\gamma_K(a')-a\otimes\varphi(b')).
\end{equation}
(For the product on the Koszul complexes, we
 used here (twice) the formulas from Section \ref{prod} with $\alpha=1$.)
But  equality \eqref{cocy2} is standard and was checked by Herr and Benois 
in~\cite{Herr,Benois}.
\end{proof}
\vskip 2mm

\subsection{Residues and duality}
\label{fanto22}
\subsubsection{The trace map}
We will describe the trace isomorphism
$${\rm Tr}_K:H^2(\sg_K,\Q_p(1))\overset{\sim}{\to}\Q_p$$
in terms of  $(\varphi,\Gamma)$-modules (cf.~\cite{Herr}, \cite{Benois}).
For this, we start with the description of ${\rm Tr}_K$ given by local class field theory.
We have isomorphisms
$${\rm cl}:{\rm Hom}(K^\dual,\Q_p)\overset{\sim}{\to}H^1(\sg_K,\Q_p)
\quad{\rm and}\quad
\Q_p\wotimes K^\dual\overset{\sim}{\to}H^1(\sg_K,\Q_p(1))$$
and  ${\rm Tr}_K$ is given by the formula
$${\rm Tr}_K({\rm cl}(\tau)\cup {\rm cl}(\alpha))=\tau(\alpha),\quad{\text{ if
$\tau\in {\rm Hom}(K^\dual,\Q_p)$ and  $\alpha\in K^\dual$.}}$$

\subsubsection{Trace map and residues}
Let $\pi:=[\epsilon]-1\in\B_{F_\infty}^{[u,v]}$. We have $\gamma(\pi)=(1+\pi)^{\chi(\gamma)}-1$
if $\gamma\in\Gamma$, and we make $\varphi$ and $\Gamma$ act on
$\frac{d\pi}{1+\pi}$ by $\varphi(\frac{d\pi}{1+\pi})=\frac{d\pi}{1+\pi}$
and $\gamma(\frac{d\pi}{1+\pi})=\chi(\gamma)\frac{d\pi}{1+\pi}$, respectively (this allows us to identify
$\Lambda(1)$ with $\Lambda\otimes\tfrac{d\pi}{1+\pi}$).
Then there exists a unique continuous linear map
$${\rm res}_\pi:\B_{F_\infty}^{[u,v]}\otimes \tfrac{d\pi}{1+\pi}\to\Q_p$$
which commutes with $\varphi$ and $\Gamma$, and sends $\sum_{k\in\Z}a_k\pi^k\,d\pi$ to $a_{-1}$.
(See~\cite[Prop.\,IV.3.3]{mira} for the existence of ${\rm res}_\pi$.)

Set  
$${\rm Tr}_{K_\infty/F_\infty}:=\sum\nolimits_{\sigma\in G_{F_\infty}/G_{K_\infty}} \sigma$$
(well defined on modules on which $\G_{K_\infty}$ acts trivially).

\begin{proposition}\label{fanto23}
We have 
\begin{align*}
{\rm Tr}_K\circ h^2_K=\tfrac{1}{|\Delta_K|}\,
{\rm res}_\pi\circ{\rm Tr}_{K_\infty/F_\infty}\quad {\rm on}\ 
(\B_{K_\infty}^{[u,v/p]}\otimes\tfrac{d\pi}{1+\pi})^{\Delta_K}
\end{align*}
\end{proposition}
\begin{proof}
If $\alpha={\rm res}_\pi\circ{\rm Tr}_{K_\infty/F_\infty}$, then
$\alpha(\varphi(x))=\alpha(x)$ and $\alpha(\sigma(x))=\alpha(x)$, 
for all $\sigma\in \sg_{\Q_p}$.
 In particular, $\alpha$ factors through 
$H^2({\rm Kos}_{\varphi,\gamma}(\B_{K_\infty}^{[u,v]}(1)))$, which is
of dimension~$1$ over~$\Q_p$.
This proves the result up to multiplication by a constant 
and to show that this constant is 
 $1$, it suffices to show that the two terms coincide on an element on which one of the two is nonzero. 

Moreover, 
we have a commutative diagram
$$\xymatrix@R=5mm@C=15mm{
\B_{K_\infty}^{[u,v/p]}\ar[r]^-{h^2_L}\ar[d]^-{[\Delta_K]}
& H^2(\G_L,\Q_p(1))\ar[d]^-{{\rm cor}_{L/K}}\ar[r]^-{{\rm Tr}_L}&\Q_p\ar@{=}[d]\\
(\B_{K_\infty}^{[u,v/p]})^{\Delta_K}\ar[r]^-{h^2_K}
& H^2(\G_K,\Q_p(1))\ar[r]^-{{\rm Tr}_K}&\Q_p}$$
Hence ${\rm Tr}_L\circ h^2_L=|\Delta_K|\,{\rm Tr}_K\circ h^2_K$ on 
$(\B_{K_\infty}^{[u,v/p]})^{\Delta_K}$.
It follows that the result holds for $K$ if and only if it holds for $L$,
and we can assume that $\Delta_K=1$, i.e. $K=K_n$, with
$n=n(K)+c(p)$.

Let $\lambda=v_p\circ{\rm N}_{K/\Q_p}$ as in the case $j=0$ in section~\ref{fanto1001}.
It follows, using the identity ${\rm N}_{F_n/\Q_p}(\zeta_{p^n}-1)=p$
(resp.~$=-2$ if $p=2$), that
\begin{align*}
{\rm cl}(\lambda)\cup{\rm cl}(\zeta_{p^n}-1)&=
v_p({\rm N}_{K/\Q_p}(\zeta_{p^n}-1))\\
&= [K_n:F_n]\,v_p({\rm N}_{F_n/\Q_p}(\zeta_{p^n}-1))=[K_n:F_n]=[K_\infty:F_\infty]
\end{align*}
But formula~(\ref{cocy2}) gives us
$${\rm cl}(\lambda)\cup{\rm cl}(\zeta_{p^n}-1)=
{\rm Tr}_K\circ h^2_K\big(\varphi^{1-n}\big(
\tfrac{1}{\pi}\otimes\tfrac{d\pi}{1+\pi}\big)\big)$$
Since ${\rm Tr}_{K_\infty/F_\infty}$ and $x\mapsto {\rm res}_\pi x\frac{d\pi}{1+\pi}$
commute
with $\varphi$ and $\Gamma_K$, we have 
$${\rm res}_\pi\circ{\rm Tr}_{K_\infty/F_\infty}\big(
\varphi^{1-n}\big(\tfrac{1}{\pi}\otimes\tfrac{d\pi}{1+\pi}\big)\big)=
[K_\infty:F_\infty]
\,{\rm res}_\pi(\tfrac{1}{\pi}\otimes\tfrac{d\pi}{1+\pi}\big)=[K_\infty:F_\infty]$$
Our proposition follows.
\end{proof}

  \subsubsection{An alternative description of the trace map}\label{fanto24} 
 Let 
$${\rm Tr}:\Qbar_p\to\Q_p$$ be the unique
 $\G_{\Q_p}$-equivariant projection
(if  $[L:\Q_p]<\infty$, the restriction of ${\rm Tr}$ to  $L$ coincides with
$\frac{1}{[L:\Q_p]}{\rm Tr}_{L/\Q_p}$).  Then ${\rm Tr}$ extends by continuity
to $\widehat K_\infty$
(normalized Tate trace; it does not extend to $\C_p$),
and ${\rm Tr}\circ\theta$ gives a well defined map $\B_{K_\infty}^{[u,v]}\to\Q_p$.

\begin{proposition}\label{fanto5}
If $\alpha\in \B^{[u,v]}_{K_\infty}$,
$${\rm Tr}_K\circ h^2_K((\varphi-1)\tfrac{\alpha}{t}\otimes\tfrac{d\pi}{1+\pi})
=-\tfrac{[K:\Q_p]}{\log\chi(\gamma_K)}{\rm Tr}\circ\theta(\alpha)$$
\end{proposition}
\begin{proof}
Denote by
$\delta_K: \B^{[u,v]}_{K_\infty}\to \Q_p$
the map
$\alpha\mapsto\delta_K(\alpha):=
{\rm Tr}_K\circ h^2_K((\varphi-1)\tfrac{\alpha}{t}\otimes\tfrac{d\pi}{1+\pi})$.
Then
$\delta_K$ is identically $0$ on ${\rm Ker}\,\theta=t\B^{[u,v]}_{K_\infty}$ 
(because $(\varphi-1)\tfrac{\alpha}{t}$ is then a coboundary), 
hence it factors through $\B^{[u,v]}_{K_\infty}/t=\widehat K_\infty$. It commutes
with the action of $\G_{\Q_p}$ (i.e.,~$\delta_{\sigma(K)}\circ\sigma=\delta_K$, for all $\sigma\in \G_{\Q_p}$). So there exists $c(K)$ such that 
$\delta_K=c(K){\rm Tr}\circ\theta$.

To determine $c(K)$, it is enough to compute the value 
of the term on the left-hand side for $\alpha=1$, 
which can be done using Proposition~\ref{fanto23}.
We have $(\varphi-1)\frac{1}{t}=\frac{1-p}{p\,t}$,
and ${\rm Tr}_{K_\infty/F_\infty}$ is multiplication
by $[K_\infty:F_\infty]$ on $\B^{[u,v/p]}_{F_\infty}$ (which contains $\frac{1}{t}$).
Finally, $\frac{1}{t}\frac{d\pi}{1+\pi}=(\frac{1}{\pi}+\sum_{n\geq 0}a_n\pi^n)\,d\pi$ in
$\B^{[u,v/p]}_{\Q_p}\otimes\frac{d\pi}{1+\pi}$, 
and ${\rm res}_\pi(\frac{1}{t}\frac{d\pi}{1+\pi})=1$.

We get $c(K)=\frac{1-p}{p\,|\Delta_K|}\,[K_\infty:F_\infty]$, and
we use (i) and (ii) of Remark~\ref{tau2} to get
$c(K)=-\tfrac{[K:\Q_p]}{\log\chi(\gamma_K)}$ which concludes the proof.
\end{proof}

\subsection{Tate's formulas} We will present now a generalization of Tate's computations of the Galois cohomology of $C$. 
 \subsubsection{Classical Tate's formulas} 
We will often use  the following well-known  isomorphisms: 
\begin{align}
\label{old1}
H^i({\sg}_{{K}},{C}(j)) & \stackrel{\sim}{\leftarrow}
\begin{cases} {K}&{\text{if  $j=0$ and  $i=0,1,$}}\\
0 &{\text{otherwise.}}\end{cases}
\end{align}
For $i=0$ the above  isomorphism is given by the canonical map ${K}\to H^0({\sg}_{K},{C})$; for $i=1$ -- by the ${K}$-linear map sending  $1$ to the $1$-cocycle $\log \chi $, where $\chi$ is the cyclotomic character. 
\begin{remark}
The following  rescaling of the map from \eqref{old1} will be useful later. Consider the composition
\begin{align*}
\alpha_{0}: {\rm Kos}_{\gamma}({{K}})\stackrel{}{\to} 
{\rm Kos}_{\gamma}({\wh{K}}_{\infty})\xleftarrow[\sim]{\lambda}  
 C(\Gamma'_K,\wh{K}'_{\infty}) \xrightarrow[\sim]{} 
 C(\Gamma_K,\wh{K}_{\infty}) \xrightarrow[\sim]{} C(\sg_K,C),
\end{align*}
where the last  three complexes are complexes of nonhomogeneous condensed 
cochains\footnote{Nonhomogeneous version of \eqref{cond11}.}, and we have 
$${\rm Kos}_{\gamma}({\wh{K}}_{\infty})=[{\wh{K}}'_{\infty}\lomapr{\gamma_K-1}{\wh{K}}'_{\infty}],
\quad {\rm Kos}_{\gamma}({K}):=[{K}\lomapr{\gamma_K-1}{K}]=[{K}\lomapr{0}{K}].
$$
Map $\lambda$ is given by the identity in degree $0$,  evaluation on $\gamma_K$ in degree $1$, and $0$ in higher degrees. 
On cohomology level $\alpha_{0}$ yields:  for $i=0$,   the canonical map ${K}\to H^0({\sg}_K,{C})$; for $i=1$, the ${K}$-linear map sending  $1$ to the $1$-cocycle 
$\frac{\log \chi}{ \log \chi (\gamma_K)}$.  
\end{remark}

 \subsubsection{Generalized   Tate's formulas}
 \begin{proposition}
   Let  $W \in C_K$ be a Banach space, nuclear Fr\'echet, or  a space of compact type equipped with a trivial action of $\sg_K$.  Then we have isomorphisms:
\begin{align}
\label{newton15}
H^{i}(\sg_K,W(j){\otimes}^{\Box}_KC) & \stackrel{\sim}{\leftarrow}
\begin{cases} W&{\text{if $j=0$ and  $i=0,1$},}\\
0 &{\text{otherwise.}}\end{cases}
\end{align}
For $j=0$, the maps are defined in the obvious way in the case  $i=0$  
and  as the cup product with the class of $\frac{\log \chi}{ \log \chi (\gamma_K)}$ in $H^1(\sg_K, C)$ in the case $i=1$.  
\end{proposition}
\begin{proof} 
Assume first that $W$ is a Banach space. We will use the fact that,
 if  $X$ is profinite and  $Y$ is a  Banach space, 
all continuous functions from  $X$ to  $Y^0/p^n$ (where  $Y^0$ is the unit ball of  $Y$)
are locally constant.  We can apply this to  $X=\sg_K\times\cdots\times \sg_K$ and use original Tate's arguments to deduce that 
the inflation map 
 $H^i(\Gamma'_K,\widehat K'_\infty\otimes^{\Box}_K W)\to H^i(\G_K, C\otimes^{\Box}_K W)$ is an  isomorphism.
 Then,  we use the fact that  $\widehat{K}'_\infty=K\oplus R$ 
with $\gamma_K-1$ invertible with a continuous inverse on 
 $R$. It follows that $H^i(\Gamma'_K,R\otimes^{\Box}_K W)=0$, for all   $i$, and 
that the canonical map  $H^i(\Gamma'_K,W)\to H^i(\Gamma'_K,\widehat K'_\infty\otimes^{\Box}_K W)$
 is an  isomorphism.
We have proved our proposition in the Banach case.

 Assume now that $W$ is a nuclear Fr\'echet space.  Write $W\simeq \lim_n W_n$, for a projective system $\{W_n\}_{n\in\N}$, of Banach spaces with compact transition maps and such that the projections $p_s: \lim_n W_n\to W_s$, $s\in\N$,  have dense images (see \cite[Ch. 16]{Schn} for why this is possible). We have 
\begin{align}\label{ciemno11}
H^i(\sg_K,W(j)\otimes^{\Box}_KC) & \simeq H^i(\sg_K,(\lim_nW_n(j))\otimes^{\Box}_KC)\simeq  H^i(\sg_K,\lim_n(W_n(j)\otimes^{\Box}_KC))\\
 & \stackrel{\sim}{\to} H^i(\sg_K,\R\lim_n(W_n(j)\otimes^{\Box}_KC)) \stackrel{\sim}{\to}\lim_nH^i(\sg_K,W_n(j)\otimes^{\Box}_KC).\notag
\end{align}
The third isomorphism follows from the fact that the pro-system $\{W_n(j)\otimes^{\Box}_KC\}_{n\in\N}$ is  Mittag-Leffler (it is a pro-system of $C$-Banach spaces with projection maps $p_s{\otimes} {\rm Id}$ having dense images).
The fourth isomorphism follows from the fact that $\R^1\lim_nH^{i-1}(\sg_K,W_n(j)\otimes^{\Box}_KC)=0$ since the pro-system $\{H^i(\sg_K,W_n(j)\otimes^{\Box}_KC)\}_{n\in\N}$ is  Mittag-Leffler: by \eqref{old1}, this system is or  trivial or  isomorphic to the pro-system $\{W_n\}_{n\in\N}$, with projection maps $p_s$ having dense images maps. 

 Having the topological isomorphisms \eqref{ciemno11}, by \eqref{old1}  we have,
 in the nontrivial cases, topological isomorphisms
$$
\lim_nH^i(\sg_K,W_n(j)\otimes^{\Box}_KC)\stackrel{\sim}{\leftarrow}\lim_nW_n \stackrel{\sim}{\leftarrow}W,
$$
as wanted.

 Finally, assume that $W$ is of compact type.  We  write $W\simeq \colim_n W_n$, for an inductive system $\{W_n\}_{n\in\N}$, of Banach spaces with injective, compact transition maps. Then
 \begin{align}\label{ciemno12}
H^i(\sg_K,W(j)\otimes^{\Box}_KC) & \simeq H^i(\sg_K,(\colim_nW_n(j))\otimes^{\Box}_KC)\simeq  H^i(\sg_K,\colim_n(W_n(j)\otimes^{\Box}_KC))\\
 & \stackrel{\sim}{\leftarrow} \colim_nH^i(\sg_K,W_n(j)\otimes^{\Box}_KC).\notag
\end{align}
The third isomorphism follows from the fact that $\Z[\sg_K]$ is a compact object in ${\rm CondAb}$. 

 Having the  isomorphisms \eqref{ciemno12}, by \eqref{old1}  we have in the nontrivial cases  isomorphisms
$$
\colim_nH^i(\sg_K,W_n(j)\otimes^{\Box}_KC)\stackrel{\sim}{\leftarrow}\colim_nW_n \stackrel{\sim}{\rightarrow}W,
$$
as wanted.
\end{proof}

\section{Pro-\'etale cohomology} In this chapter, we study    the properties of  pro-\'etale cohomology of smooth dagger  curves over $K$. Moreover we assume that the curve is either proper, or Stein, or a dagger affinoid. 
\subsection{Topology on $p$-adic pro-\'etale cohomology} Let $X$ be a  rigid analytic variety over $K$ or~$C$. 
\subsubsection{The condensed approach}

\begin{definition} \label{condensed11}
 \begin{enumerate}
 \item We define the pro-\'etale site of $X$ as
 $
 X_{\proeet}:=X^{\Diamond}_{\qproeet},
 $
 where $X^{\Diamond}$ is the diamond associated to $X$ and $X^{\Diamond}_{\qproeet}$ denotes the quasi-pro-\'etale site of $X^{\Diamond}$ \cite[Def. 14.1]{Sch21}.
 \item  For  a sheaf $\sff $ on $X_{\proeet}$ with values in $\sd({\rm CondAb})$,  the pro-\'etale cohomology complex 
 $$
 \R\Gamma_{\proeet,\Box}(X,\sff)\in \sd({\rm CondAb}).
 $$
 This is because the category ${\rm CondAb}$ is closed under all limits and colimits (see Section \ref{colibre1}). 
The  pro-\'etale cohomology groups $H^i_{\proeet,\Box}(X,\sff)$, for $i\geq 0$,  are  objects of ${\rm CondAb}$. Similarly, for a sheaf $\sff $ with values in $\sd({\rm Mod}^{\rm cond}_{\Q_p})$.
 \end{enumerate}
 \end{definition}
 If a sheaf $\sff$  in Definition \ref{condensed11} has values in $\sd({\rm Solid})$ then $\R\Gamma_{\proeet,\Box}(X,\sff)$ has values in $\sd({\rm Solid})$ as well because the category ${\rm Solid}$ is closed under all limits and colimits (see Section \ref{colibre2}). Similarly, for the category ${\rm Mod}^{\rm solid}_{\Q_p}$.
 \subsubsection{Comparison with the classical approach} 
 Recall that $\R\Gamma(X,\Q_p):=\R\Gamma_{\proeet}(X,\Q_p)\in \sd(C^{\rm Hcg}_{\Q_p})$. Locally it is a complex of Banach spaces over $\Q_p$; globally -- of Fr\'echet spaces over $\Q_p$. 

 \begin{lemma}  \label{morning1}
\begin{enumerate}
\item We have a natural quasi-isomorphism in $\sd({\rm Mod}^{\cond}_{\Q_p})$ 
$$
{\rm CD}(\R\Gamma(X,\Q_p))\simeq \R\Gamma_{\Box}(X,\Q_p).
$$
\item We have a natural   isomorphism in $\sd(C_{\Q_p})$
$$
 \R\Gamma_{\Box}(X,\Q_p)(*)_{\rm top}\simeq \R\Gamma(X,\Q_p).
$$
\item We have a natural quasi-isomorphism in ${\rm Mod}^{\cond}_{\Q_p}$
$${\rm CD}(\wt{H}^i(X,\Q_p))\simeq H^i_{\Box}(X,\Q_p). $$
\end{enumerate}
\end{lemma}
\begin{proof} Claims  (1) and (2) follow from the fact that $\R\Gamma(X,\Q_p)$ is represented locally by Galois cohomology of the fundamental group and we have Lemma \ref{niedziela1}. Claim (3) follows from claim (1) and Section \ref{res21}. 
\end{proof}
\begin{remark} \label{morning1other}By the same arguments, the analog of Lemma \ref{morning1} holds for  de Rham cohomology (de Rham complex) as well as for  Hyodo-Kato cohomology  (see \cite[Sec. 4.2]{CN4} for the definition of the latter).
\end{remark}
\subsection{Hochschild-Serre spectral sequence}  
We record  here the  Hochschild-Serre spectral sequence for pro-\'etale cohomology.
\begin{lemma}\label{Serre}{\rm (Bosco, \cite[Prop. 4.12]{Bosc})}
Let $X$ be a rigid analytic  variety over $K$. There is a  natural Hochschild-Serre spectral sequence
\begin{equation}
\label{chicago0}
E^{a,b}_2 =H^a(\sg_K,H^b_{\Box}(X_C,\Q_p(j)))\Rightarrow H^{a+b}_{\Box}(X,\Q_p(j)).
\end{equation}
\end{lemma} 
\begin{proof} We pass to the world of diamonds. 
Since $X_C\to X$ is a $\sg_K$-torsor, we have isomorphisms
\begin{equation}\label{pierre1}
(X_C/X)^n\simeq X_C\times \sg_K^{n-1},\quad n\geq 1.
\end{equation}
It suffices thus to show that, for any  adic space $Y$ over ${\rm Spa}(K,\so_K)$ and any profinite set $S$,
we have a natural quasi-isomorphism in $\sd({\rm CondAb})$ (take $Y=X_C$ and $S=\sg_K^{n-1}$ in the notation from \eqref{pierre1})
$$\R\Gamma(Y\times S,\Q_p)\simeq  \R\underline{\Hom}(\Z_{\Box}[S], \R\Gamma(Y,\Q_p)).
$$

  This is local on $Y$, hence  we may assume that 
  $Y $ is  a w-contractible space over ${\rm Spa}(K,\so_K)$. 
Then we have the following natural quasi-isomorphisms 
\begin{equation}\label{chicago10}
\R\Gamma(Y\times S,\Q_p) \simeq \underline{\Q_p(Y\times S)} \simeq  \underline{\scc(S,\Q_p(Y ))} \simeq \underline{\Hom}(\Z[S],\underline{\Q_p(Y )}) \simeq \underline{\Hom}(\Z_{\Box}[S],\underline{\Q_p(Y )}).
 \end{equation}
To see  the    first quasi-isomorphism note that, for any profinite set $T$, we have 
$$
\R\Gamma(Y\times T,\Z/p^n) \simeq \R\Gamma_{\eet}(Y\times T,\Z/p^n) \simeq \Z/p^n(Y\times T),
$$ 
where the last quasi-isomorphism holds because $Y\times T$ is a strictly  totally disconnected perfectoid space (by \cite[Lemma 7.19]{Sch21}),  and  the pro-system $\{\Z/p^n(Y\times T)\}_{n\in\N}$ is   Mittag-Leffler. The second quasi-isomorphism in \eqref{chicago10} follows from the fact that, for any profinite set $S^{\prime}$, 
 $$\scc(S^{\prime},\scc(S, \Q_p(Y )))\simeq \scc(S^{\prime} \times S,\Q_p(Y )).$$  

 Now,  since  $\Z_{\Box}[S]$ is an internally projective object in {\rm Solid}, the right-hand side of \eqref{chicago10}
identifies with $$ \R \underline{\Hom}(\Z_{\Box}[S],\underline{\Q_p(Y )})\simeq  \R \underline{\Hom}(\Z_{\Box}[S],\R\Gamma(Y,\Q_p)),
$$ as wanted.
 \end{proof}

\subsection{Cohomology of Stein curves}We will now discuss arithmetic and geometric pro-\'etale cohomology of smooth Stein curves. 
\subsubsection{Geometric  cohomology of Stein varieties} 
We start with geometric pro-\'etale cohomology. 
 Let $X$ be a smooth Stein  variety over $K$, geometrically irreducible.  By \cite[Th. 5.14]{CN5},
$\R\Gamma(X_C,\Q_p(j))\in\sd(C_{\Q_p})$ has classical  cohomology. Moreover, $H^i(X_C,\Q_p(j))$, for $i\geq 0$,  is  Fr\'echet  and we have a Galois equivariant strict map of  strictly exact sequences of Fr\'echet spaces
 \begin{equation}
 \label{Stein}
 \xymatrix@R=5mm@C=6mm{
 0\ar[r] &  \Omega^{i-1}(X_C)/\ker d\ar[r] \ar@{=}[d] & H^i(X_C,\Q_p(i))\ar[r] \ar[d]^{\alpha} &  (H^i_{\rm HK}(X_C)\wh{\otimes}_{\breve{F}}\wh{\B}^+_{\st})^{N=0,\phi=p^i}\ar[d]^{\iota_{\rm HK}\otimes\theta} \ar[r] & 0\\
  0\ar[r] &  \Omega^{i-1}(X_C)/\ker d\ar[r]^-{d} & \Omega^i(X_C)^{d=0}\ar[r] & H^i_{\dr}(X_C)\ar[r] & 0.}
 \end{equation}
 {\bf Warning:} These spaces are not nuclear Fr\'echet over $\Q_p$ ($C$ is not a nuclear Banach space over $\Q_p$ because it is not finite dimensional over $\Q_p$) !
 
 Similarly, in the condensed language, we have the following:
  \begin{lemma} The cohomology  $H^i_{\Box}(X_C,\Q_p(j))$, for $i\geq 0$,  is  Fr\'echet  and we have a Galois equivariant  map of  exact sequences of Fr\'echet spaces
\begin{equation}
 \label{Stein-cond}
 \xymatrix@R=5mm@C=6mm{
 0\ar[r] &  \Omega^{i-1}(X_C)/\ker d\ar[r] \ar@{=}[d] & H_{\Box}^i(X_C,\Q_p(i))\ar[r] \ar[d]^{\alpha} & 
  (H^i_{{\rm HK},\Box}(X_C){\otimes}^{\Box}_{\breve{F}}\wh{\B}^+_{\st})^{N=0,\phi=p^i}\ar[d]^{\iota_{\rm HK}\otimes\theta} \ar[r] & 0\\
  0\ar[r] &  \Omega^{i-1}(X_C)/\ker d\ar[r]^-{d} & \Omega^i(X_C)^{d=0}\ar[r] & H^i_{\dr,\Box}(X_C)\ar[r] & 0.}
 \end{equation}
 \end{lemma}
 \begin{proof} We apply the functor ${\rm CD}(-)$ to the diagram \eqref{Stein}. Since all the spaces in that diagram are Fr\'echet, by Lemma \ref{res2},  we obtain a map of exact sequences. By Lemma \ref{morning1} and Remark \ref{morning1other}, it remains to show that
 $$
{\rm CD} ((H^i_{{\rm HK}}(X_C)\wh{\otimes}_{\breve{F}}\wh{\B}^+_{\st})^{N=0,\phi=p^i}) \simeq (H^i_{{\rm HK},\Box}(X_C){\otimes}^{\Box}_{\breve{F}}\wh{\B}^+_{\st})^{N=0,\phi=p^i}.
 $$
 Or, since the functor ${\rm CD}(-)$ is left exact, that
 $$
 {\rm CD} (H^i_{{\rm HK}}(X_C)\wh{\otimes}_{\breve{F}}\wh{\B}^+_{\st})\simeq (H^i_{{\rm HK},\Box}(X_C){\otimes}^{\Box}_{\breve{F}}\wh{\B}^+_{\st}).
 $$
 
   But, since $H^i_{{\rm HK}}(X_C)\simeq \lim_n H^i_{{\rm HK}}(X_{n,C})$, where $H^i_{{\rm HK}}(X_{n,C})$ are of finite rank over $\breve{F}$, we have
 \begin{align*}
{\rm CD} (H^i_{{\rm HK}}(X_C)\wh{\otimes}_{\breve{F}}\wh{\B}^+_{\st}) & \simeq 
  {\rm CD}(\lim_nH^i_{{\rm HK}}(X_{n,C})\wh{\otimes}_{\breve{F}}\wh{\B}^+_{\st})\simeq   \lim_n{\rm CD} (H^i_{{\rm HK}}(X_{n,C})\wh{\otimes}_{\breve{F}}\wh{\B}^+_{\st})\\
   & \simeq 
   \lim_n({\rm CD} (H^i_{{\rm HK}}(X_{n,C})){\otimes}^{\Box}_{\breve{F}}\wh{\B}^+_{\st})\simeq   \lim_n(H^i_{{\rm HK},\Box}(X_{n,C}){\otimes}^{\Box}_{\breve{F}}\wh{\B}^+_{\st})\\
    & \simeq  (H^i_{{\rm HK},\Box}(X_C){\otimes}^{\Box}_{\breve{F}}\wh{\B}^+_{\st}),
   \end{align*}
   The   second isomorphism follows from the fact  that the functor ${\rm CD}(-)$ commutes with limits; the third one --  from the fact that  $H^i_{{\rm HK}}(X_{n,C})$ is of finite rank. The last one -- from the fact that $\wh{\B}^+_{\st}$ is Banach and $H^i_{{\rm HK},\Box}(X_C)$ is a product of spaces of finite rank. 
 \end{proof}
 {\bf Notation:} From now on we will omit the $\Box$ in $\R\Gamma_{\Box}(-,-)$ and other cohomologies. This should not cause confusion. 
  \subsubsection{Arithmetic cohomology of Stein curves} 
We pass now to  arithmetic pro-\'etale cohomology. 
  Let $X$ be a smooth Stein curve over $K$, geometrically irreducible. 
We will look at  its arithmetic pro-\'etale cohomology complex
$\R\Gamma(X,\Q_p(j))\in\sd(\Q_{p,\Box}).$
\begin{theorem}\label{final2}
\begin{enumerate}
\item The cohomology of $\R\Gamma(X,\Q_p(j)), j\in\Z,$ is  nuclear Fr\'echet. 
\item  Let $\{X_n\}_{n\in\N}$ be a strictly increasing open covering of $X$ by Stein varieties and let $ i,j\in\Z$. Then $\R^1\lim_nH^i(X_n,\Q_p(j))=0$. Hence we have an isomorphism
$$
H^i(X,\Q_p(j))\stackrel{\sim}{\to} \lim_nH^i(X_n,\Q_p(j)).
$$
\end{enumerate}
\end{theorem}
\begin{proof}{\bf  Claim (1).} By \eqref{chicago0}, we have  the Hochschild-Serre spectral sequence 
\begin{equation}
\label{ss12}
E^{a,b}_2 =H^a(\sg_K,H^b(X_C,\Q_p(j)))\Rightarrow H^{a+b}(X,\Q_p(j)).
\end{equation}
From diagram \eqref{Stein-cond}, we know that 
the only  nontrivial  cohomology groups of $X_C$ are in degrees 0,1. Hence, from the spectral sequence \eqref{ss12}, we get that $H^i(X,\Q_p(j))=0$, for $i\geq 4$,  and we have the long exact sequence 
    \begin{align}\label{ncis0}
0   \to&  H^0(\sg_K,H^0(X_C,\Q_p(j)))\to H^0(X,\Q_p(j))\to H^{-1}(\sg_K,H^1(X_C,\Q_p(j)))\\
   \to & H^1(\sg_K, H^0(X_C,\Q_p(j)))\to H^1(X,\Q_p(j))\to  H^0(\sg_K,H^1(X_C,\Q_p(j)))\notag\\
      \stackrel{d_2}{\to} & H^2(\sg_K, H^0(X_C,\Q_p(j)))\to H^2(X,\Q_p(j))\to  H^1(\sg_K,H^1(X_C,\Q_p(j)))\notag\\
            \to & H^3(\sg_K, H^0(X_C,\Q_p(j)))\to H^3(X,\Q_p(j))\to  H^2(\sg_K,H^1(X_C,\Q_p(j)))\to 0\notag
 \end{align}
 
  \vskip2mm
  (i) {\em The groups $H^0(X,\Q_p(j))$ and $H^3(X,\Q_p(j))$.} \vskip2mm
  
  Diagram \eqref{ncis0} yields the isomorphisms
 \begin{align}\label{ncis2}
 & H^0(X,\Q_p(j))= \Q_p,\\
  &  H^3(X,\Q_p(j)) \stackrel{\sim}{\to}H^2(\sg_K, H^1(X_C,\Q_p(j))).\notag
 \end{align}
 The top line of  diagram \eqref{Stein-cond} gives  the exact sequence
 \begin{equation*}
 \label{ncis1}
 0\to \so(X_C)/C\to H^1(X_C,\Q_p(1))\to {\rm HK}^1(X_C,1)\to 0
 \end{equation*}
Applying  Galois cohomology to it we get the exact sequence (we set $s:j-1$)
 \begin{align}\label{ncis4}
  0 & \to H^0(\sg_K, (\so(X_C)/C)(s))\to H^0(\sg_K,H^1(X_C,\Q_p(1))(s))\to H^0(\sg_K, {\rm HK}^1(X_C,1)(s))\\
    & \to H^1(\sg_K, (\so(X_C)/C)(s))\to H^1(\sg_K,H^1(X_C,\Q_p(1))(s))\to H^1(\sg_K, {\rm HK}^1(X_C,1)(s))\notag\\
   &\to H^2(\sg_K, (\so(X_C)/C)(s))\to H^2(\sg_K,H^1(X_C,\Q_p(1))(s))\to H^2(\sg_K, {\rm HK}^1(X_C,1)(s))\to 0.\notag
 \end{align}
 Using it, the isomorphisms \eqref{ncis2}, and the  generalized Tate's formulas \eqref{newton15}, we get the isomorphism
 \begin{equation}
 \label{sund1}
 H^3(X,\Q_p(j)) \stackrel{\sim}{\to}H^2(\sg_K, {\rm HK}^1(X_C,1)(s)).
 \end{equation}
 
 \vskip2mm
 (ii) {\em Key lemma.} Claim (1) of Theorem \ref{final2} follows from the following fact. 
 \vskip2mm
   Let $\{X_n\}_{n\in\N}$ be a strictly increasing covering of $X$ by adapted naive interiors of affinoids, i.e., there exists  a strictly increasing (Stein) covering $\{\overline{X}_n\}_{n\in\N}$ of $X$ such that $X_{n+1}$ is a naive interior in $\overline{X}_{n+1}$ adapted to $X_{n}$. 
   
    \begin{remark}\label{naive1}
By definition, a {\em naive interior} of a smooth (dagger) affinoid is a Stein subvariety whose complement is open and quasi-compact. It is easy to see that, for a pair of (dagger) affinoids 
$X_1\Subset X_2$ there exists a naive interior $X_2^0\subset X_2$ such that $X_1\subset X_2^0\subset X_2$. We will say  that $X_2$ is {\em adapted} to $X_1$. 
 \end{remark}
  
\begin{lemma}\label{paris20}
The transition maps
  \begin{equation}\label{}
  f_{i,n}: \quad H^i(X_{n+1},\Q_p(j))\to H^i(X_n,\Q_p(j)), \quad n\geq 0,
  \end{equation}
  are compact maps of  (nuclear) Fr\'echet spaces.
  \end{lemma}
  \begin{proof}
   By the  computations in (i), which can be applied to each $X_n$ since that variety  is Stein, this is clear for  $i=0$. 
   
For  $i=3$,  the isomorphism \eqref{sund1} above combined with the fact that  ${\rm HK}^1(X_{n,C},1)$ is 
an almost $C$-representation (because $H^1_{\rm HK}(X_{n,C})$ is of finite rank over $\breve{F}$ since $H^1_{\dr}(X_{n,C})$ is of finite rank over $C$) yields that   $H^3(X_n,\Q_p(j))$  is a  finite rank $\Q_p$-vector space. Hence the maps $f_{3,n}$ are as wanted. 
   
   It remains to treat the cases of $i=1,2$.
 We start with showing  that the spaces $$H^a(\sg_K,H^b(X_{n,C},\Q_p(j))),\quad a,b\in\Z,
 $$ appearing the spectral sequence \eqref{ss12} are nuclear Fr\'echet. To see that, we apply  Galois cohomology to the top row of diagram \eqref{Stein-cond} for $X_n$  and obtain the exact sequence (we set $s:=j-b$; ${\rm HK}^j(X_{n,C},i):=(H^j_{{\rm HK}}(X_{n,C}){\otimes}^{\Box}_{\breve{F}}\wh{\B}^+_{\st})^{N=0,\phi=p^i}$)
\begin{align}\label{help1}
  \to H^{a-1}(\sg_K, & {\rm HK}^b(X_{n,C},b)(s))   \stackrel{\partial_{a-1}}{\longrightarrow}  H^a(\sg_K,(\Omega^{b-1}(X_{n,C})/\ker d)(s))\\
   & \to  H^a(\sg_K,H^b(X_{n,C},\Q_p(j))) \to   H^a(\sg_K, {\rm HK}^b(X_{n,C},b)(s)) \notag \\
    & \stackrel{\partial_{a}}{\longrightarrow}   H^{a+1}(\sg_K,(\Omega^{b-1}(X_{n,C})/\ker d)(s))\to \notag
\end{align}  We claim that
the spaces  $H^{i}(\sg_K,(\Omega^{b-1}(X_{n,C})/\ker d)(s))$ and $ H^{i}(\sg_K,  {\rm HK}^b(X_{n,C},b)(s)) $ are nuclear Fr\'echet. Indeed, for  the first one this follows from  the generalized Tate's  isomorphism \eqref{newton15}:
 if nontrivial $$
 H^{i}(\sg_K,(\Omega^{b-1}(X_{n,C})/\ker d)(s))\simeq \Omega^{b-1}(X_n)/\ker d
 $$
 since $\Omega^{b-1}(X_{n,C})/\ker d\simeq (\Omega^{b-1}(X_n)/\ker d)\otimes_{K}^{\Box}C$; 
  and the fact that $\Omega^{b-1}(X_n)/\ker d$ is a nuclear Fr\'echet. 
For the second one, we use the  fact that $H^{i-1}(\sg_K, {\rm HK}^b(X_{n,C},b)(s))$ is a finite rank $\Q_p$-vector space by the isomorphism \eqref{newton15} (since  ${\rm HK}^b(X_{n,C},b)$ is an almost $C$-representation).

 The above computations  imply that
the maps $\partial_{a-1}$ and $\partial_{a}$ in \eqref{help1} are between nuclear Fr\'echet spaces hence $H^a(\sg_K,H^b(X_{n,C},\Q_p(j)))$ is an extension of two nuclear Fr\'echet spaces. In fact, it is an extension of a  finite rank $\Q_p$-vector space by a nuclear Fr\'echet space  hence a nuclear Fr\'echet space.

  We proceed now to prove Lemma \ref{paris20} for $i=1,2$. From \eqref{ncis0}, we  get a long exact sequence
 \begin{align}\label{lane2}
0 &  \to H^1(\sg_K, \Q_p(j))\to H^1(X_n,\Q_p(j))\to H^0(\sg_K,H^1(X_{n,C},\Q_p(j)))\\
  & \stackrel{d_{2,n}}{ \to } H^2(\sg_K, \Q_p(j))\to H^2(X_n,\Q_p(j)) \to H^1(\sg_K,H^1(X_{n,C},\Q_p(j)))\to 0\notag
  \end{align}
 
 (a) {\em Case $j\neq 1$.} In this case, we have   $H^2(\sg_K, \Q_p(j))=0$.  Thus, it suffices to show that the spaces $H^i(\sg_K,H^1(X_{n,C},\Q_p(j)))$, for $i=0,1$,  are of finite rank over $\Q_p$. For that, since  
  the vector spaces $H^i(\sg_K, {\rm HK}^1(X_{n,C},1)(j-1))$, for $i=0,1$, are of finite rank over $\Q_p$, it suffices to notice that $H^i(\sg_K, (\so(X_{n,C})/C)(j-1))=0$, for $i=0,1$, by the generalized Tate's formulas \eqref{newton15}.
 
  
(b) {\em Case $j=1$.} To start, we claim that the spaces $H^i(X_n,\Q_p(1))$, $i=1,2$  are nuclear Fr\'echet. Indeed, we have shown above that   the spaces $H^i(\sg_K,H^1(X_{n,C},\Q_p(j)))$ are   Fr\'echet. We also know that we have exact sequences
$$
0\to V_{0,i,n}\stackrel{g_{i,n}}{\to} V_{1,i,n}\to  H^i(X_n,\Q_p(j))\to 0,
$$
with $V_{0,i,n}, V_{1,i,n}$ solid Fr\'echet $K$-vector spaces. Moreover,  $H^i(X_n,\Q_p(j))$ is quasi-separated since it is  an extension of quasi-separated solid $K$-vector spaces (see \eqref{lane2}). It follows that the map $g_{i,n}$ is quasi-compact and, hence, the induced map 
$g_{i,n}: V_{0,i,n}(*)_{\rm top}\to V_{1,i,n}(*)_{\rm top}$ is a closed embedding. As a result, $V_{1,i,n}(*)_{\rm top}/V_{0,i,n}(*)_{\rm top}$ is a (classical) Fr\'echet space and then this implies that $H^i(X_n,\Q_p(j))\simeq {\rm CD}(V_{1,i,n}(*)_{\rm top}/V_{0,i,n}(*)_{\rm top})$ is Fr\'echet, as wanted. By Lemma \ref{ext10}, as an extension of a nuclear Fr\'echet space by a finite rank vector space, it is nuclear.

   We will show below (in   fact (c)) that the pro-systems $\{H^i(\sg_K,H^1(X_{n,C},\Q_p(j)))\}_{n\in\N}$ have compact transition maps (we will call such systems {\em compact}).  Then the pro-system $\{\ker d_{2,n}\}_{n\in\N}$ is also compact. And, by Lemma \ref{ext10.1}, the maps $f_{i,n}$, $i=1,2$, from our theorem are compact, as wanted.

  (c) {\em  The pro-systems $\{H^i(\sg_K,H^1(X_{n,C},\Q_p(1))(s))\}_{n\in\N}$, for $i=0,1$, are compact.} 
  To prove that, note that
  the vector spaces $H^i(\sg_K, {\rm HK}^1(X_{n,C},1)(s))$, for $i=0,1,$ are of finite rank over $\Q_p$. Hence, by \eqref{ncis4},  it suffices  to show that the pro-systems 
  $\{ H^i(\sg_K, (\so(X_{n,C})/C)(s))\}_{n\in\N}$ are compact. But this is clear since, by the generalized  Tate's formulas \eqref{newton15} (note that $\so(X_{n})/K$ is a nuclear Fr\'echet), these pro-systems are or trivial or we have   isomorphisms
$$
  \{ H^i(\sg_K, (\so(X_{n,C})/C)(s))\}_{n\in\N}  \simeq \{ \so(X_{n})/K\}_{n\in\N}.\qedhere
$$
\end{proof}
  
 This finishes the proof  of  claim (1) of the theorem.
 
 \vskip2mm
  {\bf Claim (2).}  It suffices to show that 
 $$
 H^i(X,\Q_p(j))\stackrel{\sim}{\to} \lim_nH^i(X_n,\Q_p(j)).
 $$
Or  that $\R^1\lim_nH^i(X_n,\Q_p(j))=0$, $i, j\in\Z$. Since the spectral sequence \eqref{ss12} degenerates at $E_3$ it suffices to show that $\R^1\lim_n E^{a,b}_3(X_n)=0$. Since cohomological dimension of $\sg_K$ is $2$, we have
 $$
 E_3^{a,b}(X_n)=\begin{cases} \ker d_2^{0,b}(X_n) & \mbox{ if } a=0,\\
  E_2^{1,b}(X_n)& \mbox{ if } a=1,\\
  \coker d_2^{2,b+1}(X_n) & \mbox{ if } a=2,
 \end{cases}
 $$
 where $d_2^{0,b}(X_n): H^0(\sg_K, H^b(X_{n,C},\Q_p(j)))\to H^2(\sg_K, H^{b-1}(X_{n,C},\Q_p(j)))$ is the only nontrivial differential in the spectral sequence \eqref{ss12}. Hence, using  the computations from  claim (1), we get
 immediately that $\R^1\lim_n E^{1,b}_3(X_n)=0$ and, since $$\R^1\lim_n H^2(\sg_K, H^{b-1}(X_{n,C},\Q_p(j)))=0,$$ that $\R^1\lim_n E^{2,b}_3(X_n)=0$.  It remains to show that 
 $$
 \R^1\lim_nE^{0,b}_3(X_n)= \R^1\lim_n\ker d_2^{0,b}(X_n)=0.
 $$ 
 
   From the exact sequence \eqref{help1}, since $H^0(\sg_K, {\rm HK}^b(X_{n,C},b)(s))$ is finite over $\Q_p$ (because ${\rm HK}^b(X_{n,C},b)(s)$  is an almost $C$-representation), it suffices to show that $\R^1\lim_nH^0(\sg_K,\Omega^{b-1}(X_{n,C})/\ker d)=0$. But, by the generalized Tate's  isomorphism \eqref{newton15}, 
 $H^0(\sg_K,\Omega^{b-1}(X_{n,C})/\ker d)\simeq \Omega^{b-1}(X_{n})/\ker d$, so it suffices to show that $\R^1\lim_n \Omega^{b-1}(X_{n})=0$ but this is known. 
\end{proof}
\subsection{Filtration} \label{filtration1} Let $X$ be a smooth  geometrically irreducible Stein analytic curve over $K$. 
Let $i,j\in\Z$. Under certain conditions, there exists an ascending filtration on $H^{i}(X,\Q_p(j))$:
$$
F^2_{i,j}=H^{i}(X,\Q_p(j))\supset F^1_{i,j}\supset F^0_{i,j}\supset F^{-1}_{i,j}=0,
$$
such that we have  isomorphisms
\begin{align*}
 & F^2_{i,j}/F^1_{i,j}\simeq H^{i-1}(\sg_K, \Q_p(j-1)),\\
& F^1_{i,j}/F^0_{i,j}\simeq H^{i-1}(\sg_K, \tfrac{\so(X_C)}{C}(j-1)),\\
  & F^0_{i,j}/F^{-1}_{i,j}\simeq H^i(\sg_K, {\rm HK}^1(X_C,1)(j-1)),
\end{align*}
where we set ${\rm HK}^1(X_C,1)=(H^1_{{\rm HK}}(X_C){\otimes}^{\Box}_{\breve{F}}\wh{\B}^+_{\st})^{N=0,\phi=p}$. 
We can  visualize this filtration in the following way:
\begin{equation}
\label{nie22}
\xymatrix@C=3mm@R=5mm{ & &0\ar[d]  & 0\ar[d]\\
0\ar[r] &  F^0_{i,j}:=H^{i}(\sg_K, \Q_p(j))\ar[r] \ar@{=}[d] & F^1_{i,j} \ar[r] \ar[d] & H^{i-1}(\sg_K, \tfrac{\so(X_C)}{C}(j-1))\ar[d] \ar[r] & 0\\
0\ar[r] &  H^{i}(\sg_K, \Q_p(j))\ar[r] & F^2_{i,j}:=H^i(X,\Q_p(j))\ar[r] \ar[d]^{} &  H^{i-1}(\sg_K, H^{1}(X_C,\Q_p(j)))\ar[r] \ar[d]^{}& 0 \\
& &   H^{i-1}(\sg_K, {\rm HK}^1(X_C,1)(j-1))\ar@{=}[r] \ar[d] &  H^{i-1}(\sg_K,{\rm HK}^1(X_C,1)(j-1))\ar[d]\\
& & 0 & 0
}
\end{equation}
The above  diagram is a map of exact  sequences. The right  column is induced by the syntomic filtration (see diagram \eqref{Stein-cond}). 
The middle  row comes from 
 the Hochschild-Serre spectral sequence \eqref{chicago0} (we note that $\Q_p(j)\simeq H^{0}(X_C,\Q_p(j))$) and the vanishing of geometric cohomologies obtained from diagram \eqref{Stein-cond}. We assume that the middle row and the right column are exact. 
 The term $F^1_{i,j}$ is defined as a pullback of the top right square.

\subsection{Arithmetic cohomology of dagger affinoids} 
Let $X$ be a smooth geometrically connected dagger affinoid over $K$. We will now study its arithmetic pro-\'etale cohomology. They key tool is the (studied above)   arithmetic pro-\'etale cohomology of smooth Stein curves.
\begin{proposition}\label{final3}
The cohomology of $\R\Gamma(X,\Q_p(j)), j\in\Z,$ is  of compact type. 
\end{proposition}
\begin{proof} Let $\{X_h\}$ be the dagger presentation of the dagger structure on $X$. Denote by $X_h^0$  a naive interior of $X_h$ adapted to $\{X_h\}$. The canonical quasi-isomorphism
 $$\R\Gamma(X, \Q_p(j))\simeq \colim_h\R\Gamma(X^0_h,\Q_p(j)),\quad j\in\Z
 $$ yields an isomorphism
$$
H^i(X,\Q_p(j))\stackrel{\sim}{\to} \colim_hH^i(X^0_h,\Q_p(j)),\quad i,j\in\Z.
$$
 We note that $X^0_h$,  $h\in\N$, is a smooth Stein variety.     By Lemma \ref{paris20},    the ind-systems $\{H^i(X^0_h,\Q_p(j))\}_{h\in\N}$, for $i\in\N$, have compact transition maps (between Fr\'echet spaces).
This proves our proposition. 
\end{proof}

\subsection{Examples} 
In this section we will compute $p$-adic pro-\'etale cohomology of open discs,  
annuli, and their boundaries --  the basic building blocks of analytic curves.
 \subsubsection{Open disc}  
Let $D$ be an open disc over $K$.
  \begin{lemma}{\rm (Geometric cohomology)} \label{ball0} Let $j\in\Z$. 
 We have $\sg_K$-equivariant  isomorphisms
\begin{equation}\label{ball1}
   H^i(D_C,\Q_p(j))\simeq \begin{cases}  \Q_p(j) & \mbox{ if } i=0,\\
(\so(D_C)/C)(j-1) & \mbox{ if } i=1,\\
0 & \mbox{ if }  i\geq 2.
\end{cases}
\end{equation}
Moreover,  $d:\so(D_C)/C{\to} \Omega^1(D_C)$ is an isomorphism. 
 \end{lemma}
 \begin{proof} The Hyodo-Kato isomorphism $\iota_{\rm HK}:H^i_{\rm HK}(D_C){\otimes}_{\breve{F}}^{{\Box}}C\simeq H^i_{\dr}(D_C)$ and the fact that 
 $H^i_{\dr}(D_C)=0$, for $i\geq 1$, yield 
 that $H^i_{\rm HK}(D_C)=0$, for $i\geq 1$. 
Now, since $D$ is Stein, our lemma follows  from diagram \eqref{Stein-cond}. 

  The last claim is equivalent to $H^1_{\rm dR}(D_C)=0$.
 \end{proof}
  
Consider now the Hochschild-Serre spectral sequence (from Lemma \ref{Serre}): 
\begin{equation}
\label{HS} 
E^{a,b}_2 =H^a(\sg_K,H^b(D_C,\Q_p(j)))\Rightarrow H^{a+b}(D,\Q_p(j)).
\end{equation} 
By Lemma \ref{ball0}, the only nonzero rows are those of degrees $b=0,1$. We get:
   
\begin{lemma}{\rm (Arithmetic cohomology)}\label{mist1} Let $i\geq 0, j\in\Z$. 
 We have  exact sequences
   \begin{align*}
   & 0\to H^i(\sg_K, H^{0}(D_C,\Q_p(j)))\to H^i(D, \Q_p(j)) \to H^{i-1}(\sg_K, H^{1}(D_C,\Q_p(j)))\to 0,\\
 &   0\to H^i(\sg_K, \Q_p(j))\to H^{i}(D,\Q_p(j))\to H^{i-1}(\sg_K, (\so(D_C)/C)(j-1))\to 0.
 \end{align*}
\end{lemma}
 \begin{proof}
The second exact sequence is a translation of the first, granted formula (\ref{ball1}). 
The spectral sequence \eqref{HS} yields the   exact sequence
 \begin{align}
 \label{mistake1}
 0 & \to H^0(\sg_K,H^0(D_C,\Q_p(j)))\to H^0(D,\Q_p(j))\to H^{-1}(\sg_K,H^1(D_C,\Q_p(j)))\\
  &{\to} H^1(\sg_K,H^0(D_C,\Q_p(j)))\to H^1(D,\Q_p(j))\to H^0(\sg_K,H^1(D_C,\Q_p(j))) \notag\\
  & \stackrel{d_2}{\to} H^2(\sg_K,H^0(D_C,\Q_p(j))\to H^2(D,\Q_p(j))\to H^1(\sg_K,H^1(D_C,\Q_p(j)))\to 0\notag
 \end{align}
First, we  prove that it splits into short exact sequences. If $j \neq 1$, this is trivial as the last terms of the first two lines are $0$.  

Let us now assume that $j=1$. We need to show that
the differential $d_2$ in (\ref{mistake1}) is $0$ or, equivalently, 
that the canonical map 
$H^2(\sg_K,H^0(D_C,\Q_p(j)))\to H^2(D,\Q_p(j))$ is injective. But this map is induced by the projection $D\to K$ and any rational point in $D$ yields a section (such a point always exists).
\end{proof}
\begin{remark} We note that the groups $H^i(D_C,\Q_p(j))$ and $H^i(D,\Q_p(j))$ are  Fr\'echet and nuclear Fr\'echet spaces, by Lemma \ref{Stein} and Theorem \ref{final2},  respectively. 
\end{remark}
 \subsubsection{Open annulus}  Let $A$ be an open annulus over $K$. 
\begin{lemma}{\rm (Geometric cohomology)} \label{annulus1.0} Let $j\in\Z$. 
\begin{enumerate}
\item We have $\sg_K$-equivariant isomorphisms 
\begin{equation} \label{annulus1}
     H^i(A_C,\Q_p(j))\simeq \begin{cases} \Q_p(j)& \mbox{ if } i=0,\\
  0 & \mbox{ if }   i\geq 2.
  \end{cases}
  \end{equation}
  \item We have  a $\sg_K$-equivariant  exact sequence 
   \begin{equation}
   \label{annulus1.1}
    0\to (\so(A_C)/C)(j-1)\to H^1(A_C,\Q_p(j))\to \Q_p(j-1)\to 0
    \end{equation}
admitting  a $\sg_K$-equivariant  $\Q_p$-linear  splitting.
\end{enumerate}
\end{lemma}
\begin{proof}   Since $A$ is  Stein we can use 
 the Galois equivariant map of strictly exact sequences (\ref{Stein}) for $X=A$. 
From the Hyodo-Kato isomorphism $\iota_{\rm HK}:H^i_{\rm HK}(A_C){\otimes}_{\breve{F}}^{{\Box}}C\simeq H^i_{\dr}(A_C)$ and the fact that $H^i_{\dr}(A_C)\simeq C$, for $i=0,1$, and $H^i_{\dr}(A_C)=0$, for $i\geq 2$, we see that 
$$H^0_{\rm HK}(A_C)\simeq \breve{F},\quad H^1_{\rm HK}(A_C)\simeq \breve{F}, \quad H^i_{\rm HK}(A_C)\simeq 0, \, i\geq 2. 
$$ 
The group $H^1_{\rm HK}(A_C)$ is generated by the  Hyodo-Kato symbol  $c_1^{\rm HK}(z)$, for  an arithmetic coordinate $z$ of the annulus.   Frobenius acts on $c_1^{\rm HK}(z)$ by multiplication by $p$ and monodromy is trivial. This implies that
$$
(H^i_{\rm HK}(A_C){\otimes}^{\Box}_{\breve{F}}\wh{\B}^+_{\st})^{N=0,\phi=p^i}\simeq 
\begin{cases}\Q_p & \mbox{if $i=0,1,$}\\
0 & \mbox{if $i\geq 2$.}
\end{cases}
$$
This yields the isomorphisms in our lemma.

   Assume now that $i=1$. Then  the diagram (\ref{Stein}) becomes
 $$
 \xymatrix@R=5mm@C=6mm{
 0\ar[r] & \so(A_C)/C \ar[r] \ar@{=}[d] & H^1(A_C,\Q_p(1))\ar[r]\ar[d]^{\alpha} & \Q_p\ar[d]^{\can} \ar[r] & 0\\
  0\ar[r] & \so(A_C)/C \ar[r]&   \Omega^1(A_C)\ar[r] & C\ar[r] & 0.
 }
 $$
 The top row yields the exact sequence (\ref{annulus1.1}). The term $\Q_p$ comes from a Hyodo-Kato term and is generated by the Hyodo-Kato symbol $c_1^{\rm HK}(z)$. The term $C$ comes from de Rham cohomology and is generated by the de Rham symbol $c_1^{\rm dR}(z)$. These symbols are compatible with each other (via the Hyodo-Kato isomorphism $\iota_{\rm HK}$) and are also compatible with the pro-\'etale symbol $c_1^{\proeet}(z)$. Sending $c_1^{\rm HK}(z)$ to $c_1^{\proeet}(z)$ yields the wanted splitting of the exact sequence (\ref{annulus1.1}).
\end{proof}
 Take now  the Hochschild-Serre spectral sequence: 
\begin{equation}
\label{HSann} 
E^{a,b}_2 =H^a(\sg_K,H^b(A_C,\Q_p(j)))\Rightarrow H^{a+b}(A,\Q_p(j)).
\end{equation} 
By \eqref{annulus1}, the only nonzero rows are those of degrees $b=0,1$. We get:
  
\begin{lemma}{\rm (Arithmetic cohomology)} \label{mist3}
 We have   exact sequences 
   \begin{align*}
 &   0\to H^{i}(\sg_K, H^{0}(A_C,\Q_p(j)))\to H^{i}(A,\Q_p(j))\to H^{i-1}(\sg_K, H^{1}(A_C,\Q_p(j)))\to 0,\\
  &   0\to H^{i-1}(\sg_K, (\so(A_C)/C)(j-1))\to H^{i-1}(\sg_K, H^{1}(A_C,\Q_p(j)))\to H^{i-1}(\sg_K, \Q_p(j-1))\to 0.
\end{align*}
\end{lemma}

 \begin{proof} The second sequence is obtained from the (split) exact sequence  (\ref{annulus1.1}). 
 
  For the rest, we argue exactly as in the case of an open disc (see the proof of 
 Lemma~\ref{mist1}) with the exception of the triviality of the map $d_2$ in the exact sequence \ref{mistake1} when $j=1$: in this case, a rational point in $A$ does not always exist but it does after taking 
a base change to a finite extension $L$ of $K$. Then the map $H^2(\sg_L,H^0(A_C,\Q_p(j)))\to H^2(A_L,\Q_p(j))$ is injective and the triviality of $d_2$ follows. 
\end{proof}
\begin{remark} We note that the groups $H^i(A_C,\Q_p(j))$ and $H^i(A,\Q_p(j))$ are Fr\'echet and nuclear Fr\'echet spaces, by Lemma \ref{Stein} and Theorem \ref{final2},  respectively. 
\end{remark}

 \subsubsection{Ghost circle}
 Take now the ghost circle    $Y:=\partial D$.  
 \begin{lemma}{\rm (Geometric cohomology)}
 \label{ghost0} Let $i\in\N$, $j\in\Z$. 
 \begin{enumerate}
 \item 
 We have  $\sg_K$-equivariant   isomorphisms 
$$
   H^i(Y_C,\Q_p(j))\simeq \begin{cases} \Q_p(j) & \mbox{ if } i=0,\\
  0 & \mbox{ if }  i\geq 2.
  \end{cases}
  $$
  \item We have a  $\sg_K$-equivariant   (split)  exact sequence 
  \begin{equation} 
  \label{ball2-ghost1}
  0\to (\so(Y_C)/C)(j-1)\to H^1(Y_C,\Q_p(j))\to \Q_p(j-1)\to 0. 
  \end{equation}
\end{enumerate}
 \end{lemma}
 \begin{proof}
   By definition, we have 
   \begin{align*}
   \R\Gamma(Y_C,\Q_p(j)) &=\colim_{0 < \varepsilon <1}\R\Gamma(D_C \moins D_C(\varepsilon), \Q_p(j))\\
    & =\colim_{0 < \varepsilon <1}\R\Gamma( A_C(\varepsilon), \Q_p(j)),
   \end{align*}
   where $D_C(\varepsilon)$ is the closed disc of radius $\varepsilon$ (over $C$) and $A_C(\varepsilon):=D_C \moins D_C(\varepsilon)$.
   
   Applying Lemma \ref{annulus1.0},   we get immediately that   ${H}^i(Y_C,\Q_p(j))=0$, for $i\geq 2$, and 
$$
   {H}^1(Y_C,\Q_p(j))  \stackrel{\sim}{\leftarrow}  \colim_{0 < \varepsilon <1}{H}^1( A_C(\varepsilon), \Q_p(j)).
   $$
   From the  exact sequence \eqref{annulus1.1}, we get the exact sequence
   \begin{align*}
0 & \to  \colim_{0 < \varepsilon <1} (\so(A_C(\varepsilon)/C)(j-1))\to \colim_{0 < \varepsilon <1}{H}^1( A_C(\varepsilon), \Q_p(j)) \\
 & \to \colim_{0 < \varepsilon <1}\Q_p(j-1)\to 0.
\end{align*}
Since   $\colim_{0 < \varepsilon <1} (\so(A_C(\varepsilon)/C)(j-1))\stackrel{\sim}{\to} \so(Y_C)/C$,  we get that $  {H}^1(Y_C,\Q_p(j)) $  fits into the  exact sequence \eqref{ball2-ghost1}.
   
      From  Lemma \ref{annulus1.0}, 
   we also get the isomorphism
   $$
   \colim_{0 < \varepsilon <1}{H}^0( A_C(\varepsilon), \Q_p(j)) \stackrel{\sim}{\to}  {H}^0(Y_C,\Q_p(j)).
$$
From it and  the  exact sequence \eqref{annulus1.1},   we get the isomorphism
   $$
      \Q_p(j)\stackrel{\sim}{\to}  {H}^0(Y_C,\Q_p(j)).
   $$
 This    finishes the proof of the lemma.
    \end{proof}

\begin{lemma}{\rm (Arithmetic cohomology)}\label{mist1-ghost} Let $i\in\N, j\in \Z$.  
 We have    exact sequences 
   \begin{align*}
&    0\to H^i(\sg_K, H^{0}(Y_C,\Q_p(j)))\to H^i(Y, \Q_p(j)) \to H^{i-1}(\sg_K, H^{1}(Y_C,\Q_p(j)))\to 0,\\
   & 0\to H^{i-1}(\sg_K, (\so(Y_C)/C)(j-1))\to  H^{i-1}(\sg_K, H^{1}(Y_C,\Q_p(j)))\to H^{i-1}(\sg_K, \Q_p(j-1))\to 0.
   \end{align*}
\end{lemma}
 \begin{proof}
The second exact sequence is obtained from the (split) exact sequence  from Lemma \ref{ghost0}. 

  For the first exact sequence, we write
   \begin{align*}
   \R\Gamma(Y,\Q_p(j)) &=\colim_{0 < \varepsilon _K<1}\R\Gamma(D \moins D(\varepsilon_K), \Q_p(j))\\
    & =\colim_{0 < \varepsilon_K <1}\R\Gamma( A(\varepsilon_K), \Q_p(j)),
   \end{align*}
   where $\varepsilon_K$ are chosen so that the annuli $A(\varepsilon_K)$ are defined over $K$. By Lemma \ref{mist3}, this yields the exact sequence
  \begin{align*}
      0 & \to \colim_{0 < \varepsilon_K <1}H^i(\sg_K, H^{0}(A(\varepsilon_K)_C,\Q_p(j)))\to H^i(Y, \Q_p(j))  \\ & \to 
   \colim_{0 < \varepsilon_K <1}H^{i-1}(\sg_K, H^{1}(A(\varepsilon_K)_C,\Q_p(j)))\to 0.
   \end{align*}
  Hence it suffices to show that
   \begin{align}\label{chicagowarm}
    \colim_{0 < \varepsilon_K <1}H^i(\sg_K, H^{0}(A(\varepsilon_K)_C,\Q_p(j))) &\stackrel{\sim}{\to} H^i(\sg_K, H^{0}(Y_C,\Q_p(j))),\\
      \colim_{0 < \varepsilon_K <1}H^{i-1}(\sg_K, H^{1}(A(\varepsilon_K)_C,\Q_p(j)))  & \stackrel{\sim}{\to} H^{i-1}(\sg_K, H^{1}(Y_C,\Q_p(j))).\notag
   \end{align}
   But this is clear since $\Z[\sg_K]$ is a compact object in ${\rm CondAb}$. 
\end{proof}

 \subsubsection{Boundary of an open annulus}  Let $A$ be an open annulus over $K$. 
  
\begin{corollary}{\rm (Geometric cohomology)}
\label{Vvk} Let $i\in\N, j\in\Z$. 
There is a  $\sg_K$-equivariant canonical  isomorphism  $$H^i(\partial A_C,\Q_p(j)) \stackrel{\sim}{\to}
   H^i(Y_C,\Q_p(j))^{\oplus 2}.
   $$
   Hence we have   $\sg_K$-equivariant  isomorphisms
$$
   H^i(\partial A_C,\Q_p(j)) \simeq \begin{cases}
    \Q_p(j)^{\oplus 2} & \mbox{if } i=0,\\
    0 & \mbox{if } i\geq 2,
  \end{cases}
 $$
and  a  $\sg_K$-equivariant  (split)  exact sequence: 
   $$ 0\to (\so( Y_C)/C)(j-1)^{\oplus 2}\to H^1(\partial A_C,\Q_p(j))\to \Q_p(j-1)^{\oplus 2}\to 0
   $$
 \end{corollary}
 \begin{proof}
We write $A= \{ z \in K \; : \; |a| < |z| < |b| \}$
with $a,b \in K$. Then
$$ \partial A\simeq  \varinjlim_{|a|< \delta \le \varepsilon < |b|} A \moins A(\delta, \varepsilon) = Y_a\sqcup  Y_b, $$
where $A(\delta, \varepsilon):= \{ z \in K \; : \; \delta \le |z| \le \varepsilon \}$ and $Y_a, Y_b$ are the two ghost circles at the boundary of $A$. Our corollary now follows from Lemma \ref{ghost0}.
 \end{proof}
 
\begin{corollary}\label{mist4}{\rm (Arithmetic cohomology)}
 We have a canonical  isomorphism 
  $$
    H^{i}(\partial A,\Q_p(j))\simeq H^{i}(Y,\Q_p(j))^{\oplus 2}
    $$
    and   exact sequences
    \begin{align*}
 &   0\to H^{i}(\sg_K, H^{0}(\partial A_C,\Q_p(j)))\to H^{i}(\partial A,\Q_p(j))\to H^{i-1}(\sg_K, H^{1}(\partial A_C,\Q_p(j)))\to 0,\\
 &   0\to H^{i-1}(\sg_K, \tfrac{\so(\partial A_C)}{C^{\oplus2}}(j-1))\to H^{i-1}(\sg_K,H^{1}(\partial A,\Q_p(j)))
\to H^{i-1}(\sg_K, \Q_p^{\oplus2}(j-1))\to 0.
 \end{align*}
\end{corollary}
 \begin{proof} Since $\partial A$ is a disjoint union of two ghost circles, this follows immediately from Lemma \ref{mist1-ghost}.
\end{proof}
\begin{remark}
The arithmetic pro-\'etale cohomology $   H^{i}(Y,\Q_p(j))$ and $   H^{i}(\partial A,\Q_p(j))$ are  direct sums of nuclear Fr\'echet spaces and spaces of compact type (both over $\Q_p$) (see Proposition \ref{fanto8} for an explicit splitting). 
\end{remark}

\section{Pro-\'etale cohomology with compact support}\label{def-compact}
In this chapter we will study properties of compactly supported pro-\'etale cohomology of smooth dagger curves. 
\subsection{Compactly supported  cohomology}We start with  briefly  reviewing  the definition of pro-\'etale cohomology with compact support from \cite{AGN}. 
\subsubsection{Partially proper varieties} \label{morning2}
 Let $X$ be a smooth partially proper rigid analytic variety  over $K, C$. We define its $p$-adic pro-\'etale cohomology with compact support by: 
\begin{equation}
\label{defc}
 \R\Gamma_{c}(X, \Q_p(r)):=[ \R\Gamma(X, \Q_p(r)) \to \R\Gamma(\partial X, \Q_p(r))]\in \sd(\Q_{p}),  \quad r\ge 0, 
 \end{equation}
with
\[ \R\Gamma(\partial X, \Q_p(r)):= \colim_Z \R\Gamma(X \moins Z, \Q_p(r))\in \sd(\Q_{p}), \]
where the colimit is taken over  admissible quasi-compact opens  $Z\subset X$. 
From the definition, we get  a distinguished triangle
\[ \R\Gamma_{c}(X, \Q_p(r)) \to \R\Gamma(X, \Q_p(r)) \to \R\Gamma(\partial X, \Q_p(r)). \]
By \cite[Sec. 2.1]{AGN}, $\R\Gamma_{c}(X, \Q_p(r))$  is a cosheaf for the analytic topology on $X$.

If $X$ is a proper variety then the  definition  \eqref{defc} yields that we have the canonical  isomorphism 
\[ \R\Gamma_{c}(X, \Q_p(r)) \stackrel{\sim}{\to} \R\Gamma(X, \Q_p(r)) \]
and the cohomology groups  of $\R\Gamma(X,\Q_p(j))$ are classical: they are finite dimensional vector spaces over $\Q_p$ equipped with their canonical Hausdorff  topology. By Lemma \ref{morning1} so is the cohomology of the complex $\R\Gamma_{ \Box}(X, \Q_p(j))$ (which can be identified with the cohomology of $\R\Gamma(X,\Q_p(j))$ via the functor ${\rm CD}$).
\begin{remark}
The de Rham and Hyodo-Kato cohomologies  with compact support $\R\Gamma_{\dr,c}(X)$ and $\R\Gamma_{{\rm HK},c}(X)$ can be defined in an analogous way (see \cite{Ch90}).
\end{remark}

\subsubsection{Dagger affinoids}   Let $X$ be a smooth dagger affinoid over $K,C$ with a presentation $\{X_h\}$. We set
 $$
 \R\Gamma_{c}(X, \Q_p(r)):=\R\lim_h\R\Gamma_{c}(X^{0}_h,\Q_p(r))\in \sd(\Q_{p,\Box}),
 $$
 where $X^{0}_h$ denotes a  naive interior\footnote{See Remark \ref{naive1} for a discussion.} of $X_h$ adapted to the presentation $\{X_h\}$. This definition  is independent of the interiors chosen. Alternatively, we can set
 $$
  \R\Gamma_{c}(X, \Q_p(r)):=\R\Gamma_{\wh{X}}(X_h,\Q_p(r))\in \sd(\Q_{p,\Box}).
 $$
 This is independent of $h$.

\subsection{Hochschild-Serre spectral sequence}  
We record the Hochschild-Serre spectral sequence for pro-\'etale cohomology with compact support.
\begin{lemma}\label{Serre-c}
Let $X$ be a smooth partially proper variety over $K$. There is a  spectral sequence
\begin{equation}
\label{sseqc}
E^{a,b}_2 =H^a(\sg_K,H^b_{ c}(X_C,\Q_p(j)))\Rightarrow H^{a+b}_{c}(X,\Q_p(j)).
\end{equation}
\end{lemma} 

\begin{proof} 
  By definition (see \eqref{defc}), we have  
   \begin{align*}
   \label{passage1}
\rg_{c}(X,\Q_p) & = [\rg(X,\Q_p)\to \rg(\partial X,\Q_p)], \\
  \rg(\partial X,\Q_p) & :=\colim_{Z}\rg(X\moins Z,\Q_p),
\end{align*}
where the colimit is taken over admissible quasi-compact opens $Z$  in $X$. This yields
 natural quasi-isomorphisms
\begin{align*}
\rg_{c}(X,\Q_p) & =  [\R\Gamma(X,\Q_p)\to \R\Gamma(\partial X,\Q_p)]\\
& \simeq [\R\Gamma(\sg_K,\R\Gamma(X_C,\Q_p))\to \R\Gamma(\sg_K,\R\Gamma(\partial X_C,\Q_p))]\\
 & \simeq\rg(\sg_K, [\R\Gamma(X_C,\Q_p)\to \R\Gamma(\partial X_C,\Q_p)])\\
 & = \rg(\sg_K,\rg_{c}(X_C,\Q_p) ).
\end{align*}
The second quasi-isomorphism  needs a justification.  The quasi-isomorphism involving $X$ follows from Lemma  \ref{Serre}. For $\partial X$ we have quasi-isomorphisms
\begin{align*}
 \R\Gamma(\partial X,\Q_p) & =\colim_Z \R\Gamma(X\moins Z,\Q_p)\simeq   \colim_Z\rg(\sg_K,\R\Gamma((X\moins Z)_{C},\Q_p))\\
 &  \simeq \rg(\sg_K,\colim_Z \R\Gamma((X\moins Z)_{C},\Q_p))\simeq \rg(\sg_K, \R\Gamma(\partial X_C,\Q_p)).
\end{align*}
The second quasi-isomorphism follows from Lemma \ref{Serre}; the third one from the fact that $\Z[\sg_K^i]$ is a compact object in ${\rm CondAb}$.
\end{proof}

\subsection{Compactly supported cohomology of Stein  curves} We turn now to the study of compactly supported pro-\'etale cohomology of smooth Stein curves. 
\subsubsection{Geometric compactly supported cohomology}  We will briefly review here computations of geometric compactly supported cohomology from \cite[Sec. 8.2]{AGN}.

  Let $X$ be a geometrically connected smooth Stein  curve over $K$. Its geometric compactly supported pro-\'etale cohomology 
$\R\Gamma_{c}(X_C,\Q_p(j))\in\sd(\Q_{p,\Box})$ is studied in loc. cit. by using 
 a Galois equivariant  exact sequence ($i\geq 0$)
$$\xymatrix@R=3mm{
H^{i-1}{\rm HK}_{c}(X_C,i)\ar[r] &  H^{i-1}{\rm DR}_{c}(X_C,i)\ar[r]& H^i_{c}(X_C,\Q_p(i))\ar[r] & 
  H^{i}{\rm HK}_{c}(X_C,i)\ar[r]&
  H^{i}{\rm DR}_{c}(X_C,i),
}
$$
where we set 
\begin{align*}
 {\rm HK}_{c}(X_C,i) 
 {\rm DR}_{c}(X_C,i) &: = (\rg_{\dr,c}(X){\otimes}^{{\rm L}\Box}_{K}\B^+_{\dr})/F^i\\ 
& \simeq \big(H^1_{c}(X,\so_X){\otimes}^{\Box}_{K}(\B^+_{\dr}/F^i)\to H^1_{c}(X,\Omega^1_X){\otimes}^{\Box}_{K}(\B^+_{\dr}/F^{i-1})\big)[-1].
\end{align*}
This sequence is obtained from  a comparison theorem between pro-\'etale cohomology and syntomic cohomology. It yields the following facts: 
\begin{lemma} We have 
\begin{enumerate}
\item  vanishings: $H^i_{c}(X_C,\Q_p)=0$ for $i\neq 1,2$.
\item an isomorphism: 
\begin{equation}
\label{kosc1}
H^1_{c}(X_C,\Q_p(1)) \xrightarrow{\sim} (H^1_{{\rm HK},c}(X_C){\otimes}^{\Box}_{\breve{F}} \wh{\B}^+_{\st})^{N=0, \varphi=1}.
\end{equation}
\item  an  exact sequence: 
\begin{equation}
\label{vend0}
  \xymatrix@R=3mm{
H^1{\rm HK}_{c}(X_C,2) \ar[r] &  H^{1}{\rm DR}_{c}(X_C,2)\ar[r] & H^2_{c}(X_C,\Q_p(2))\ar[r] & \Q_p(1)\ar[r] & 0.
}
\end{equation}
\end{enumerate}
\end{lemma}
 The above map ${\rm Tr}_{X_C}:H^2_{c}(X_C,\Q_p(2))\to  \Q_p(1)$ is {\it the geometric trace map}.
 \subsubsection{Arithmetic compactly supported  cohomology}\label{van11} 
  Let $X$ be a geometrically connected smooth Stein  curve over $K$. We will now look at  its arithmetic compactly supported pro-\'etale cohomology complex
$\R\Gamma_{c}(X,\Q_p(j))\in\sd(\Q_{p,\Box}).$
\begin{theorem}\label{final1}
The cohomology of $\R\Gamma_{c}(X,\Q_p(j))$ is  of compact type. 
\end{theorem}
\begin{proof}

    The only  nontrivial geometric cohomology groups are in degrees 1,2 hence, from the spectral sequence \eqref{sseqc}, we get that $H^0_{c}(X,\Q_p(j))=0$ and we have the long exact sequence 
    \begin{align*}
0   \to&  H^0(\sg_K,H^1_{c}(X_C,\Q_p(j)))\to H^1_{c}(X,\Q_p(j))\to H^{-1}(\sg_K,H^2_{c}(X_C,\Q_p(j)))\\
   \to & H^1(\sg_K, H^1_{c}(X_C,\Q_p(j)))\to H^2_{c}(X,\Q_p(j))\to  H^0(\sg_K,H^2_{c}(X_C,\Q_p(j)))\\
      \stackrel{d_2}{\to} & H^2(\sg_K, H^1_{c}(X_C,\Q_p(j)))\to H^3_{c}(X,\Q_p(j))\to  H^1(\sg_K,H^2_{c}(X_C,\Q_p(j)))\\
            \to & H^3(\sg_K, H^1_{c}(X_C,\Q_p(j)))\to H^4_{c}(X,\Q_p(j))\to  H^2(\sg_K,H^2_{c}(X_C,\Q_p(j)))\to 0
 \end{align*}
 It follows that
 \begin{align*}
 & H^0_{c}(X,\Q_p(j)))=0,\\
  & H^1_{c}(X,\Q_p(j)) \stackrel{\sim}{\to }H^0(\sg_K,H^1_{c}(X_C,\Q_p(j))),\\
  &  H^4_{c}(X,\Q_p(j)) \stackrel{\sim}{\to}H^2(\sg_K, H^2_{c}(X_C,\Q_p(j))).
 \end{align*}
 Hence, by \eqref{kosc1}, $$ H^1_{c}(X,\Q_p(j))\simeq H^0(\sg_K, {\rm HK}^1_{c}(X_C,1) (j-1)),
 $$ 
 where we set ${\rm HK}^1_{c}(X_C,1):=H^1{\rm HK}_{c}(X_C,1)$.  It follows that $ H^1_{c}(X,\Q_p(j))$  is a colimit of finite rank $\Q_p$-vector spaces hence of compact type. 
By \cite[Sec. 8.3]{AGN}, we have  that $$
H^4_{c}(X,\Q_p(j)) \simeq\begin{cases}
\Q_p & \mbox { if } j=2,\\
0 & \mbox{otherwise.}
\end{cases}
$$ 
  
 It remains to treat   $H^i_{c}(X,\Q_p(j))$, for $i=2,3$. We have a long exact sequence
 \begin{align}\label{zlosc1}
0 &  \to H^1(\sg_K, H^1_{c}(X_C,\Q_p(j)))\to H^2_{c}(X,\Q_p(j))\to H^0(\sg_K,H^2_{c}(X_C,\Q_p(j)))\\
  & \stackrel{d_2}{ \to } H^2(\sg_K, H^1_{c}(X_C,\Q_p(j)))\to H^3_{c}(X,\Q_p(j)) \to H^1(\sg_K,H^2_{c}(X_C,\Q_p(j)))\to 0\notag
  \end{align}
 ($\bullet$) {\em Case $j\neq 2$.} We claim that in this case  $H^2(\sg_K, H^1_{c}(X_C,\Q_p(j)))=0$. Indeed, we have 
  $$H^2(\sg_K, H^1_{c}(X_C,\Q_p(j)))\simeq H^2(\sg_K,{\rm HK}^1_{c}(X_C,1)(j-1)).$$
 Now the slopes of  Frobenius on $ H^1_{{\rm HK},c} (X_C)$ are between $0$ and $1$,
which implies that, in the case this group  is of finite rank,  ${\rm HK}^1_{c}(X_C,1)$ is an extension of 
an unramified finite dimensional $\Q_p$-representation $V$ of $\G_K$
by the $C$-points $W$ of a connected BC. Since we have
$H^2(\G_K,W(j-1))=0$, for any $j$, and $H^2(\G_K,V(j-1))=0$, for $j\neq 2$, this implies $H^2(\sg_K,{\rm HK}^1_{c}(X_C,1)(j-1))=0$, as wanted.  The general case is obtained by writing $H^1_{{\rm HK},c} (X_C)$ as a colimit of groups of  finite rank.

   (i) {\em First sequence.} Let us now look at the first exact sequence from \eqref{zlosc1}.
\begin{equation}
\label{firstseq}
  0   \to H^1(\sg_K, H^1_{c}(X_C,\Q_p(j)))\to H^2_{c}(X,\Q_p(j))\to H^0(\sg_K,H^2_{c}(X_C,\Q_p(j)))\to 0 
\end{equation}
Let $\{X_n\}_{n\in\N}$ be a strictly increasing open covering of $X$ by adapted naive interiors of dagger affinoids. We have an isomorphism
$$
\colim_nH^i_{c}(X_n,\Q_p(j))\stackrel{\sim}{\to} H^i_{c}(X,\Q_p(j)),\quad i,j\in\Z.
$$
We have  analogs of  sequence \eqref{firstseq} for $X_n$'s.  

It suffices to prove  the following result.
\begin{lemma}\label{glowa1}
 The transition maps $f_{2,n}: H^2_{c}(X_n,\Q_p(j))\to H^2_{c}(X_{n+1},\Q_p(j))$ are compact maps between spaces of compact type.
 \end{lemma}
 \begin{proof} Consider the canonical transition map
 \begin{align*}
& f_{1,n}: H^1(\sg_K, H^1_{c}(X_{n,C},\Q_p(j)))  \to H^1(\sg_K, H^1_{c}(X_{n+1,C},\Q_p(j))).
\end{align*}
We have $$H^1(\sg_K, H^1_{c}(X_{n,C},\Q_p(j)))\simeq H^1(\sg_K, {\rm HK}^1_{c}(X_{n,C},1)(j-1)).
$$  Hence it is of finite rank over $\Q_p$. This implies, by 
  Lemma \ref{ext10}, Lemma \ref{ext10.1}, and the exact sequence \eqref{firstseq}, that  it suffices to show that
the   canonical transition map
\begin{align*}
& f_{3,n}:  H^0(\sg_K,H^2_{c}(X_{n,C},\Q_p(j)))  \to H^0(\sg_K,H^2_{c}(X_{n+1,C},\Q_p(j)))
\end{align*}
is a compact map of spaces of compact type.

 By \eqref{vend0}, we have the exact sequence, for $s=n,n+1$,
    $$
0\to H^0(\sg_K,  (\coker g(X_s))(j-2)){\to} H^0(\sg_K,H^2_{c}(X_{s,C},\Q_p(2))(j-2))\to H^0(\sg_K,\Q_p(j-1)) \to, 
   $$
   where $g(X_s): H^1{\rm HK}_{c}(X_{s,C},2)\to H^1{\rm DR}_{c}(X_{s,C},2)$ is the canonical map. But
   $$
   \coker g(X_s)\simeq \coker((H^1_{{\rm HK},c}(X_{s,C}){\otimes}_{\breve{F}}^{ \Box}t\wh{\B}^+_{\st})^{N=0,\phi=p^2}\to H^1_{\dr,c}(X_{s,C}){\otimes}_{K}^{ \Box}C(1)\hookrightarrow H^1_{c}(X_{s},\so_{X_s}){\otimes}_{K}^{ \Box}C(1)).
   $$
Now,  using  generalized Tate's formulas  (see \eqref{newton15}), we easily see that  $ H^0(\sg_K, (\coker g(X_s))(j-2))$  is of finite  rank over $\Q_p$ unless $j=1$. Hence, for $j\neq 1$, the map $f_{3,n}$ is   a map between finite rank  vector spaces over $\Q_p$; hence it is compact.

  Assume now that $j=1$.   It suffices to show that the transition maps between the spaces $H^0(\sg_K,(\coker g(X_n))(-1))$ are compact and the spaces themselves are of compact type. From the definition of the map $g(X_n)$, we get an exact sequence
    \begin{equation}
  \label{orange1}
0\to A_n\to  H^1_{c}(X_n,\so_{X_n}){\otimes}_{K}^{ \Box}C\to  (\coker g(X_n))(-1)\to 0,
\end{equation}
where $A_n$ is an almost $C$-representation.  Applying Galois cohomology to this sequence, we get the exact sequence
$$0\to H^0(\sg_K,A_n) \to H^1_{c}(X_n,\so_{X_n})\to H^0(\sg_K,(\coker g(X_n))(-1)) \to H^1(\sg_K,A_n).
$$
Since, by generalized Tate's theorem (see \eqref{newton15}), $H^0(\sg_K,A_n) $ and $H^1(\sg_K,A_n)$ are of finite rank over $\Q_p$, it suffices to show that the canonical map $f: H^1_{c}(X_n,\so_{X_n})\to H^1_{c}(X_{n+1},\so_{X_{n+1}})$ is compact and the cokernel of the map $g: H^0(\sg_K,A_n) \to H^1_{c}(X_n,\so_{X_n})$ is of compact type. 
For that, choose rigid analytic affinoids $Y_1$, $Y_1$ such that
$$
X_{n}\subset Y_1\Subset Y_2\subset X_{n+1}.
$$ Then we have maps
\begin{equation}
\label{factorization}
\Omega(X_{n+1})\to \Omega(Y_2)\stackrel{\tilde{f}}{\longrightarrow}\Omega(Y_1)\to \Omega(X_{n}).
\end{equation}
The spaces $ \Omega(Y_i)$, $i=1,2$, are Banach and the map $\tilde{f}$ is compact. Taking the strong duals of the terms of  \eqref{factorization} and using Serre's duality for the end terms we get maps
$$
H^1_{c}(X_n,\so_{X_n}){\to} \Omega(Y_1)^*\stackrel{\tilde{f}^*}{\longrightarrow} \Omega(Y_2)^*\to H^1_{c}(X_{n+1},\so_{X_{n+1}})
$$
factorizing the map $f$. By \cite[Lemma 16.4]{Schn}, the map $\tilde{f}^*$ is compact; by \cite[Remark 16.7]{Schn}, so is the map $f$, as wanted. 

The statement about the cokernel of the map $g$ follows from the fact that the space $H^1_{c}(X_n,\so_{X_n})$ is of compact type and the 
image of $g$ is finite dimensional.
  \end{proof}

   (ii) {\em Second sequence.}  We pass now to the second exact  sequence from \eqref{zlosc1} which became the isomorphism
 \begin{align*}
 H^3_{c}(X,\Q_p(j)) \stackrel{\sim}{\to} H^1(\sg_K,H^2_{c}(X_C,\Q_p(j))).
  \end{align*}
  We cover $X$ with a strictly increasing system $\{X_n\}_{n\in\N}$ of adapted naive interiors of dagger affinoids. 
It suffices to show that the transition maps 
\begin{equation*}
f_n:\quad H^1(\sg_K,H^2_{c}(X_{n,C},\Q_p(j)))\to
 H^1(\sg_K,H^2_{c}(X_{{n+1},C},\Q_p(j)))
 \end{equation*} are compact maps between spaces of compact type.

     By \eqref{vend0}, we have the exact sequence 
    $$
H^0(\sg_K,\Q_p(j-1))\to  H^1(\sg_K,  (\coker g(X_s))(j-2)){\to} H^1(\sg_K,H^2_{c}(X_{s,C},\Q_p(2))(j-2))\to H^1(\sg_K,\Q_p(j-1)).
 $$
  Using  the generalized Tate's formulas  (see \eqref{newton15}), we easily see that  $ H^1(\sg_K, (\coker g(X_s))(j-2))$  is of finite  rank over $\Q_p$ unless $j=1$. Hence, for $j\neq 1$, the map $f_n$ is compact as  a map between finite rank vector spaces over $\Q_p$. 
 
    Assume now that $j=1$.   It suffices to show  that the transition maps between the spaces $H^1(\sg_K,(\coker g(X_n))(-1))$ are maps of compact type between spaces of compact type. By  generalized  Tate's theorem (see \eqref{newton15}), from the exact sequence \eqref{orange1}, we get the exact sequence
$$\to H^1(\sg_K,A_n) \to H^1_{c}(X_n,\so_{X_n})\to H^1(\sg_K,(\coker g(X_n))(-1)) \to H^2(\sg_K,A_n)\to .
$$
Since $H^1(\sg_K,A_n) $ and $H^2(\sg_K,A_n)$ are of finite rank over $\Q_p$, it suffices to show  that the canonical map $H^1_{c}(X_n,\so_{X_n})\to H^1_{c}(X_{n+1},\so_{X_{n+1}})$ is a compact  map of spaces of compact type but  this we have done above. 
  
  ($\bullet$) {\em Case $j=2$.}    We cover $X$ with a strictly increasing system $\{X_n\}_{n\in\N}$ of adapted naive interiors of dagger affinoids.   It suffices to show  that all the terms in the exact sequence \eqref{zlosc1} on level $X_n$ are of finite rank  over $\Q_p$ (this is true more generally for $j\neq 1$).
   But, by \eqref{vend0}, since $H^0(\sg_K,\Q_p(1))=0$, we have 
    $$
 H^0(\sg_K,  \coker g(X_n))\stackrel{\sim}{\to} H^0(\sg_K,H^2_{c}(X_{n,C},\Q_p(2))).
   $$
And we have shown above that this is of finite rank over $\Q_p$. 
\end{proof}

\subsection{Arithmetic cohomology of dagger affinoids} 
Let $X$ be a smooth geometrically connected dagger affinoid over $K$. We will now study its arithmetic pro-\'etale cohomology. They key tool is the (studied above)   arithmetic pro-\'etale cohomology of smooth Stein curves.
\begin{proposition}\label{final3.1}
The cohomology of $\R\Gamma_{c}(X,\Q_p(j)), j\in\Z,$ is nuclear Fr\'echet. Moreover, 
 if $\{X_h\}$ is  the dagger presentation of the dagger structure on $X$, then we have an isomorphism
$$
H^i_{c}(X,\Q_p(j))\stackrel{\sim}{\to} \lim_hH^i_{c}(X^0_h,\Q_p(j)),\quad i,j\in\Z,
$$
where $X_h^0$ is a naive interior of $X_h$ adapted to $\{X_h\}$.
\end{proposition}
\begin{proof} Let $\{X_h\}$ be the dagger presentation of the dagger structure on $X$. Denote by $X_h^0$  a naive interior of $X_h$ adapted to $\{X_h\}$. 
Note that $X^0_h$ is a smooth and Stein rigid analytic variety over $K$.  
From Section \ref{van11}, we have 
 \begin{align*}
 & H^0_{c}(X^0_h,\Q_p(j)))=0,\\
  & H^1_{c}(X^0_h,\Q_p(j)) \stackrel{\sim}{\to }H^0(\sg_K,H^1_{c}(X^0_{h,C},\Q_p(j)))\simeq H^0(\sg_K, {\rm HK}^1_{c}(X^0_{h,C},1) (j-1)),\\
  &  H^4_{c}(X^0_h,\Q_p(j)) \stackrel{\sim}{\to}H^2(\sg_K, H^2_{c}(X^0_{h,C},\Q_p(j)))\simeq\begin{cases}
\Q_p & \mbox { if } j=2,\\
0 & \mbox{otherwise.}
\end{cases}
 \end{align*}
These are  finite rank $\Q_p$-vector spaces. Moreover, by Lemma \ref{glowa1}, the transition maps
$$
H^i_{c}(X^0_h,\Q_p(j))\to H^i_{c}(X^0_{h+1},\Q_p(j)),\quad i=2,3,
$$
  are compact maps between spaces of compact type. 
  
  By \cite[discussion after Prop. 16.5]{Schn}, since  the pro-system 
  $\{H^i_{c}(X^0_h,\Q_p(j))\}_{h\in\N}$, $i\geq 0$,  is compact,  it is equivalent to a pro-system of Banach spaces with dense transition maps. Hence it is Mittag-Leffler by Section \ref{solid-frechet} and we have  $\R^1\lim_hH^i_{c}(X^0_h,\Q_p(j))=0$, $i\geq 0$.     It follows that
   $$
   H^i_{c}(X,\Q_p(j))\stackrel{\sim}{\to} \lim_hH^i_{c}(X^0_h,\Q_p(j))
   $$
 and  the cohomology of $\R\Gamma_{c}(X,\Q_p(j)), j\in\Z,$ is nuclear Fr\'echet,
  as wanted.
\end{proof}
 
 \subsection{Examples}\label{compact-11}
 
We will now determine the compactly supported $p$-adic pro-\'etale cohomology groups of an open disc and an open annulus.

 \subsubsection{Open disc} We start with the cohomology  $H^{i}_{c}(D, \Q_p(j))$ of an  open disc $D$ over $K$.   
It immediately follows from Lemma \ref{ball0} and Lemma \ref{ghost0}, and the definition of compactly supported cohomology that we have: 
\begin{lemma}{\rm (Geometric cohomology)}
\label{PotB0} Let $i\in\N, j\in\Z$. Then:
\begin{enumerate}
\item The canonical maps $H^{i}(D_C, \Q_p(j))\to H^{i}(Y_C, \Q_p(j))$ are   injective; hence we have a   $\sg_K$-equivariant  isomorphism
$$
H^{i}_{c}(D_C, \Q_p(j))\stackrel{\sim}{\leftarrow}  H^{i-1}(Y_C, \Q_p(j))/ H^{i-1}(D_C, \Q_p(j)).
$$
\item 
The groups $H^{i}_{c}(D_C, \Q_p(j))$ are $0$ unless $i=2$, in which case there is a   $\sg_K$-equivariant  (split)  exact sequence: 
\[  0\to (\so(Y_C)/\so(D_C))(j-1)\to H_{c}^2(D_C,\Q_p(j))\to \Q_p(j-1)\to 0. \]
\end{enumerate}
\end{lemma} 

  Now we pass to the arithmetic cohomology. The canonical maps
\begin{equation}
\label{inj21}
H^i(D,\Q_p(j))\to  H^i(Y, \Q_p(j))
\end{equation}
 are   injective; this follows  from the  descriptions of 
these groups in Lemma \ref{mist1} and Lemma \ref{mist1-ghost}. 
Hence we have an isomorphism
\begin{equation}
\label{inj22}H^i_{c}(D,\Q_p(j))\stackrel{\sim}{\leftarrow}H^{i-1}(Y,\Q_p(j))/H^{i-1}(D,\Q_p(j)).
\end{equation}

\begin{lemma}{\rm (Arithmetic Cohomology)}
\label{PotB1}
Let $i\in\N, j\in\Z$. There is an   exact sequence: 
\[ 0\to  H^{i-2}(\sg_K, \tfrac{\so(Y_C)}{\so(D_C)}(j-1))\to H^{i}_{c}(D,\Q_p(j)) \to H^{i-2}(\sg_K,  \Q_p(j-1))\to0 \]
\end{lemma}

\begin{proof} 
Consider  the following commutative diagram:
 $$
 \xymatrix@R=6mm{ & 0\ar[d]  & 0\ar[d] & \\
  0\ar[r] & H^{i-1}(\sg_K,\Q_p(j))\ar[d]^{\wr} \ar[r] & H^{i-1}( D,\Q_p(j))\ar[d] \ar[r] & H^{i-2}(\sg_K, H^1(D_C,\Q_p(j)))\ar[d]^{f_1} \ar[r] & 0\\
0\ar[r] & H^{i-1}(\sg_K,\Q_p(j))\ar[r]  & H^{i-1}(Y,\Q_p(j))\ar[r] \ar[d]& H^{i-2}(\sg_K, H^1(Y_C,\Q_p(j)))\ar[r] \ar[d]^{f_2} & 0\\
 & & H^{i}_c(D,\Q_p(j))\ar[r] \ar[d] & H^{i-2}(\sg_K, H^2_{c}(D_C,\Q_p(j)))\\
  & & 0 & 
  }
 $$
The first two rows are  exact by Lemma \ref{mist1}  and Lemma \ref{mist1-ghost}, respectively.  The middle column is  exact by formula \eqref{inj22}. We claim that  the third column is  exact, the map $f_1$ is injective, and the map $f_2$ is surjective. Indeed, by Lemma \ref{PotB0},  Lemma \ref{ball0},  and  Lemma  \ref{ghost0}, it suffices to show that  the canonical map
 $$
  H^{i-2}(\sg_K,  \tfrac{\so(D_C)}{C}(j-1))\to H^{i-2}(\sg_K,  \tfrac{\so(Y_C)}{C}(j-1))
 $$
 is  injective. But,
  by the generalized Tate's formula \eqref{newton15}, this map is either a map between trivial objects or  is isomorphic  (for $j=1$ and $i=2,3$) to the canonical map $
 \so(D)/K\to \so(Y)/K.
 $
 And that map is clearly   injective. 
 
  The above discussion shows that we have 
  an isomorphism: 
\begin{equation}
\label{PotB00}
H_{c}^{i}(D,\Q_p(j))  \xrightarrow{\sim} H^{i-2}(\sg_K, H_{c}^2(D_C,\Q_p(j))), \quad \text{for all $i\in\N, j \in \Z$.} 
\end{equation}
Now we evoke 
 Lemma \ref{PotB0}.
\end{proof}


   \subsubsection{Open annulus} Let $A$ be an open annulus over $K$. 
We start with the geometric cohomology. 
\begin{lemma}{\rm (Geometric cohomology)} \label{deszcz5}
\begin{enumerate}
\item Let $i\in\N, j\in\Z$. The canonical maps
$H^i(A_C,\Q_p(j))\to H^i(\partial A_C, \Q_p(j))$ are    injective.  Hence we have $\sg_K$-equivariant  isomorphisms
\begin{align*}
H^i_{c}(A_C,\Q_p(j))& \stackrel{\sim}{\leftarrow}H^{i-1}(\partial A_C,\Q_p(j))/H^{i-1}(A_C,\Q_p(j)),\\
H^i_{c}(A_C,\Q_p(j))& =0,\quad \mbox{if } i >2.
\end{align*}
\item We have 
a $\sg_K$-equivariant  isomorphism and a (split)  exact sequence:
\begin{align}
\label{Vvk1}
&H^1_{c}(A_C,\Q_p(j))\simeq \Q_p(j),\\
0\to\tfrac{\O(\partial A_C)}{\O(A_C)\oplus C}(j-1)\to &H^2_{c}(A_C,\Q_p(j))\to\Q_p(j-1)\to 0. \notag
\end{align}
\end{enumerate}
\end{lemma}
\begin{proof}
This follows from the  description of 
the groups involved  in  \eqref{annulus1}, \eqref{annulus1.1} and Corollary \ref{Vvk}.
\end{proof}

 Now we pass to  arithmetic cohomology. The canonical  maps
$$H^i(A,\Q_p(j))\to H^i(\partial A, \Q_p(j))$$ are    injective. This follows  from the  description of 
the groups involved in Lemma \ref{mist3} and Corollary \ref{mist4}.
Hence we have an isomorphism
\begin{equation}
\label{form}
H^i_{c}(A,\Q_p(j))\stackrel{\sim}{\leftarrow} H^{i-1}(\partial A,\Q_p(j))/H^{i-1}(A,\Q_p(j)).
\end{equation}
\begin{lemma}\label{mist5}{\rm (Arithmetic Cohomology)}
Let $i \in\N$, $j\in\Z$. There is an    exact sequence
   \begin{align*}
   0\to H^{i-1}(\sg_K, H^{1}_{c}(A_C,\Q_p(j)))\to H^{i}_{c}(A,\Q_p(j))\to H^{i-2}(\sg_K, H^{2}_{c}(A_C,\Q_p(j)))\to 0.
 \end{align*}
\end{lemma}
\begin{proof}  
Consider  the following commutative diagram:
 $$
 \xymatrix@R=6mm{ & 0\ar[d]  & 0\ar[d] & \\
  0\ar[r] & H^{i-1}(\sg_K,\Q_p(j))\ar[d] \ar[r] & H^{i-1}( A,\Q_p(j))\ar[d] \ar[r] & H^{i-2}(\sg_K, H^1(A_C,\Q_p(j)))\ar[d]^{f_1} \ar[r] & 0\\
0\ar[r] & H^{i-1}(\sg_K,\Q_p(j))\ar[r] \ar[d] & H^{i-1}(\partial A,\Q_p(j))\ar[r] \ar[d]& H^{i-2}(\sg_K, H^1(\partial A_C,\Q_p(j)))\ar[r] \ar[d]^{f_2} & 0\\
& H^{i-1}(\sg_K,H^1(A_C,\Q_p(j)))\ar[r] \ar[d] & H^{i}_{c}(A,\Q_p(j))\ar[r] \ar[d] & H^{i-2}(\sg_K, H^2_{c}(A_C,\Q_p(j)))\\
  &0  & 0 & 
  }
 $$
The first two rows are  exact by Lemma \ref{mist3}  and Corollary \ref{mist4}, respectively.  The middle column is  exact by formula \eqref{form}. The first column is  exact by the same formula and the fact that the canonical map $H^0(A_C,\Q_p(j))\to H^0(\partial A_C,\Q_p(j))$ is $\sg_K$-equivariantly split. 

  It suffices now to show that  the third column is  exact, the map $f_1$ is injective, and the map $f_2$ is surjective. To see that, use Lemma \ref{deszcz5} and note that,
 by Lemma \ref{annulus1.0} and  Corollary \ref{Vvk},  the map $f_1$  can be rewritten as the canonical map
 $$
  H^{i-2}(\sg_K,  (\tfrac{\so(A_C)}{C}\oplus \Q_p)(j-1))\to H^{i-2}(\sg_K,  (\tfrac{\so(Y_C)}{C}\oplus \Q_p)^{\oplus 2}(j-1)),
 $$
 which, by the generalized Tate's formula \eqref{newton15}, is either a map between trivial objects or  is isomorphic  (for $j=1$ and $i=2,3$) to the canonical map
 $$
 \tfrac{\so(A)}{K}\oplus \Q_p\to \big(\tfrac{\so(Y)}{K}\oplus \Q_p\big)^{\oplus 2}.
 $$
 And that map is clearly   injective. We note that this also proves that the map $f_2$ is  surjective, as wanted. 
\end{proof}

\section{Trace maps and pairings} In this chapter we will discuss pro-\'etale and coherent trace maps and pairings.

\subsection{Pro-\'etale trace maps.} \label{traces} We start with pro-\'etale trace maps. 
 \begin{proposition}\label{traces11}
 Let $X$ be a smooth Stein variety, a smooth dagger affinoid, or a proper variety  over $K$ of dimension $1$. Assume that it is geometrically irreducible. Then
 \begin{enumerate}
 \item There exists a natural geometric trace map
 \begin{equation}\label{trg}
 {\rm Tr}_{X_C}: H^{2}_{c}(X_C,\Q_p(1))\to\Q_p.
 \end{equation}
 It is an isomorphism if $X$ is proper.
 \item There exists a natural arithmetic trace map
 \begin{equation}
\label{tra}
 {\rm Tr}_{X}: H^{4}_{c}(X,\Q_p(2))\to\Q_p.
 \end{equation}
 It is an isomorphism. 
 \item The above trace maps are functorial for open immersions and compatible with the Hyodo-Kato and de Rham trace maps.
 \end{enumerate} 
 \end{proposition}
 \begin{proof}
In the case $X$ is partially proper, the arithmetic trace map is defined using the geometric trace map and the Galois cohomology trace map
$$
 {\rm Tr}_{X}: H^{4}_{c}(X,\Q_p(2))\simeq H^2(\sg_K,H^{2}_{c}(X,\Q_p(2)))\lomapr{ {\rm Tr}_{X_C}(1)} H^2(\sg_K,\Q_p(1))\lomapr{{\rm Tr}_K}\Q_p.
$$
It was shown to be an isomorphism in  \cite[Sec. 8.3]{AGN}.  

In the case $X$ is a dagger affinoid, the arithmetic trace map is defined as the composition
$$
{\rm Tr}_X:\quad H^4_{c}(X,\Q_p(2))\stackrel{\sim}{\to} \lim_h H^4_{c}(X^0_h,\Q_p(2)) \verylomapr{\lim_h{\rm Tr}_{X^0_h}}\Q_p.
$$
Here the first isomorphism follows from the fact that $\R^1 \lim_h H^3_{c}(X^0_h,\Q_p(2))=0$, which was shown in the proof of Proposition \ref{final3.1}.
\end{proof}

  We will also need a derived version of the  arithmetic  trace map:
$${\rm Tr}_X:\quad \rg_{c}(X,\Q_p(2))[4]\to \Q_p.
$$
We define it as  the composition 
$$
\rg_{c}(X,\Q_p(2))[4]\stackrel{\sim}{\leftarrow}(\tau_{\leq 4}\rg_{c}(X,\Q_p(2)))[4]\stackrel{\can}{\longrightarrow}H^4_{c}(X,\Q_p(2))\lomapr{ {\rm Tr}_{X}}\Q_p.
$$
   \subsection{Pro-\'etale pairings for  partially proper varieties} Let $X$ be a partially proper rigid anaytic variety over $K,C$. Cup product on pro-\'etale cohomology induces  pairings
   \begin{align*}
   \R\Gamma(X, \Q_p) \otimes^{ {\rm L}\Box}_{\Q_{p}} \R\Gamma(\partial X, \Q_p)\to  \R\Gamma(\partial X, \Q_p)
   \end{align*}
   as the composition
    \begin{align*}
   \R\Gamma(X, \Q_p) & \otimes^{{\rm L}\Box}_{\Q_{p}} \R\Gamma(\partial X, \Q_p)=
     \R\Gamma(X, \Q_p) \otimes^{ {\rm L}\Box}_{\Q_{p}} (\colim_Z\R\Gamma( X\moins Z, \Q_p))\\
     &\stackrel{\sim}{\leftarrow}
   \colim_Z (   \R\Gamma(X, \Q_p) \otimes^{{\rm L}\Box}_{\Q_{p}} \R\Gamma( X\moins Z, \Q_p))\verylomapr{\colim_Z\cup}\colim_Z\R\Gamma( X\moins Z, \Q_p)\\
 & =\R\Gamma(\partial X, \Q_p).
   \end{align*}
  These pairings are  compatible with the pairings
     \begin{align*}
   \R\Gamma(X, \Q_p) \otimes^{{\rm L}\Box}_{\Q_{p}} \R\Gamma(X, \Q_p)\to  \R\Gamma(X, \Q_p).
   \end{align*}
They yield (Galois equivariant over  $C$)  pairings
\begin{align}
\label{samolot1}
 \R\Gamma(X, \Q_p) \otimes^{{\rm L}\Box}_{\Q_{p}} \R\Gamma_{ c}(X, \Q_p)\to \R\Gamma_{ c}(X, \Q_p),
 \end{align}
which are compatible with the passage from $K$ to $C$.

    Let now $X$ be partially proper and as in  Proposition \ref{traces11}. The (twisted) pairing \eqref{samolot1} composed with the trace map \eqref{trg} yields a geometric pairing: 
 \begin{align*}
  \R\Gamma(X_C,\Q_p(j))\otimes^{{\rm L} \Box}_{\Q_{p}}  \R\Gamma_{c}(X_C,\Q_p(1-j))[2]{\to} \Q_p.
 \end{align*}
Similarly, using the trace map \eqref{tra}, we obtain an  arithmetic pairing
   \begin{align*}
   \R\Gamma(X,\Q_p(j))\otimes^{{\rm L}\Box}_{\Q_{p}}  \R\Gamma_{c}(X,\Q_p(2-j))[4]{\to}  \Q_p.
\end{align*}
 These pairings are compatible and are also compatible  with pro-\'etale pairings on $ \R\Gamma(X_C,\Q_p(j))$ and $ \R\Gamma(X,\Q_p(j))$, respectively. 
\subsection{Pro-\'etale pairings for dagger affinoids} Let $X$ be a smooth dagger affinoid  over $K,C$. Cup product on pro-\'etale cohomology 
   \begin{align}
   \label{samolot2}
   \R\Gamma(X, \Q_p) \otimes^{{\rm L}\Box}_{\Q_{p}} \R\Gamma_{c}(X, \Q_p)\to  \R\Gamma_{c}(X, \Q_p)  
 \end{align}
 is defined by the composition
 \begin{align*}
    \R\Gamma(X, \Q_p) \otimes^{{\rm L}\Box}_{\Q_{p}} \R\Gamma_{c}(X, \Q_p) &= (\colim_h \R\Gamma(X_h, \Q_p)) 
  \otimes^{{\rm L}\Box}_{\Q_{p}} (\colim_n\R\Gamma_{\wh{X}}(X_n, \Q_p))\\
& = \colim_{h,n} (\R\Gamma(X_h, \Q_p) \otimes^{{\rm L}\Box}_{\Q_{p}} \R\Gamma_{\wh{X}}(X_n, \Q_p))\\
& \stackrel{\sim}{\to} \colim_{h} (\R\Gamma(X_h, \Q_p) \otimes^{{\rm L}\Box}_{\Q_{p}} \R\Gamma_{\wh{X}}(X_h, \Q_p))\\
&  \verylomapr{\colim_h\cup} \colim_h  \R\Gamma_{\wh{X}}(X_h, \Q_p)= \R\Gamma_{c}(X, \Q_p).
\end{align*}
Here $\{X_h\}_{h\in\N}$ is the dagger presentation of $X$.

  These pairings  are  Galois equivariant over  $C$  
and  compatible with the passage from $K$ to $C$.

     The (twisted) pairing \eqref{samolot2} composed with the trace map \eqref{trg} yields a geometric pairing: 
 \begin{align*}
  \R\Gamma(X_C,\Q_p(j))\otimes^{{\rm L} \Box}_{\Q_{p}}  \R\Gamma_{c}(X_C,\Q_p(1-j))[2]{\to} \Q_p.
 \end{align*}
Similarly, using the trace map \eqref{tra}, we obtain an  arithmetic pairing
   \begin{align*}
   \R\Gamma(X,\Q_p(j))\otimes^{{\rm L}\Box}_{\Q_{p}}  \R\Gamma_{c}(X,\Q_p(2-j))[4]{\to}  \Q_p.
\end{align*}
These pairings are compatible and are also compatible 
with pro-\'etale pairings on $ \R\Gamma(X_C,\Q_p(j))$ and $ \R\Gamma(X,\Q_p(j))$, respectively.

\subsection{Coherent pairings} 
We will list now the coherent pairings that we will use. 

\subsubsection{Ghost circle} 
Let $D$ be an open disc over $K$; let $Y:=\partial D$ be the boundary of $D$, a ghost circle.  Let $Y_C:=Y\times C$.
The ring  $\O(Y)$ (resp.~$\O(Y_C)$) is the Robba ring  with coefficients in $K$ (resp.~$C$). Topologically, $\so(Y)$ is a direct sum of a nuclear Fr\'echet space and a space of compact type (both over $\Q_p$). 

 The map \begin{equation}\label{nie10}
(f,g)\mapsto {\rm Tr}_{K}{\rm res}(f d g )
\end{equation}
induces a perfect pairing\footnote{We used here the fact that, for nuclear Fr\'echet spaces and spaces of compact type, the passage between locally convex topological vector spaces and solid vector spaces works well on the level of tensor products and $\Hom$s, see Section \ref{solid-product}.} 
\begin{equation} \label{fanto11} \cup: \O(Y)/K\otimes^{\Box}_{\Q_p}\O(Y)/K\to\Q_p \end{equation}
between the solid $\Q_p$-vector spaces $\O(Y)/K$ and $\O(Y)/K$.

 Similarly, for $L=K,C$, the map 
 \begin{equation}
 \label{nie11}
 (f,g)\mapsto {\rm res}(f d g )
 \end{equation}
induces a perfect pairing 
\begin{equation} \label{fanto11C} \cup: \O(Y_L)/L\otimes^{\Box}_{L}\O(Y_L)/L\to L
 \end{equation}
between the solid $L$-vector spaces $\O(Y_L)/L$ and $\O(Y_L)/L$.

\begin{remark}\label{fanto2}
We can see one of the copies of $\O(Y_L)/L$
 as the space $\Omega^1(Y_L)_0$ of differential forms with residue equal to $0$
(via $h\mapsto d h $). The map $(f,\omega)\mapsto {\rm res}(f\omega)$
induces a perfect duality between the $L$-spaces $\O(Y_L)$ and $\Omega^1(Y_L)$
and $\Omega^1(Y_L)_0$ is exactly the orthogonal of $L\subset\O(Y_L)$. 
\end{remark}

\subsubsection{Open disc}  
Let $D$ and $Y$ be as above.  We see easily that $\so(Y)/\so(D)$ is the
  dual of $\so(D)/K$ for the pairing \eqref{nie10}, i.e., that we have a 
  perfect pairing
$$
 \cup:  \so(Y)/\so(D)\otimes^{\Box}_{\Q_p} \so(D)/K\to \Q_p.
$$
  Similarly, we check  that $\so(Y_L)/\so(D_L)$ is the
 dual of $\so(D_L)/L$ for the pairing \eqref{nie11}, i.e., that we have a 
  perfect pairing
$$
 \cup:  \so(Y_L)/\so(D_L)\otimes^{\Box}_L \so(D_L)/L\to L.
$$
     This  pairing   can be thought of as  Serre duality pairing
     \begin{align*}
     &   \cup:  H^1_{c}(D_L,\so)\otimes^{\Box}_L H^0(D_L,\Omega^1) \to L.
     \end{align*}
 This is because we have  isomorphisms (the second is just $d:\O\to\Omega^1$)
 $$\so(Y_L)/\so(D_L)\stackrel{\sim}{\to} H^1_{c}(D_L,\so),\quad  \so(D_L)/L\stackrel{\sim}{\to} H^0(D_L,\Omega^1).
 $$

\subsubsection{Open annulus}   
Let $A$ be an open annulus over $K$; let $\partial A$ be the boundary of $A$,  a disjoint union $Y_a\sqcup Y_b$ of two ghost circles. 
   We see easily that $\so(\partial A)/(\so(A)\oplus K)$ is the
  dual of $\so(A)/K$ for the pairing \eqref{nie10} taken twice and followed by the addition map $\Q_p^{\oplus 2}\stackrel{+}{\to}\Q_p$. That is, that we have a 
  perfect pairing
$$
 \cup:  \so(\partial A)/(\so(A)\oplus K)\otimes^{\Box}_{\Q_p} \so(A)/K\to \Q_p.
$$
  Similarly, we check  that $\so(\partial A_L)/(\so(A_L)\oplus L)$ is the
  dual of $\so(A_L)/L$ for the pairing \eqref{nie11}, i.e., that we have a 
  perfect pairing
$$
 \cup:  \so(\partial A_L)/(\so(A_L)\oplus L)\otimes^{\Box}_{L} \so(A_L)/L\to L.
$$
  
   This  pairing   can be seen  as induced by the  Serre duality pairing
     \begin{align*}
     &   \cup:  H^1_{c}(A_L,\so)\otimes^{\Box}_L H^0(A_L,\Omega^1) \to L.
     \end{align*}
 This is because we have  an  isomorphism
 $$\so(\partial A_L)/\so(A_L)\stackrel{\sim}{\to} H^1_{c}(A_L,\so) 
 $$
 and the  exact sequence
 $$
 0\to \so(A_L)/L\overset{d}{\to} H^0(A_L,\Omega^1)\to L\tfrac{dz}{z}\to 0.
 $$

\section{Poincar\'e duality for a ghost circle}\label{fanto0} 
Let $K$ be a finite extension of $\Q_p$. We will prove in this chapter arithmetic Poincar\'e 
duality for a ghost circle $Y$ over $K$ that will be essential in later computations. The numerology
suggests that $Y$ is a ``proper'' variety of total dimension $\frac{3}{2}$, 
hence $Y_C$ is of dimension~$\frac{1}{2}$.

\subsection{Arithmetic duality theorem}  
We start with a proof that  assumes an Explicit Reciprocity Law (which will be proved in the next section).

\begin{theorem}\label{main-arithmeticY} {\rm (Arithmetic duality)}  Let $Y$ be a ghost circle over $K$.

{\rm (i)} There is a  trace map isomorphism
$${\rm Tr}_Y: \, H^3(Y,\Q_p(2))\stackrel{\sim}{\to} \Q_p. 
$$

{\rm (ii)} Let $i,j\in\Z$.
The  pairing
\begin{equation}
\label{pairing1}
\cup:\quad H^i(Y,\Q_p(j)) \otimes_{\Q_p}^{\Box}  H^{3-i}(Y,\Q_p(2-j))\stackrel{\cup}{\longrightarrow} H^3(Y,\Q_p(2))\xrightarrow[\sim]{{\rm Tr}_Y}\Q_p
\end{equation}
is a perfect duality, i.e., we have  an  induced   isomorphism 
\begin{align*}
 & \gamma_{Y,i}: \quad H^i(Y,\Q_p(j))\stackrel{\sim}{\to} H^{3-i}(Y,\Q_p(2-j))^*.
\end{align*}
\end{theorem}
\begin{proof}
Define geometric and arithmetic trace maps as follows:
\begin{align*}
& {\rm Tr}_{Y_C}:  H^1(Y_C,\Q_p(1)){\to}\Q_p,\\
& {\rm Tr}_Y:   H^3(Y,\Q_p(2))\stackrel{\sim}{\to} H^2(\sg_K,H^1(Y_C,\Q_p(2)))\xrightarrow[\sim]{{\rm Tr}_{Y_C}(1)}H^2(\sg_K,\Q_p(1))\xrightarrow[\sim]{{\rm Tr_{K}}} \Q_p.
\end{align*}
Here, the first trace map comes from the exact sequence in  \eqref{ball2-ghost1} 
and ${\rm Tr}_{Y_C}(1)$ is an isomorphism because $H^2(\sg_K,\tfrac{\so(Y_C)}{C})=0$. 
This proves (i);
let us turn to (ii). 

\vskip1mm
$\bullet$ {\em Filtration on cohomology.} 
Let $i,j\in\Z$. There exists an ascending filtration on $H^{i}(Y,\Q_p(j))$:
\begin{equation}\label{kwakus111}
F^2_{i,j}=H^{i}(Y,\Q_p(j))\supset F^1_{i,j}\supset F^0_{i,j}\supset F^{-1}_{i,j}=0,
\end{equation}
such that we have  isomorphisms
\begin{align*}
 & F^2_{i,j}/F^1_{i,j}\simeq H^{i-1}(\sg_K, \Q_p(j-1)),\\
& F^1_{i,j}/F^0_{i,j}\simeq H^{i-1}(\sg_K, \tfrac{\so(Y_C)}{C}(j-1)),\\
  & F^0_{i,j}/F^{-1}_{i,j}\simeq H^i(\sg_K, \Q_p(j)).
\end{align*}
This follows from Section \ref{filtration1}, Lemma \ref{mist1-ghost}, and Lemma \ref{ghost0}.

\vskip1mm
$\bullet$ {\em Identification of pairings on graded pieces.} Assume the following:
\begin{theorem}{\rm (Explicit Reciprocity Law)}
\label{kwakus11}

{\rm (i)} The pairing (\ref{pairing1}) is compatible with the above filtration.
In particular $F^a_{i,j}$ and $F^b_{3-i,2-j}$ are orthogonal if $a+b\leq 1$.

{\rm (ii)}
On the associated grading the pairing  (\ref{pairing1})
yields a pairing induced by the Galois cohomology pairing and pairing (\ref{fanto11C}). 
\end{theorem}
 From claim (i), we obtain  the following commutative diagram with  exact rows (all the vertical maps are induced from pro-\'etale cup products and the  trace ${\rm Tr}_Y$):
\begin{equation}
\xymatrix@C=4mm@R=5mm{
0\ar[r] & H^{i-1}(\sg_K, \tfrac{\so(Y_C)}{C}(j-1))\ar[r] \ar[d]^{\wr} & H^{i-1}(\sg_K,H^1(Y_C,\Q_p(j)))\ar[d]^{\beta_{Y,i}}  \ar[r]  &H^{i-1}(\sg_K, \Q_p(j-1))\ar[r]  \ar[d]^{\wr}& 0\\
0\ar[r] & H^{2-i}(\sg_K,\tfrac{\so(Y_C)}{C}(1-j))^*\ar[r] & (F^1_{3-i,2-j})^* \ar[r]  & H^{3-i}(\sg_K,  \Q_p(2-j))^*\ar[r] & 0
}
\end{equation}
By Theorem~\ref{kwakus11},  claim (ii), 
the left and the right vertical arrows are  isomorphisms. 
It follows that we have an isomorphism (we skipped the indices to lighten the notation)
$$
\beta_{Y,i}: \quad (F^2/F^0)H^{i}(Y,\Q_p(j))\stackrel{\sim}{\to} (F^1 H^{3-i}(Y,\Q_p(2-j)))^*.
$$

  Similarly, we obtain the following commutative diagram with  exact rows (again, all the vertical maps are induced from pro-\'etale cup products and the trace ${\rm Tr}_Y$)
\begin{equation}
\label{DNote-ghost}
 \xymatrix@R=5mm@C=4mm{ 
   0\ar[r] & H^i(\sg_K, \Q_p(j))\ar[r]\ar[d]^{\alpha_{Y,i}}_{\wr}  & H^{i}(Y,\Q_p(j))\ar[d]^{\gamma_{Y,i}}\ar[r] &  
H^{i-1}(\sg_K, H^1(Y_C,\Q_p(j)))\ar[r]\ar[d]^{\beta_{Y,i}}_{\wr} & 0\\
    0\ar[r] & H^{2-i}(\sg_K, \Q_p(1-j))^* \ar[r] &  H^{3-i}(Y,\Q_p(2-j))^* \ar[r] &  
(F^1_{3-i, 2-j})^* \ar[r] &  0
 }
\end{equation}
By Theorem~\ref{kwakus11}, the map $\alpha_{Y,i}$ is induced by the Galois pairing hence is an  isomorphism. It follows that so is the map $\gamma_{Y,i}$, as wanted.
\end{proof}
\subsection{Proof of the Explicit  Reciprocity Law} 
The goal of the rest of the chapter is to prove Theorem~\ref{kwakus11}. The result is immediate
for $i\geq 4$ since then  all the terms are $0$. 

    For $i=0,3$, by symmetry, we may assume that $i=0$. We have 
   $$
    H^0(Y,\Q_p(j))\simeq \begin{cases} \Q_p & \mbox{ if } j=0,\\
    0  & \mbox{ if } j\neq 0;
    \end{cases}\quad     H^3(Y,\Q_p(2-j))\simeq \begin{cases} \Q_p & \mbox{ if } j=0,\\
    0 &  \mbox{ if } j\neq 0.
    \end{cases}
   $$
  Here, to compute,  $H^3$ we have used   Lemma \ref{mist1-ghost} and Lemma \ref{ghost0}.
Now  we need to study  the pairing 
   $$
   \cup:\quad H^0(Y,\Q_p)\otimes_{\Q_{p}}^{\Box} H^3(Y,\Q_p(2))\to H^3(Y,\Q_p(2))\xrightarrow[\sim]{{\rm Tr}_Y}\Q_p.
   $$
   We have
   \begin{align*}
  &  F^0_{0,0}=   F^1_{0,0}=   F^2_{0,0}=H^0(Y,\Q_p),\\
    &  F^0_{3,2}=   F^1_{3,2}=  0,\quad  F^2_{0,0}=H^3(Y,\Q_p(2)).
   \end{align*}
   Claim (i) of Theorem~\ref{kwakus11} follows immediately.  Claim (ii) 
 is easy to check by following the isomorphisms appearing in  Lemma \ref{mist1-ghost}, Lemma \ref{ghost0} and the definition of the trace map ${\rm Tr_Y}$, and by compatibility of the Hochschild-Serre spectral sequence  with products.  (The Hochschild-Serre spectral sequence for $Y$ is obtained by taking $\colim$ of the Hochschild-Serre spectral sequences for annuli \eqref{HSann}. The fact that $\colim$ commutes with Galois cohomology is proved as in the proof of Lemma \ref{ghost0}.)

   So it remains  to look at 
$i=1,2$ and, by symmetry, we may assume that $i=1$. That is, we are studying the pairing
\begin{equation}
\label{main-pairing}
\cup:\quad H^1(Y,\Q_p(j))\otimes_{\Q_{p}}^{\Box} H^2(Y,\Q_p(2-j))\to H^3(Y,\Q_p(2))\xrightarrow[\sim]{{\rm Tr}_Y}\Q_p.
\end{equation}

\subsection{Descent to $\widehat K_\infty$}\label{fanto6}
Since $\Qbar_p/ K_\infty$ is almost \'etale, 
we can compute
the pro-\'etale cohomology of $Y_K:=Y$ using $Y_{\widehat K_\infty}$ and the latter is computed via syntomic methods as in \cite{CN1} or in \cite{Sally}. 

 \subsubsection{$(\partial, \varphi,\gamma)$-Koszul complexes} 
Let $\widetilde Y^{[u,v]}=Y_{\Q_p}\times U^{[u,v]}$ (see section~\ref{tau1}). Then we have
$$\widetilde\O^{[u,v]}:=\O(\widetilde Y^{[u,v]})\simeq \O(Y_{\Q_p})\wotimes \B^{[u,v]}_{K_\infty},
\quad \widetilde\O^{[u,v]}/t\simeq \O(Y_{\widehat K_\infty}).$$
 The algebra $\O( Y_{\Q_p})$ is a direct sum of a nuclear Fr\'echet space 
and a space of compact type and  $\B^{[u,v]}_{K_\infty}$ is a Banach space.
 Let $j\in\Z$. Let us choose a uniformizer $T$ of $\so({Y_{\Q_p}})$ and define  
Frobenius $\varphi$ by sending $T$ to $T^p$ and let $\partial=T\frac{d}{dT}$. We denote by 
\begin{equation} \label{FM-Koszul}
{\rm Kos}_{\partial,\varphi,\gamma}(\widetilde\O^{[u,v]}(j))
\quad{\rm and}\quad
{\rm Kos}_{\partial,\varphi,\gamma}(\widetilde\O^{[u,v]}(j))^{\Delta_K}
\end{equation}
the total complex of the double complex
\begin{equation}
\label{FM-double}
\xymatrix@C=13mm@R=6mm{
\widetilde\O^{[u,v]}(j)\ar[d]^-{\gamma_K-1}\ar[r]^-{(t\partial,\varphi-1)}&\widetilde\O^{[u,v]}(j-1)\oplus
\widetilde\O^{[u,v/p]}(j)\ar@<-12mm>[d]^-{\gamma_K-1}\ar@<12mm>[d]^-{\gamma_K-1}
\ar[r]^-{(\varphi-1)-t\partial}&\widetilde\O^{[u,v/p]}(j-1)\ar[d]^-{\gamma_K-1}\\
\widetilde\O^{[u,v]}(j)\ar[r]^-{(t\partial,\varphi-1)}&\widetilde\O^{[u,v]}(j-1)\oplus
\widetilde\O^{[u,v/p]}(j)\ar[r]^-{(\varphi-1)-t\partial}&\widetilde\O^{[u,v/p]}(j-1)}
\end{equation}
and the complex obtained by taking fixed points by $\Delta_K$ of each of its terms\footnote{In what follows, a superscript $\Delta_K$ means taking fixed points by $\Delta_K$ 
of each terms; we don't need to take derived fixed points since $\Delta_K$ is finite 
and the modules are $L$-vector spaces. All the statements that follow continue to hold
for fixed points by $\Delta_K$ but we will not state them explicitly.}.
Since they are based on Fontaine-Messing syntomic cohomology (see Remark \ref{fanto7}), 
we will call the Koszul complexes~\eqref{FM-Koszul}  {\em FM-Koszul complexes}. 

They sometimes appear in  a different form
  \begin{equation} \label{HK-Koszul}
{\rm Kos}^{\rm HK}_{\partial,\varphi,\gamma}(\widetilde\O^{[u,v]}(j)),
\quad{\rm and}\quad
{\rm Kos}^{\rm HK}_{\partial,\varphi,\gamma}(\widetilde\O^{[u,v]}(j))^{\Delta_K},
  \end{equation}
   which we will call {\em HK-Koszul complexes} as they are based on Hyodo-Kato 
syntomic cohomology. The first one is the total complex of the double complex 
$$\xymatrix@C=13mm@R=6mm{
\widetilde\O^{[u,v]}(j-1)\ar[d]^-{\gamma_K-1}\ar[r]^-{(\partial,\frac{\varphi}{p}-1,\can )}&\widetilde\O^{[u,v]}(j-1)\oplus
\widetilde\O^{[u,v/p]}(j-1)\oplus (\widetilde\O^{[u,v]}/t)(j-1) \ar@<-25mm>[d]^-{\gamma_K-1}\ar@<-6mm>[d]^-{\gamma_K-1}\ar@<20mm>[d]^-{\gamma_K-1}
\ar[r]^-{(\varphi-1)-\partial+0}&\widetilde\O^{[u,v/p]}(j-1)\ar[d]^-{\gamma_K-1}\\
\widetilde\O^{[u,v]}(j-1)\ar[r]^-{(\partial,\frac{\varphi}{p}-1,\can)}&\widetilde\O^{[u,v]}(j-1)\oplus
\widetilde\O^{[u,v/p]}(j-1)\oplus (\widetilde\O^{[u,v]}/t)(j-1)\ar[r]^-{(\varphi-1)-\partial+0}&\widetilde\O^{[u,v/p]}(j-1)}
$$
The two Koszul complexes are related by a quasi-isomorphism in $\sd(\Q_{p,\Box})$
\begin{equation}
\label{beta-map}
\beta_j: {\rm Kos}_{\partial,\varphi,\gamma}(\widetilde\O^{[u,v]}(j))\stackrel{\sim}{\to} {\rm Kos}^{\rm HK}_{\partial,\varphi,\gamma}(\widetilde\O^{[u,v]}(j))
\end{equation}
   given by the maps $(t, {\rm Id} \oplus t\oplus 0, {\rm Id})$ 
in the top and bottom rows (same for $\Delta_K$-fixed points).
   
 \begin{remark}\label{fanto7}{\rm ({\em Relation to syntomic cohomology}.)} 
 
 (i) One can easily show that the horizontal complexes in \eqref{FM-double} compute the syntomic cohomology of
$Y_{\widehat K_\infty}$ (more exactly, $\rg_{\rm syn}(Y_{\widehat K_\infty},1)(j-1)$; the
twist $(j-1)$ does not play a role at the level of $K_\infty$ but intervenes in the computation of the arithmetic cohomology). 
The complex ${\rm Kos}_{\partial,\varphi,\gamma}(\widetilde\O^{[u,v]}(j))$ is then given by the mapping fiber
\begin{equation}
\label{ogrod11}
[\xymatrix@R=6mm{\rg_{\rm syn}(Y_{\widehat K_\infty},1)(j-1)\ar[r]^{\gamma_K-1}
&\rg_{\rm syn}(Y_{\widehat K_\infty},1)(j-1)}].
\end{equation}
By applying ${\rm Res}_{p^{-n}\Z_p}$, for $n$ big enough,
we obtain a quasi-isomorphic complex in which
$\B_{K_\infty}^{[u,v]}$ is replaced by $\B_{K_n}^{[u,v]}$
and which looks like $\rg_{\rm syn}(Y_K,1)(j-1)$ except for the arithmetic variable 
which is treated a little bit differently 
(there is an action of the entire $\Gamma_K$ and not just of its Lie algebra).

(ii) More precisely, the  double complex \eqref{FM-Koszul} can be rewritten in the following way:
\begin{equation}
\label{FM-syntomic}
\xymatrix@C=13mm@R=6mm{
F^1\widetilde\O^{[u,v]}(j-1)\ar[d]^-{\gamma_K-1}\ar[r]^-{(d,\frac{\varphi}{p}-1)}&\widetilde\Omega^{[u,v]}(j-1)\oplus
\widetilde\O^{[u,v/p]}(j-1)\ar@<-12mm>[d]^-{\gamma_K-1}\ar@<12mm>[d]^-{\gamma_K-1}
\ar[r]^-{(\frac{\varphi}{p}-1)-d}&\widetilde\Omega^{[u,v/p]}(j-1)\ar[d]^-{\gamma_K-1}\\
F^1\widetilde\O^{[u,v]}(j-1)\ar[r]^-{(d,\frac{\varphi}{p}-1)}&\widetilde\Omega^{[u,v]}(j-1)\oplus
\widetilde\O^{[u,v/p]}(j-1) \ar[r]^-{(\frac{\varphi}{p}-1)-d}&\widetilde\Omega^{[u,v/p]}(j-1).
}
\end{equation}
We remind the reader that $F^1\widetilde\O^{[u,v]}=t\widetilde\O^{[u,v]}$,
and that multiplication by $t$ is equivalent to a Tate twist by $(1)$. 
An isomorphism from the complex \eqref{FM-Koszul} to \eqref{FM-syntomic}
can be  given by the map 
\[
\left [ {\begin{array}{cccc}
a & x & y & z\\
a' & x' & y' & z'\\
\end{array} }\right  ]\Rightarrow \left [ {\begin{array}{cccc}
ta & x \frac{dT}{T} & ty & z \frac{dT}{T}\\
ta' & x' \frac{dT}{T} & ty' & z' \frac{dT}{T}\\
\end{array} }\right ]
\]

 (iii) We have chosen to twist everything in (i)  by $(1)$  
but one could twist by $(r)$ for any $r\geq 1$. 
That is, take $F^r$,  $F^{r-1}\Omega^1$, $\frac{\varphi}{p^r}$, etc. 
 \end{remark}
 \subsubsection{Products on mapping fibers}  
\label{prod}We will recall here the well-known formulas  for  products on mapping fibers that we will need (for details, see, for example, \cite[Prop. 3.1]{Niz1}). 
 
Let $A_i^{\bullet}$, $C_i^{\bullet}$ complexes of condensed $\Q_p$-vector spaces (for $i=1,2,3$) and $f_i, g_i$ morphisms of complexes $A_i^{\bullet} \to C_i^{\bullet}$. Assume that, for all $\alpha$ in $\Q_p$, there are maps
\[ \cup_{\alpha} : A^{\bullet}_1 \otimes_{\Q_p} A^{\bullet}_2 \to A^{\bullet}_3 \text{ and  } \cup_{\alpha} : C^{\bullet}_1 \otimes_{\Q_p}  C^{\bullet}_2 \to C^{\bullet}_3 \]
such that the $\cup_{\alpha}$'s are morphisms of complexes which commute with the $f_i$'s and $g_i$'s, all the $\cup_{\alpha}$'s are homotopic and we can choose the homotopies such that they commute with the $f_i$'s and $g_i$'s. 

If we take the mapping fiber
\[ D_i^{\bullet}:= [ A_i^{\bullet} \xrightarrow{f_i-g_i} C_i^{\bullet}] \] 
and, for all $\alpha \in \Q_p$,   the products
\[ \cup_{\alpha} : D^{\bullet}_1 \otimes_{\Q_p}  D^{\bullet}_2 \to D^{\bullet}_3 \] 
can be defined  (on the level of sections) by the formula
\begin{equation}
\label{FaSi} \gamma_1 \cup_{\alpha} \gamma_2 = (a_1 \cup_{\alpha} a_2, c_1 \cup_{\alpha}  w_{\alpha}(a_2) + (-1)^{\deg(a_1)} w_{1-\alpha}(a_1) \cup_{\alpha} c_2), 
\end{equation}
where $(a_i,c_i) \in A^{\bullet}_i \oplus C^{\bullet-1}_i$ represents $\gamma_i$, and, for $\beta\in\R$, $w_{\beta}=\beta f_i(a_i)+(1- \beta) g_i(a_i)$. 

Then: 

\begin{enumerate}
\item The $\cup_{\alpha}$'s are morphisms of complexes, which commute with the projections $D^{\bullet}_i \to A^{\bullet}_i$. 
\item The $\cup_{\alpha}$'s are homotopic.
\item If $\tilde{A}^{\bullet}_1$, $\tilde{C}^{\bullet}_i$, $\tilde{f}_i$, $\tilde{\cup}_{\alpha}$ are another set of data as above and $A_i^{\bullet} \to \tilde{A}_i^{\bullet}$, $C_i^{\bullet} \to \tilde{C}_i^{\bullet}$ are morphisms of complexes which commute with $\cup_{\alpha}$ and $\tilde{\cup}_{\alpha}$, then the induced morphism $D^{\bullet}_i \to \tilde{D}^{\bullet}_i$ commutes with $\cup_{\alpha}$ and $\tilde{\cup}_{\alpha}$ defined by \eqref{FaSi}.  
\item If $g_i=0$, then the  products 
$\cup_0, \cup_1 : D_1^{\bullet} \otimes_{\Q_p}  D_2^{\bullet} \to D_3^{\bullet} $ induce products $\tilde{\cup}_0, \tilde{\cup}_1 : A_1^{\bullet} \otimes_{\Q_p}  D_2^{\bullet} \to D_3^{\bullet}$ such that the following diagrams are commutative:
\[ \xymatrix@R=6mm{
A_1^{\bullet} \otimes_{\Q_p}  D_2^{\bullet} \ar[r]^-{\tilde{\cup}_0} & D_3^{\bullet} \\ D_1^{\bullet} \otimes_{\Q_p}  D_2^{\bullet} \ar[u] \ar[ru]_-{\cup_0} &
}
\hspace{2cm}
\xymatrix@R=6mm{
D_2^{\bullet} \otimes_{\Q_p}  A_1^{\bullet} \ar[r]^-{\tilde{\cup}_1} & D_3^{\bullet} \\ D_2^{\bullet} \otimes_{\Q_p}  D_1^{\bullet} \ar[u] \ar[ru]_-{\cup_1} &
}. \]
Explicitly, for $a_i \in A^{\bullet}_i$ and $c_i \in C^{\bullet}_i$, we have
\[ a_1 \tilde{\cup}_0 (a_2, c_2)= (a_1 \cup_0 a_2, (-1)^{\deg(a_1)} f_1(a_1) \cup_0 c_2), \]
\[ (a_2, c_2) \tilde{\cup}_1 a_1 = (a_2 \cup_1 a_1, c_2 \cup_1 f_1(a_1)). \]
\end{enumerate}
\subsubsection{Structures on  $(\partial, \varphi,\gamma)$-Koszul complexes} 
 We equip the complexes  ${\rm Kos}_{\partial,\varphi,\gamma}(\widetilde\O^{[u,v]}(j)) $ with the following structures: 

 ($\bullet$) {\em Products.}  Using (twice)  the formulas from Section \ref{prod}, we define  cup products ($\alpha\in\Q_p$)
$$
\cup_{\alpha}: {\rm Kos}_{\partial,\varphi,\gamma}(\widetilde\O^{[u,v]}(j_1))\otimes^{\Box}_{\Q_p} {\rm Kos}_{\partial,\varphi,\gamma}(\widetilde\O^{[u,v]}(j_2))\to {\rm Kos}_{\partial,\varphi,\gamma}(\widetilde\O^{[u,v]}(j_1+j_2)).
$$

 ($\bullet$) {\em Filtration.} 
 We note that the operators $\partial,\gamma_K$ do not change the exponents in the powers of~$T$ (and that the operator $\varphi$ sends $T$ to $T^p$). In particular, we can separate the constant term from the rest, which allows us to write
\begin{equation}
\label{write1}
{\rm Kos}_{\partial,\varphi,\gamma}(\widetilde\O^{[u,v]}(j))=
{\rm Kos}_{\partial,\varphi,\gamma}(\widetilde\O^{[u,v]}_0(j))
\oplus {\rm Kos}_{\varphi,\gamma}(\B_{K_\infty}^{[u,v]}(j))
\oplus {\rm Kos}_{\varphi,\gamma}(\B_{K_\infty}^{[u,v]}(j-1))[-1].
\end{equation}
Here the subscript  $0$ in $\widetilde\O^{[u,v]}_0$ in the first term denotes the series for which the constant term is $0$. The complex ${\rm Kos}_{\varphi,\gamma}(\B_{K_\infty}^{[u,v]}(s))$
(for $s=j,j-1$)
is the one defined in Section \ref{prelim} and is quasi-isomorphic to $\rg(\sg_K,\Q_p(s))$ (see \eqref{Galois1}).

 We define an ascending filtration on the complex ${\rm Kos}_{\partial,\varphi,\gamma}(\widetilde\O^{[u,v]}(j))$: 
 \begin{align*}
   F^{-1}:=0,\hskip2mm &F^0:={\rm Kos}_{\varphi,\gamma}(\B_{K_\infty}^{[u,v]}(j)),\\
&  F^1:={\rm Kos}_{\partial,\varphi,\gamma}(\widetilde\O^{[u,v]}_0(j))
\oplus {\rm Kos}_{\varphi,\gamma}(\B_{K_\infty}^{[u,v]}(j)),\\
&  F^2:={\rm Kos}_{\partial,\varphi,\gamma}(\widetilde\O^{[u,v]}(j)).
 \end{align*}
This filtration does not depend on the choice of $T$ (but the splitting does):
this is obvious for $F^0$ and,
on the diagram (\ref{FM-syntomic}), $F^1$ corresponds to the kernel of the residue map
on differential forms.
 We have canonical quasi-isomorphisms in $\sd(\Q_{p,\Box})$:
 \begin{align*}   
 &   (F^0/F^{-1}){\rm Kos}_{\partial,\varphi,\gamma}(\widetilde\O^{[u,v]}(j))\simeq {\rm Kos}_{\varphi,\gamma}(\B_{K_\infty}^{[u,v]}(j)),\\
& (F^1/F^0){\rm Kos}_{\partial,\varphi,\gamma}(\widetilde\O^{[u,v]}(j))\simeq {\rm Kos}_{\partial,\varphi,\gamma}(\widetilde\O^{[u,v]}_0(j)),\\
&  (F^2/F^1){\rm Kos}_{\partial,\varphi,\gamma}(\widetilde\O^{[u,v]}(j))\simeq {\rm Kos}_{\varphi,\gamma}(\B_{K_\infty}^{[u,v]}(j-1))[-1].
\end{align*}

 ($\bullet$) {\em Trace map.}  We define the trace map 
 $$
 {\rm Tr_{\rm Kos}}:\quad H^3 {\rm Kos}_{\partial,\varphi,\gamma}(\widetilde\O^{[u,v]}(2))    
 \overset{\sim}{\to}\Q_p
 $$
 as the composition 
 \begin{align*}
 {\rm Tr}_{\rm Kos}:\quad H^3 {\rm Kos}_{\partial,\varphi,\gamma}(\widetilde\O^{[u,v]}(2)) & \stackrel{\sim}{\to} H^3\big((F^2/F^1){\rm Kos}_{\partial,\varphi,\gamma}(\widetilde\O^{[u,v]}(1))\big)\\
  & \simeq H^2 {\rm Kos}_{\varphi,\gamma}(\B_{K_\infty}^{[u,v]}(1))
  \overset{\sim}{\to}\Q_p,
\end{align*}
where the last map is $\frac{1}{|\Delta_K|}{\rm res}_{\pi}\circ {\rm Tr}_{K_{\infty}/F_{\infty}}$
 (see Proposition~\ref{fanto23}).

\subsubsection{Quasi-isomorphism with pro-\'etale cohomology}
  Recall that, if we start with the complex computing the cohomology 
of the $\pi_1$ via perfectoid methods, we get the same complex\footnote{More precisely, one needs to do it for each of the annuli converging to the ghost circle and then go to the limit.} (but with slightly
different period rings) as 
${\rm Kos}_{\partial,\varphi,\gamma}(\widetilde\O^{[u,v]}(j))$ 
and with $t\partial$ replaced by $\tau-1$, where 
$\tau$ is the generator of the geometric part 
of ${\rm Aut}(\O(Y_{\widehat K_\infty}[T^{1/p^\infty}])/\O(Y))$ (this group
is the semi-direct product of $\Z_p(1)$ and $\Gamma_K$, and $\tau$ is the generator
of $\Z_p(1)$ given by our fixed choice of compatible system of roots of unity).
Passing from $\tau-1$ to $t\partial$
corresponds to passing from $\Z_p(1)$ to its Lie algebra as in~\cite{CN1} or~\cite{Sally};
change of period rings is done as in~\cite{CN1} or~\cite{Sally}.  
It follows that we have a quasi-isomorphism in $\sd({\Q_{p,\Box}})$
\begin{equation}\label{berlin1}
\alpha_j: {\rm Kos}_{\partial,\varphi,\gamma}(\widetilde\O^{[u,v]}(j))^{\Delta_K}
 \stackrel{\sim}{\to} \rg(Y_K,\Q_p(j)).
\end{equation}
\begin{remark}\label{functoriality}
{\rm (i)}
It is easy to see that the maps \eqref{berlin1} and \eqref{Galois1} are compatible.

{\rm (ii)} The following commutative diagram\footnote{In which $L$ has the same meaning as in
 Section~\ref{tau1}.}  makes it possible to assume
that $\Delta_K=1$ in the proofs:
$$
\xymatrix@C=15mm@R=6mm{
{\rm Kos}_{\partial,\varphi,\gamma}(\widetilde\O^{[u,v]}(j))
\ar[r]^{\sim}\ar@<3mm>[d]^{[\Delta_K]}
&\rg(Y_L,\Q_p(j))\ar[d]^{{\rm cor}_L^K}\\
{\rm Kos}_{\partial,\varphi,\gamma}(\widetilde\O^{[u,v]}(j))^{\Delta_K}
\ar[r]^{\sim}\ar@<3mm>[u]^{\rm id}
&\rg(Y_K,\Q_p(j))\ar@<5mm>[u]^{{\rm res}_K^L}}
$$
\end{remark}

\begin{lemma}\label{compatibility}
On cohomology level, the quasi-isomorphism $\alpha_j$ from \eqref{berlin1} is compatible with 
 products,  filtrations, and trace maps.
\end{lemma}
\begin{proof}

 (i) {\em Products.}    The easiest way to see this is to trace the geometric  part of the quasi-isomorphism $\alpha_j$ via the analog of the big commutative diagram used in the proof of Theorem 7.5 in \cite{Sally}. We choose $r$ large enough (see Remark \ref{fanto7}). The best path is via the top row (right-to-left), then all the way down, and then to the right along  the bottom row. All the morphisms used along the way are clearly compatible with cup products. That treats the case of the cohomology of the geometric $\pi_1$. 
 
 Now we apply continuous group cohomology for the Galois group of the base field. This and the subsequent almost \'etale descent are clearly compatible with products. What remains is the passage from 
 (nonhomogeneous) continuous cochains of $<\gamma_K>$  to the Koszul complexes for $\gamma_K$. But a map from the former to the latter can be given by the identity in degree $0$ and the evaluation at $\gamma_K$ in degree $1$; this is easily  checked to be compatible with products.
 
 (ii) {\em Filtrations.} This is easy to see for $F^0$: both the pro-\'etale  $F^0$ and the $(\phi,\Gamma)$-module $F^0$ come from the corresponding complexes of a point (given by $K$). The required compatibility follows then easily from functoriality of the comparison map $\alpha_j$ (see Remark \ref{functoriality}).  Moreover, $\alpha_j$ restricted to $F^0$ is an isomorphism (on cohomology level).
 
  It remains to check compatibility for $F^1$. Or, in light of the above, for $F^1/F^0$. Note that after moding out $F^0$ from both sides of the map $\alpha_j$, the pro-\'etale $F^1$ arises from the Galois (for the group $\sg_K$)  kernel of the pro-\'etale residue map 
  $${\rm res}_{\proeet}: H^1(Y_C, \Q_p(j))\to \Q_p(j-1)$$ (see  \eqref{kwakus111}). 
Recall that this map comes from the syntomic residue map 
\begin{equation}
\label{syntmap}
{\rm res}_{\synt}(j-1): H^1_{\synt}(Y_C, 1)(j-1)\to \B_{\crr}^{+,\phi=1}(j-1)
\end{equation}
(See diagram \eqref{Stein},  take $i=1$, and note that the slope of Frobenius on the Hyodo-Kato cohomology is $1$.)  
By changing the period rings, the Galois cohomology of the map~\eqref{syntmap}  can be seen as the residue map from diagram \eqref{ogrod11} to diagram
  $$
  \xymatrix@C=13mm@R=6mm{
\B_{K_{\infty}}^{[u,v]}(j-1)
\ar[d]^-{\gamma_K-1}
\ar[r]^-{\varphi-1}&\B_{K_{\infty}}^{[u,v/p]}(j-1)\ar[d]^-{\gamma_K-1}\\
\B_{K_{\infty}}^{[u,v]}(j-1)
 \ar[r]^-{\varphi-1}&\B_{K_{\infty}}^{[u,v/p]}(j-1),
}
  $$
  which is the Koszul complex representing the Galois cohomology of $\Q_p(j-1)$ shifted by $[-1]$. It is clear that the kernel of this map is  $F^1/F^0$, as wanted.
 
(iii) {\em Trace maps}.
The $(\phi,\Gamma)$-module trace ${\rm Tr}_{\rm Kos}$  is defined as the ``Galois'' cohomology 
of the $(\phi,\Gamma)$-module residue map followed by 
$\frac{1}{|\Delta_K|}{\rm res}_{\pi}\circ {\rm Tr}_{K_{\infty}/F_{\infty}}$.
 On the other hand, the pro-\'etale trace is defined as the Galois cohomology of the  pro-\'etale residue map followed by Galois cohomology trace. 
By point (ii) the first maps of the compositions agree.
 Hence it suffices to show the compatibility of 
$\frac{1}{|\Delta_K|}{\rm res}_{\pi}\circ {\rm Tr}_{K_{\infty}/F_{\infty}}$ with the Galois trace. 
But this was done in Proposition \ref{fanto23}. 
\end{proof}

\subsubsection{Identification of graded pieces}
 Via the comparison morphism $\alpha_j$ from \eqref{berlin1}, and using Lemma \ref{compatibility}, we get   isomorphisms:
 \begin{align}\label{grad1}
 &  \alpha^3_{j}:  H^i(F^2/F^1){\rm Kos}_{\partial,\varphi,\gamma}(\widetilde\O^{[u,v]}(j))\stackrel{\sim}{\to} H^{i-1}(\sg_K,\Q_p(j-1)),\\
 &   \alpha^2_j: H^i(F^1/F^0){\rm Kos}_{\partial,\varphi,\gamma}(\widetilde\O^{[u,v]}(j))\stackrel{\sim}{\to}  H^{i-1}(\sg_K,\tfrac{\so(Y_C)}{C}(j-1)),\notag \\
 & \alpha^1_j:  H^i(F^0/F^{-1}){\rm Kos}_{\partial,\varphi,\gamma}(\widetilde\O^{[u,v]}(j))\stackrel{\sim}{\to} H^{i}(\sg_K,\Q_p(j)).\notag
 \end{align}

Recall that $(F^1/F^0){\rm Kos}_{\partial,\varphi,\gamma}(\widetilde\O^{[u,v]}(j))
\simeq  H^i{\rm Kos}_{\partial,\varphi,\gamma}(\widetilde\O^{[u,v]}_0(j))$
 \begin{lemma}
\label{ident1}
We have   isomorphisms
$$H^i{\rm Kos}_{\partial,\varphi,\gamma}(\widetilde\O^{[u,v]}_0(j))\xleftarrow{\sim}
\begin{cases}
\O(Y_{K})_0 &{\text{if $j=1$ and $i=1,2$,}}\\
0 &{\text{if $j\neq 1$ or $i\neq 1,2$.}}
\end{cases}$$
\end{lemma}
\begin{proof}
We have (in $\sd({\Q_{p,\Box}})$)
$${\rm Kos}_{\partial,\varphi,\gamma}(\widetilde\O^{[u,v]}_0(j))\simeq 
\big[\xymatrix{{\rm Kos}_{\partial,\varphi}(\widetilde\O^{[u,v]}_0(j))\ar[r]^-{\gamma_K-1}
&{\rm Kos}_{\partial,\varphi}(\widetilde\O^{[u,v]}_0(j))}\big]$$
and since $t\partial: \widetilde\O^{[u,v/p]}_0\to \widetilde\O^{[u,v/p]}_0$ is
an isomorphism, we get a quasi-isomorphism in $\sd({\Q_{p,\Box}})$:
$${\rm Kos}_{\partial,\varphi}(\widetilde\O^{[u,v]}_0(j))\simeq
\big(\,\widetilde\O^{[u,v]}_0(j)\overset{t\partial}{\longrightarrow}\widetilde\O^{[u,v]}_0(j-1)\,\big).
$$
(The twist is $(j-1)$ and not $(j)$ because of the $\chi(\gamma_K)^{-1}$ appearing in the vertical arrow,
necessary to have a commutative diagram as $\gamma_K\cdot t\partial=\chi(\gamma_K)t\partial\cdot\gamma_K$.)

Since $\partial:\widetilde\O^{[u,v]}_0\to \widetilde\O^{[u,v]}_0$ is an isomorphism,
and since $\widetilde\O^{[u,v]}_0/t\simeq \O(Y_{\widehat K_\infty})_0$,
we get  an  isomorphism
\begin{equation}
\label{ident3}
H^i{\rm Kos}_{\partial,\varphi}(\widetilde\O^{[u,v]}_0(j))\stackrel{\sim}{\leftarrow}
\begin{cases}
\O(Y_{\widehat K_\infty})_0(j-1) &{\text{if $i=1$,}}\\
0 &{\text{if $i\neq 1$.}}
\end{cases}
\end{equation}
  Finally, we have  isomorphisms (the first one is the natural injection, the second one
is the cup-product with $\frac{\log\chi}{\log\chi(\gamma_K)}$):
$$H^i(\Gamma,\O(Y_{\widehat K_\infty})_0(j-1))\stackrel{\sim}{\leftarrow}
\begin{cases}
\O(Y_{K})_0 &{\text{if $j=1$ and $i=0,1$,}}\\
0 &{\text{if $j\neq 1$ or $i\neq 0,1$.}}
\end{cases}$$
 This is enough to prove  our lemma. 
\end{proof}
\begin{remark}\label{ident2} 
(i) In what follows we will use a convenient choice for the isomorphisms 
in Lemma~\ref{ident1} which we will call $\omega_i$, for $i=1,2$, and call $\omega$ 
a lift to the derived category. 
Set 
 \begin{align*}
& {\rm Kos}_{\gamma}(\O(Y_{K})_0 )  :=(\O(Y_{K})_0 \stackrel{\gamma_K-1}{\longrightarrow}\O(Y_{K})_0 )=(\O(Y_{K})_0 \stackrel{0}{\longrightarrow}\O(Y_{K})_0 ),\\
& {\rm Kos}_{\gamma}(\O(Y_{\wh{K}_{\infty}})_0 )  :=(\O(Y_{\widehat K_\infty})_0\stackrel{\gamma_K-1}{\longrightarrow}\O(Y_{\wh{K}_{\infty}})_0 ).
\end{align*}
We take the following composition (where the second map is induced by the isomorphism
$\widetilde\O^{[u,v]}/t\simeq  \O(Y_{\widehat K_\infty})$)
\begin{align*}
\kappa: {\rm Kos}_{\gamma}(\O(Y_{K})_0 )[-1] 
\lomapr{\can}{\rm Kos}_{\gamma}(\O(Y_{\wh{K}_{\infty}})_0 )[-1] 
\longrightarrow
{\rm Kos}^{\rm HK}_{\partial,\varphi,\gamma}(\widetilde\O^{[u,v]}_0(1)).
\end{align*}
The isomorphism in Lemma \ref{ident1} comes from the map $\omega: =\beta_1^{-1}\kappa$, where
$\beta_1$ is the map~(\ref{beta-map}):
$$\beta_1: {\rm Kos}_{\partial,\varphi,\gamma}(\widetilde\O^{[u,v]}_0(1))\stackrel{\sim}{\to}{\rm Kos}^{\rm HK}_{\partial,\varphi,\gamma}(\widetilde\O^{[u,v]}_0(1)).
  $$
  
  (ii) Similarly, the isomorphism from \eqref{ident3} can be lifted to the derived category. We define it as $\omega_{\infty}:=\beta_{1,\infty}^{-1}\kappa_{\infty}$, where
  \begin{align*}
 &  \kappa_{\infty}: \O(Y_{\wh{K}_{\infty}})_0[-1]  
\longrightarrow
{\rm Kos}^{\rm HK}_{\partial,\varphi}(\widetilde\O^{[u,v]}_0(1)),\\
 &  \beta_{1,\infty}: {\rm Kos}_{\partial,\varphi}(\widetilde\O^{[u,v]}_0(1))\stackrel{\sim}{\to}{\rm Kos}^{\rm HK}_{\partial,\varphi}(\widetilde\O^{[u,v]}_0(1)),
  \end{align*}
  and the second map is the restriction of $\beta_1$. 
\end{remark}

From the above, we can deduce the following result 
(compare with Lemma~\ref{mist1-ghost}; of course the splittings depend on the choice of $T$):
\begin{proposition}\label{fanto8} Let $j\in\Z$. 
We have $H^i(Y,\Q_p(j))=0$ if $i\geq 4$ and, if $i\leq 3$, then we have  isomorphisms
$$H^i(Y,\Q_p(j))\simeq 
\begin{cases}
\O(Y)_0 \oplus H^i(\sg_K,\Q_p(j))\oplus H^{i-1}(\sg_K,\Q_p(j-1))&{\text{if $j=1$ and $i=1,2$,}}\\
H^i(\sg_K,\Q_p(j))\oplus H^{i-1}(\sg_K,\Q_p(j-1)) &{\text{if $j\neq 1$ or $i\neq 1,2$.}}
\end{cases}$$
\end{proposition}

\subsection{Identification of the cup-product}\label{fanto9} 

\subsubsection{Formula for the cup-product}
We are going to use a ``basis'' $e_1,e_2,e_3$ to
write the differentials in the Koszul complex 
${\rm Kos}_{\partial,\varphi,\gamma}(\widetilde\O^{[u,v]}(j))$ 
in a way reminiscent
to differentials on differential forms (i.e, $e_i$ behaves as $dx_i$ in computations; 
we don't explicit the powers of $\chi(\gamma_K)$ involved in the different Tate twists)
\begin{align*}
dx&=(\varphi-1)x\cdot e_1+(\gamma_K-1)x\cdot e_2+t\partial x\cdot e_3,\\
d(ae_1+be_2+ce_3)&=
(t\partial b-(\gamma_K-1)c)\cdot e_2\wedge e_3
+((\varphi-1)c-t\partial a)\cdot e_1\wedge e_3\\
&\quad \quad +((\varphi-1)b-(\gamma_K-1)a)\cdot e_1\wedge e_2,\\
d(a\cdot e_2\wedge e_3+b\cdot e_1\wedge e_3+c\cdot e_1\wedge e_2)&=
((\varphi-1)a+(\gamma_K-1)b-t\partial c)\cdot e_1\wedge e_2\wedge e_3.
\end{align*}
The cup-product (we use, twice,  the formulas from Section \ref{prod} with $\alpha=1$)
$$\cup^{\rm Kos}:{\rm Kos}^1_{\partial,\varphi,\gamma}(\widetilde\O^{[u,v]}(j))\otimes^{\Box}_{\Q_p}
{\rm Kos}^2_{\partial,\varphi,\gamma}(\widetilde\O^{[u,v]}(2-j))\to 
{\rm Kos}^3_{\partial,\varphi,\gamma}(\widetilde\O^{[u,v]}(2))$$
between the terms of degree $1$ and $2$
is then given by:
\begin{equation}
\label{ciezko1}
(a\cdot e_1+b\cdot e_2+c\cdot e_3)\cup^{\rm Kos} (a'\cdot e_2\wedge e_3+b'\cdot e_1\wedge e_3+c'\cdot e_1\wedge e_2)=
(-a\cup\varphi a'+b\cup\gamma_K b'+ c\cup c')\cdot e_1\wedge e_2\wedge e_3.
\end{equation}

\subsubsection{Orthogonality and reduction to the Poitou-Tate pairing}
Consider now  the trace map
$${\rm Tr}_Y:H^3({\rm Kos}_{\partial,\varphi,\gamma}(\widetilde\O^{[u,v]}(2)))\to\Q_p.$$
Then  ${\rm Tr}_Y(x\cup^{\rm Kos}y)$ can be computed in the following way:

$\bullet$ write $x\cup^{\rm Kos}y$ as $a\cdot e_1\wedge e_2\wedge e_3$, with
$a\in \widetilde\O^{[u,v/p]}$;

$\bullet$ consider the constant term $a_0\in\B_{K_\infty}^{[u,v/p]}$ of $a$;

$\bullet$ we have ${\rm Tr}_Y(x\cup^{\rm Kos}y)={\rm Tr}_K\circ h_K^2(a_0)$, 
where ${\rm Tr}_K$ is the trace  map
defined in  Section~\ref{fanto22}. 
\vskip 2mm

The restriction of the pairing $(x,y)\mapsto {\rm Tr}_Y(x\cup^{\rm Kos}y)$ 
to $1$- and $2$-cocycles, respectively,
 factors through 
$H^1(Y_K,\Q_p(j))\otimes^{\Box}_{\Q_p} H^2(Y_K,\Q_p(2-j))$ and gives the pairing from Theorem~\ref{main-arithmeticY}.
Since only the constant term plays a role in the computation of the trace map, 
we have the following orthogonalities:
\begin{align*}
{\rm Kos}^1_{\partial,\varphi,\gamma}(\widetilde\O^{[u,v]}_0(j))&\perp
\big({\rm Kos}^2_{\varphi,\gamma}((\B_{K_\infty}^{[u,v]}(2-j))\oplus
{\rm Kos}^1_{\varphi,\gamma}((\B_{K_\infty}^{[u,v]}(1-j))\big), \\
\big({\rm Kos}^1_{\varphi,\gamma}((\B_{K_\infty}^{[u,v]}(j))\oplus
{\rm Kos}^0_{\varphi,\gamma}((\B_{K_\infty}^{[u,v]}(j-1))\big)&\perp {\rm Kos}^2_{\partial,\varphi,\gamma}(\widetilde\O^{[u,v]}_0(2-j))
\end{align*} 
(because the product 
of a series with constant term $0$ by a constant gives a series with a constant term $0$).
We also have an orthogonality:
\[
{\rm Kos}^1_{\varphi,\gamma}((\B_{K_\infty}^{[u,v]}(j))\perp
{\rm Kos}^2_{\varphi,\gamma}((\B_{K_\infty}^{[u,v]}(2-j))
\]
because all the terms 
$a\cup\varphi a'$, $b\cup\chi(\gamma_K)^{-1}\gamma_K b'$, $c\cup c'$ are $0$. 
This proves claims (i) of Theorem~\ref{kwakus11} 
(because the only statement that needs to be checked is that
$F^0H^1(Y_K,\Q_p(j))\otimes^{\Box}_{\Q_p} F^0H^3(Y_K,\Q_p(j))$ maps to $F^1H^3(Y_K,\Q_p(j))$. But, in fact, we have just shown that this pairing is $0$). 

  Finally, the restriction to the cocycles in
\begin{align*}
{\rm Kos}^1_{\varphi,\gamma}((\B_{K_\infty}^{[u,v]}(j))
&\times {\rm Kos}^1_{\varphi,\gamma}((\B_{K_\infty}^{[u,v]}(1-j)),\\
{\rm Kos}^0_{\varphi,\gamma}((\B_{K_\infty}^{[u,v]}(j-1))
&\times {\rm Kos}^1_{\varphi,\gamma}((\B_{K_\infty}^{[u,v]}(1-j)), \\
{\rm Kos}^0_{\varphi,\gamma}((\B_{K_\infty}^{[u,v]}(j-1))
&\times {\rm Kos}^2_{\varphi,\gamma}((\B_{K_\infty}^{[u,v]}(2-j)) 
\end{align*}
is the usual cup-product coming from the theory of  $(\varphi,\Gamma)$-modules and hence  is equal to the one from the Poitou-Tate duality. This proves the Galois part of claim (ii) from Theorem~\ref{kwakus11}.

\subsubsection{The perfection of the coherent  part of the pairing}
 It  remains to understand the restriction of the pairing to the cocycles in 
\begin{equation}
\label{pinokio1}
{\rm Kos}^1_{\partial,\varphi,\gamma}(\widetilde\O^{[u,v]}_0(j))\times
{\rm Kos}^2_{\partial,\varphi,\gamma}(\widetilde\O^{[u,v]}_0(2-j)).
\end{equation}

If $j\neq 1$, the cohomology groups of both terms in (\ref{pinokio1}) are $0$ and so the pairing is identically $0$. We can thus assume that $j=1$ and in that case we have, by Lemma \ref{ident1}, the isomorphisms:
$$\omega_1: \O(Y)_0\stackrel{\sim}{\to}H^1({\rm Kos}_{\partial,\varphi,\gamma}(\widetilde\O^{[u,v]}_0(1))),
\quad
\omega_2: \O(Y)_0\stackrel{\sim}{\to} H^2({\rm Kos}_{\partial,\varphi,\gamma}(\widetilde\O^{[u,v]}_0(1))).$$ 
\begin{lemma}
The following diagram commutes: 
$$
\xymatrix@R=6mm@C=15mm{
\O(Y)_0 \otimes^{\Box}_{\Q_p} \O(Y)_0 \ar[r]^{\cup^{\rm coh}} \ar@<-10mm>[d]^{\omega_1}_{\wr} 
\ar@<10mm>[d]^{\omega_2}_{\wr} & \Q_p\ar@{=}[d]\\
H^1({\rm Kos}_{\partial,\varphi,\gamma}(\widetilde\O^{[u,v]}_0(1)))
\otimes^{\Box}_{\Q_p} H^2({\rm Kos}_{\partial,\varphi,\gamma}(\widetilde\O^{[u,v]}_0(1)))
\ar[r]^-{{\rm Tr}_Y\circ\cup^{\rm Kos}}& \Q_p
}
$$
In particular, since the top pairing is perfect so is the bottom pairing.
\end{lemma}
\begin{proof}
We start with $f,g\in\O(Y)_0$, and we will construct cocycles 
$$z^1(g)\in{\rm Kos}^1_{\partial,\varphi,\gamma}(\widetilde\O^{[u,v]}_0(1)),\quad 
z^2(f)\in {\rm Kos}^2_{\partial,\varphi,\gamma}(\widetilde\O^{[u,v]}_0(1))$$ 
representing $\omega_1(g)$ and $\omega_2(f)$. 
Then we are going to show that:
 \begin{equation}
 \label{goal1}
{\rm Tr}_Y\big( {\rm cl}(z^1(g))\cup^{\rm Kos}  {\rm cl}(z^2(f))\big)=
\tfrac{1}{\log\chi(\gamma_K)}{\rm Tr}_{K/\Q_p}(g\cup^{\rm coh} f).
 \end{equation}

$\bullet$ For $g$, take  its image $\kappa(g)\in {\rm Kos}^{{\rm HK},1}_{\partial,\varphi,\gamma}(\widetilde\O^{[u,v]}_0(1))$ 
(see remark~\ref{ident2} for the definition of the map $\kappa$). 
Now, take $\tilde g\in\widetilde\O^{[u,v]}$ that lifts $g$ 
and consider the  cocycle (in ${\rm Kos}^1_{\partial,\varphi,\gamma}(\widetilde\O^{[u,v]}_0(1))$)
$$z^1(g):=-((\varphi-1)\tfrac{\tilde g}{t}(1)\cdot e_1+
(\gamma_K-1)\tfrac{\tilde g}{t}(1)\cdot e_2
+t\partial\tfrac{\tilde g}{t}\cdot e_3).$$
(The twist (1)  plays a role only  in the action of $\gamma_K$ and compensates for the action on $t$;
it follows that $(\gamma_K-1)\tfrac{\tilde g}{t}(1)\in \widetilde\O^{[u,v]}$ since
$g$ is fixed by $\gamma_K$. The cocycle $z^1(g)$ has then values in the desired group 
and is not a coboundary as $\tfrac{\tilde g}{t}\notin \widetilde\O^{[u,v]}$ when $g\neq 0$.) 
We easily check 
that ${\rm cl}(\beta_1(z^1(g)))={\rm cl}(\kappa(g))$ 
(see formula~(\ref{beta-map}) for the definition of the map $\beta_1$); hence
$\omega_1(g)$ is represented by $z^1(g)$. 

\vskip1mm
$\bullet$ For $f$, take its image $\kappa(f)\in {\rm Kos}^{{\rm HK},2}_{\partial,\varphi,\gamma}(\widetilde\O^{[u,v]}_0(1))$. Now, take 
 $\tilde f\in\widetilde\O^{[u,v]}$ that lifts $f$
and consider  the   cocycle (in ${\rm Kos}^2_{\partial,\varphi,\gamma}(\widetilde\O^{[u,v]}_0(1))$) 
$$z^2(f):=-\partial \tilde f\cdot e_2\wedge e_3-(\varphi-1)\tfrac{\tilde f}{t}(1)\cdot e_1\wedge e_2.$$
We easily check 
that ${\rm cl}(\beta_1(z^2(f)))={\rm cl}(\kappa(f))$; hence
$\omega_2(f)$ is represented by $z^2(f)$.

\vskip1mm
$\bullet$  Then (use formula \eqref{ciezko1}) 
$$z^1(g)\cup^{\rm Kos} z^2(f)=[-(\varphi-1)\tfrac{\tilde g}{t}\cdot \varphi(\partial \tilde f)+
t\partial\tfrac{\tilde g}{t}\cdot (\varphi-1)\tfrac{\tilde f}{t}]e_1\wedge e_2\wedge e_3.$$ 
We can write $\tilde f$ and $\tilde g$ as series in $T$ and  reduce to the 
case $\tilde f=\alpha T^i$ and $\tilde g=\beta T^j$, for $i,j\neq 0$, (and so $f=\theta(\alpha)T^i$, $g=\theta(\beta)T^j$).
Using $\varphi(T)=T^p$, $\varphi(t)=pt$, $\partial T^k=kT^k$, we obtain the formula
$$z^1(g)\cup^{\rm Kos} z^2(f)=
-\big[\big(\tfrac{\varphi(\beta)T^{pj}}{p\,t}-\tfrac{\beta T^j}{t}\big)i\varphi(\alpha)T^{pi}-
j\beta T^j\big(\tfrac{\varphi(\alpha)T^{pi}}{p\,t}-\tfrac{\alpha T^i}{\,t}\big)\big ] e_1\wedge e_2\wedge e_3.$$
In order to get a nonzero  constant term we need $j+pi=0$ or $i+j=0$.

--- If $j+pi=0$, the constant term is
$-(-\beta i\varphi(\alpha)-\frac{j\beta\varphi(\alpha)}{\,p})\frac{1}{t}=0$.

--- If $i+j=0$, the constant term is $-\frac{i\varphi(\alpha)\varphi(\beta)}{p\,t}-
\frac{j\alpha\beta}{\,t}=j(\varphi-1)\frac{\alpha\beta}{t}$.

It follows from Proposition~\ref{fanto5} that
$${\rm Tr}_Y(z^1(g)\cup^{\rm Kos} z^2(f))=\begin{cases}
\frac{-[K:\Q_p]}{\log\chi(\gamma_K)}\,j\,{\rm Tr}(\theta(\alpha)\theta(\beta))
 &{\text{if $i+j=0$,}}\\ 0&{\text{if $i+j\neq 0$,}}\end{cases}$$
Using that ${\rm Tr}_{K/\Q_p}=[K:\Q_p]\,{\rm Tr}$ on $K$,
we deduce that (see \eqref{nie10})
$$
{\rm Tr}_Y(z^1(g)\cup^{\rm Kos} z^2(f))
=\tfrac{-1}{\log\chi(\gamma_K)}{\rm Tr}_{K/\Q_p}( {\rm res} (f d g))
=\tfrac{1}{\log\chi(\gamma_K)} {\rm Tr}_{K/\Q_p}(g\cup^{\rm coh} f).
$$
This proves equality \eqref{goal1}, which we wanted.
\end{proof}

   We have proved that  the pairing 
$$
{\rm Tr}_Y\circ\cup^{\rm Kos}: 
 H^1( (F^1/F^0){\rm Kos}_{\partial,\varphi,\gamma}(\widetilde\O^{[u,v]}(j)))\otimes^{\Box}_{\Q_p} H^2( (F^1/F^0){\rm Kos}_{\partial,\varphi,\gamma}(\widetilde\O^{[u,v]}(2-j)))\to \Q_p
$$
is
perfect since or it is trivial or a multiple (for $j=1$) of the coherent pairing 
$$
{\rm Tr}_{K/\Q_p}\circ \cup^{\rm coh}: \tfrac{\so(Y_K)}{K}\otimes^{\Box}_{\Q_p}\tfrac{\so(Y_K)}{K}\to\Q_p.
$$
It follows, by  \eqref{grad1}, that  the pairing  
\begin{equation}
\label{proetale1}
\cup^{\proeet}: H^{0}(\sg_K,(\tfrac{\so(Y_C)}{C})(j-1))
\otimes^{\Box}_{\Q_p} H^{1}(\sg_K,(\tfrac{\so(Y_C)}{C})(1-j))\to\Q_p
\end{equation}
induced from pro-\'etale pairing is also perfect.

\subsubsection{Identification of the pairing \eqref{proetale1}} 
To finish the proof of Theorem~\ref{kwakus11},
it remains to show that, for $j=1$,  
the   pairing \eqref{proetale1} is equal to the one induced from  Galois pairing and coherent pairing:
\begin{align*}
\cup^{\rm Gal} : H^{0}(\sg_K, \tfrac{\so(Y_C)}{C})\otimes^{\Box}_{\Q_p} H^{1}(\sg_K, \tfrac{\so(Y_C)}{C})\stackrel{\cup}{\to} H^1(\sg_K,C)\xleftarrow[\sim]{\tfrac{\log \chi}{\log \chi(\gamma_K)}} K\lomapr{{\rm Tr}_{K}}\Q_p.
\end{align*}
  But, since  both Galois cohomology  groups $H^{0}(\sg_K, \tfrac{\so(Y_C)}{C})$ and $H^{1}(\sg_K, \tfrac{\so(Y_C)}{C})$ are isomorphic to $\frac{\so(Y_K)}{K}$,
  this can be checked  by pulling back these pairings to $\frac{\so(Y_K)}{K}$. 
By Proposition~\ref{night1} below, 
the pro-\'etale pairing pullbacks to the coherent pairing $\cup^{\rm coh}$; 
by Lemma \ref{bonn3} below, 
  so does the Galois-coherent pairing $\cup^{\rm Gal}$. 
This proves that $\cup^{\proeet}=\cup^{\rm Gal}$, as wanted. 

\begin{lemma}\label{bonn3}
The following diagram is commutative:
$$
\xymatrix@R=6mm{
 H^{0}(\sg_K, \tfrac{\so(Y_C)}{C})\otimes^{\Box}_{\Q_p} H^{1}(\sg_K, \tfrac{\so(Y_C)}{C})\ar[r]^-{\cup}   & H^1(\sg_K,C) &  K\ar[l]^-{\sim}_-{\frac{\log\chi}{\log \chi(\gamma_K)}}\ar@{=}[d]\ar[r]^{\rm Tr_K} & \Q_p\ar@{=}[d]\\
   \tfrac{\so(Y)}{K}\otimes^{\Box}_{\Q_p}  \tfrac{\so(Y)}{K}\ar[rr]^{\cup} \ar@<7mm>[u]^{\wr}_{\can}\ar@<-7mm>[u]^{\wr} _{ \cup \frac{\log\chi}{\log \chi(\gamma_K)}}&  &  K \ar[r]^{\rm Tr_K} & \Q_p.
}
$$
In particular, the  Galois-coherent pairing defined by the top row is perfect. 
\end{lemma}
\begin{proof} We compute. Going up and then right in the above diagram (and stopping at $K$) we get:
\begin{align*}
\{f,g\} & \to \{f,\{\sigma\mapsto \tfrac{\log \chi(\sigma)}{ \log \chi (\gamma_K)}g\}\} \xrightarrow{\cup}
   \{\sigma\mapsto {\rm res}(fd(\tfrac{\log \chi(\sigma)}{ \log \chi (\gamma_K)}g))\}\\
   & = \{\sigma\mapsto \tfrac{\log \chi(\sigma)}{ \log \chi (\gamma_K)}{\rm res}(fdg)\}.
\end{align*}
Going right and then up we get: 
\begin{align*}
\{f,g\} & \stackrel{\cup}{\to}{\rm res}(f dg)\to  \{\sigma\mapsto \tfrac{\log \chi(\sigma)}{\log \chi (\gamma_K)}{\rm res}(fdg)\},
\end{align*}
as wanted.   
\end{proof}

\begin{proposition}
\label{night1}  Let $i=1,2$. The following diagram commutes
$$
\xymatrix@R=6mm{
H^i{\rm Kos}_{\partial,\varphi,\gamma}(\widetilde\O^{[u,v]}_0(1))\ar[r]^-{\alpha^2_1}_-{\sim} & H^{i-1}(\sg_K,\tfrac{\so(Y_C)}{C})\\
\so(Y_K)/K\ar[ru]^{f_i}_{\sim}\ar[u]^{\omega_i}_{\wr},
}$$
where $f_1=\can, f_2=\tfrac{\log \chi}{\log \chi (\gamma_K)}$.
\end{proposition}
\begin{proof} Consider the  commutative diagram (we shortened ${\rm Kos}$ to ${\rm K}$, ${\rm res}_{\proeet}$ to ${\rm res}$;  removed  subscripts  from pro-\'etale cohomology; and set $s:=i-1$):
$$
\xymatrix@C=13pt@R=6mm{  H^{s}{\rm K}_{\gamma}(\so(Y_K)_0)\ar@/_30mm/[ddd]^{\omega_i}\ar[d]^{\can}_{\wr }\ar[r]^{\sim}& \so(Y_K)_0\ar[d]_{\wr}\ar@/^55mm/[ddddd]^{f_i}\\
H^{s}{\rm K}_{\gamma}(\so(Y_{\wh{K}_{\infty}})_0)\ar[d]^{\kappa}_{\wr}\ar@{=}[r]& H^{s}{\rm K}_{\gamma}( \so(Y_{\wh{K}_{\infty}})_0)\ar[d]^{\kappa_{\infty}}_{\wr} & 
H^{s}(\Gamma,  \so(Y_{\wh{K}_{\infty}})_0)\ar@/^39mm/[ldddd]^{\mu_{\infty}}_{\wr}\ar[l]_-{\gamma_K}^-{\sim} \\
H^{s+1}{\rm K}^{\rm HK}_{\partial,\varphi,\gamma}(\widetilde\O^{[u,v]}_0(1))) \ar[r]&  H^{s}{\rm K}_{\gamma_K}(H^1{\rm K}^{\rm HK}_{\partial,\varphi}(\widetilde\O^{[u,v]}_0(1))) \\
H^{s+1}{\rm K}_{\partial,\varphi,\gamma}(\widetilde\O^{[u,v]}_0(1)))\ar[u]_{\beta_1}^{\wr} \ar[r] \ar[d]^{\alpha_1}& 
 H^{s}{\rm K}_{\gamma}(H^1{\rm K}_{\partial,\varphi}(\widetilde\O^{[u,v]}_0(1))) \ar[u]_{\beta_{1,\infty}}^{\wr}\ar[d]^{\alpha_{1,\infty}} \\
H^{s+1}(Y,\Q_p(1))^{{\rm res}=0}\ar@/_10mm/[rd]^{\alpha_1^2}\ar[r] & H^{s}(\sg_K,H^1(Y_C, \Q_p(1))^{{\rm res}=0}) \\
 & H^{s}(\sg_K,  \so(Y_{C})_0)\ar[u]_-{\wr}
}
$$
The bottom square commutes by the proof of Lemma \ref{compatibility}. The rest commutes basically by constructions of the involved maps. Tracing this diagram proves our proposition.
\end{proof}

\section{Arithmetic Poincar\'e duality}
 Let  $K$ be a finite extension of $\Q_p$.  This chapter is devoted to the proof of the arithmetic Poincar\'e duality for smooth dagger curves over $K$. We start with stating the duality, then we prove it for proper curves, where it is an easy consequence of the geometric Poincar\'e duality.  After that we prove it for an open disc and an open annulus over $K$ (via a reduction to the Poincar\'e duality for the ghost circle proved earlier). This then allows us to treat the case of wide open curves (a special type of Stein curves) and we treat the case of general Stein curves by a limit argument. Finally, an analogous limit argument  yields the duality for a dagger affinoid. 
 
 \subsection{The statement of arithmetic Poincar\'e duality}The goal of this paper is to prove the following theorem: 

\begin{theorem}{\rm (Arithmetic Poincar\'e duality)} \label{main-arithmetic} Let $X$ be a smooth geometrically irreducible dagger variety of dimension $1$ over $K$. Assume that $X$ is proper, Stein, or affinoid. Then: 
\begin{enumerate}
\item
There is a natural  trace map isomorphism
$${\rm Tr}_X: \, H^4_{c}(X,\Q_p(2))\stackrel{\sim}{\to} \Q_p. 
$$
\item For $i,j\in\Z$, 
the  pairing
$$
H^i(X,\Q_p(j))\otimes_{\Q_{p}}^{\Box} H^{4-i}_{c}(X,\Q_p(2-j))\stackrel{\sim}{\to } H^4_{c}(X,\Q_p(2))\xrightarrow[\sim]{{\rm Tr}_X}\Q_p
$$
is a perfect duality, i.e., we have  induced  isomorphisms
\begin{align*}
 &\gamma_{X,i}:  H^i(X,\Q_p(j))\stackrel{\sim}{\to} H^{4-i}_{c}(X,\Q_p(2-j))^*,\\
& \gamma_{X,i}^c: H^i_{c}(X,\Q_p(j))\stackrel{\sim}{\to} H^{4-i}(X,\Q_p(2-j))^*.
\end{align*}
\end{enumerate}
\end{theorem}
Here we wrote $(-)^*$ for the internal $\Hom$ in the category of solid $\Q_p$-vector spaces.

\subsection{The case of proper curves}\label{case1} A proper smooth curve is the analytification of an algebraic smooth curve for which
Theorem \ref{main-arithmetic} has been known for quite a while\footnote{Of course, modulo certain identifications.}.
But, actually,
Theorem \ref{main-arithmetic} holds for any smooth proper geometrically irreducible  rigid analytic variety  $X$ over $K$ of dimension $d$ (see Corollary \ref{proper-arith}). This follows from the recently proved
geometric Poincar\'e duality (Theorem~\ref{proper-geom})
and  local Galois duality. 

  Recall that, if $X$ is a smooth proper variety over $K$, geometrically irreducible, then its pro-\'etale cohomology complex
$\R\Gamma(X_L,\Q_p(j))$, for $L=K,C$, has  classical cohomology and the cohomology groups $H^i(X_L,\Q_p(j))$ are finite dimensional $\Q_p$-vector spaces with their canonical Hausdorff topology (see Section \ref{morning2}). Over $C$, it satisfies Poincar\'e duality: 
\begin{theorem}\label{proper-geom}
{\rm (Zavyalov, Mann \cite[5.5.7]{Zav21}, \cite[Th. 1.1.1]{Mann})}
Let ${X}$ be a smooth proper geometrically irreducible rigid analytic variety of pure dimension $d$ over $K$. Then:
\begin{enumerate}[label=(\roman*)] 
\item There is a Galois-equivariant $\Q_p$-linear trace  map isomorphism
\[ {\rm Tr}_{X_C}: \, H^{2d}({X}_C,\Q_p(d))\stackrel{\sim}{\to} \Q_p. \]
\item For $i\in\N, j\in\Z$, the trace map $ {\rm Tr}_{X_C}$  induces a perfect Galois-equivariant pairing of finite rank $\Q_p$-vector spaces:
\[ H^i({X}_C,\Q_p(j))\otimes_{\Q_{p}}^{\Box}H^{2d-i}({X}_C,\Q_p(d-j))\stackrel{\sim}{\to} H^{2d}({X}_C,\Q_p(d)) \xrightarrow[\sim]{{\rm Tr}_{X_C}} \Q_p. \]
\end{enumerate}
\end{theorem}

 Combining it with local Galois duality we obtain: 
\begin{corollary}{\rm ({\em Arithmetic Poincar\'e duality})} \label{proper-arith} Let ${X}$ be a smooth proper geometrically irreducible  rigid analytic variety of pure dimension $d$ over $K$. Then:
\begin{enumerate}[label=(\roman*)]
\item There is a natural $\Q_p$-linear trace map isomorphism
\[ {\rm Tr}_X: \, H^{2d+2}({X},\Q_p(d+1)) \stackrel{\sim}{\to} \Q_p. \]
\item For $i\in\N, j\in\Z$, the pairing
\[ H^i({X},\Q_p(j))\otimes_{\Q_{p}}^{\Box} H^{2d+2-i}({X},\Q_p(d+1-j))\stackrel{\sim}{\to} H^{2d+2}({X},\Q_p(d+1)) \xrightarrow[\sim]{{\rm Tr}_X} \Q_p \]
is a perfect duality of finite rank  $\Q_p$-vector spaces.
\end{enumerate}
\end{corollary}

\begin{proof}
This follows from Theorem \ref{proper-geom} and from the Hochschild-Serre spectral sequence: 
\begin{equation}\label{HoSer}
E^{a,b}_2(j)  =H^a(\sg_K,H^b({X}_C,\Q_p(j)))\Rightarrow H^{a+b}({X},\Q_p(j)).
\end{equation}
The trace map ${\rm Tr}_X$ comes from the composition:  
\[ {\rm Tr}_X: \, H^{2d+2}({X},\Q_p(d+1))\simeq H^2(\sg_K, H^{2d}({X}_C,\Q_p(d+1))) \xrightarrow[\sim]{{\rm Tr}_{X_C}(1)} H^2(\sg_K,\Q_p(1)) \xrightarrow{\sim} \Q_p. \]
By Theorem \ref{proper-geom}, it is an isomorphism.

To prove the duality, note that the only nonzero terms of $E^{a,b}_2(j)$ are those with  degrees $0 \le a \le 2$ and $0 \le b \le 2d$; hence   the spectral sequence degenerates at $E_3$.
As the cup product commutes with the differentials $d_2^{a,b} : E_2^{a,b} \to E_2^{a+2, b-1}$ of \eqref{HoSer}, we get a commutative diagram with exact rows: 
 {\footnotesize \[ \xymatrix@R=6mm@C=10pt{ 
0 \ar[r] & E_{3}^{a,b}(j) \ar[r] \ar[d] & H^a(\sg_K,H^b(j)) \ar[r]^-{d} \ar[d]^{\rotatebox{90}{$\sim$}} & H^{a+2}(\sg_K,H^{b-1}(j)) \ar[r] \ar[d]^{\rotatebox{90}{$\sim$}} & E_{3}^{a+2,b-1}(j) \ar[r] \ar[d] & 0 \\
0 \ar[r] &E_{3}^{2-a,2d-b}(j^*)^* \ar[r]  & H^{2-a}(\sg_K,H^{2d-b}(j^*))^* \ar[r]^{\tilde{d}^*}  & H^{-a}(\sg_K,H^{2d{+}1{-}b}(j^*))^* \ar[r]  & E_{3}^{-a,2d{+}1{-}b}(j^*)^* \ar[r]  & 0} \]}
where $d:=d_2^{a,b}$, $\tilde{d}:=d_2^{-a,2d+1-b}$ and  we set $j^*:=d+1-j$ and $H^b(j):=H^b({X}_C,\Q_p(j))$.
The second and third vertical arrows are isomorphisms by Theorem \ref{proper-geom} and Tate's duality. We deduce that the two other vertical maps are isomorphisms as well.   

The Hochschild-Serre spectral sequence   \eqref{HoSer}  induces a descending filtration on $ H^i({X}, \Q_p(j))$: $$H^i({X}, \Q_p(j))=F^0_{i,j}\supset F^1_{i,j}\supset F^2_{i,j}\supset  F^3_{i,j}=0.$$ Since it degenerates at $E_3$, i.e., $E_3=E_{\infty}$, the above diagram  gives a perfect duality 
\[{\rm gr}^{\bullet}_F H^i({X}, \Q_p(j)) \otimes_{\Q_{p}}^{\Box} {\rm gr}^{2-\bullet}_F H^{2d+2-i}({X}, \Q_p(d+1-j)) \to {\rm gr}^{2}_F H^{2d+2}({X}, \Q_p(d+1)) \stackrel{\sim}{\to} \Q_p. \]
By devissage (via the filtration F), this perfect duality lifts to the following perfect pairings  (in the presented order):
\begin{align*}
(F^0/F^2)H^i({X},\Q_p(j)) & \otimes_{\Q_p}^{\Box}  F^1H^{2d+2-i}({X},\Q_p(d+1-j))\to \Q_p,\\
F^0H^i({X},\Q_p(j)) &  \otimes_{\Q_p}^{\Box}   F^0H^{2d+2-i}({X},\Q_p(d+1-j))\to \Q_p.
\end{align*}
 This concludes the proof.   
\end{proof}

    We specialize now to curves. Let $X$ be a  proper smooth geometrically irreducible curve  over $K$.   
  By \cite[Rem. 7.16]{AGN}, our geometric trace map ${\rm Tr}_{X_C}: H^{2}(X_C,\Q_p(1)) \to \Q_p$ is equal to the trace map used by Zavyalov in \cite{Zav21}, i.e., to the rigid analytic version of the Berkovich trace map
 (see  \cite[Sec. 5.3]{Zav21}, \cite[Sec. 7.2]{Ber}). Hence  in this case Theorem \ref{main-arithmetic} follows from Zavyalov's arithmetic duality stated in Corollary \ref{proper-arith}. 
\subsection{The case of an  open disc} We will now prove Theorem \ref{main-arithmetic} for an open disc. 
Let $D$ be an open disc $D$ over $K$. 
\begin{proposition}\label{main-arithmeticD} {\rm (Arithmetic duality for an open disc)} Theorem \ref{main-arithmetic} holds for $D$. 
\end{proposition}
\begin{proof}
(i)  The trace map $${\rm Tr}_D: \, H^4_{c}(D,\Q_p(2)){\to} \Q_p
$$ was defined in Section \ref{traces}. It will be convenient to have a different description of this map. Let $Y:=\partial D$ be the boundary of $D$, a ghost circle. Consider  the composition
$$
t_D: H^4_{c}(D,\Q_p(2))\xleftarrow[\sim]{\partial}H^3(Y,\Q_p(2))\xrightarrow[\sim]{{\rm Tr}_Y}\Q_p.
$$
 The first  strict  isomorphism holds   because $ H^i(D,\Q_p(2))=0$, for $i\geq 3$, by Lemma \ref{mist1}. 
 An alternative definition of $t_D$ is the following. First, we define the geometric trace zig-zag:
 $$
 t_{D_C}: H^2_{c}(D_C,\Q_p(1))\xleftarrow[]{\partial}H^1(Y_C,\Q_p(1))\xrightarrow[]{{\rm Tr}_{Y_C}}\Q_p.
 $$
 Then the arithmetic trace $t_D$ is obtained as a composition of $H^2(\sg_K, t_{D_C}(1))$ with ${\rm Tr}_K: H^2(\sg_K,\Q_p(1))\stackrel{\sim}{\to}\Q_p$ via the identifications coming from \eqref{PotB00} and Lemma \ref{mist1-ghost}.

  We claim that $ t_{D}={\rm Tr}_{D}$. For that,  it  is enough to prove that  the geometric traces $
 t_{D_C}, {\rm Tr}_{D_C}
 $ are equal after we apply $H^2(\sg_K, (-)(1))$. This will follow if we prove  that the following diagram commutes
   $$
   \xymatrix@R=6mm{
   H^1(Y_C,\Q_p(1))\ar[r]^-{{\rm Tr}_{Y_C}} \ar[d]^{\partial}& \Q_p\\
   H^2_{c}(D_C,\Q_p(1))\ar[ur]_-{{\rm Tr}_{D_C}}
   }
   $$
    Or, as can be seen by unwinding the definitions of trace maps,  that the following diagram commutes
   $$
 \xymatrix@R=6mm{
H^1_{\synt}(Y_C,2)\ar[r] \ar[d]^{\partial}&  (H^1_{{\rm HK}}(Y_C){\otimes}^{\Box}_{\breve{F}}\wh{\B}^+_{\st})^{N=0,\phi=p^2}\ar[d]^{\partial}\\
 H^2_{\synt,c}( D_C,2)\ar[r] &  (H^2_{{\rm HK},c}( D_C){\otimes}^{\Box}_{\breve{F}}\wh{\B}^+_{\st})^{N=0,\phi=p^2}
 }
 $$
But this follows easily from the definitions.

  (ii) We will first show that the  pairing
\begin{equation}
\label{pairing1D}
H^i(D,\Q_p(j))\otimes^{\Box}_{\Q_p} H^{4-i}_{c}(D,\Q_p(2-j))\stackrel{\cup}{\to} H^4_{c}(D,\Q_p(2))\xrightarrow[\sim]{{\rm Tr}_D}\Q_p
\end{equation} 
 induces the  isomorphism 
\begin{align}\label{cieplo1}
  \gamma_{D,i}: \quad H^i(D,\Q_p(j))\stackrel{\sim}{\to} H^{4-i}_{c}(D,\Q_p(2-j))^*.
\end{align}

  ($\bullet$) {\em Compatibility of pairings.} The pairing \eqref{pairing1D} 
is compatible with the pairing \eqref{pairing1}, i.e., the diagram 
 \begin{equation}
 \label{nie3}
 \xymatrix@R=6mm{
 H^i(D,\Q_p(j))\otimes^{\Box}_{\Q_p} H^{i^{\prime}}_{c}(D,\Q_p(j^{\prime}))\ar@<-15mm>@{^(->}[d]^{\can}\ar[r]^-{\cup} &  H^{i+i^{\prime}}_{c}(D,\Q_p(j+j^{\prime})) \\
  H^i(Y,\Q_p(j))\otimes^{\Box}_{\Q_p} H^{i^{\prime}-1}(Y,\Q_p(j^{\prime}))\ar[r]^-{\cup} \ar@<-15mm>@{->>}[u]_{\partial}&  H^{i+i^{\prime}-1}(Y,\Q_p(j+j^{\prime})) \ar@{->>}[u]^{\partial}
 }
 \end{equation}
 commutes (up to a sign): 
 for $a\in  H^i(D,\Q_p(j))(S)$ and $b\in  H^{i^{\prime}-1}(Y,\Q_p(j^{\prime}))(S)$, where $S$ is an extremally disconnected set, we have 
 $$
\partial(\can(a)\cup b)=(-1)^{i}a\cup \partial(b).
 $$
 This follows easily from the formula \eqref{FaSi} (take $\alpha=0$). The   injectivity and  surjectivity of the vertical maps in the diagram  follow from \eqref{inj21}. 

  ($\bullet$) {\em Filtration on cohomology.} By Section \ref{filtration1} and Lemma \ref{mist1},  there exists an ascending filtration on $H^{i}(D,\Q_p(j))$:
$$
F^2_{i,j}=H^{i}(D,\Q_p(j))\supset F^1_{i,j}\supset F^0_{i,j}\supset F^{-1}_{i,j}=0,
$$
such that 
\begin{align*}
 & F^1_{i,j}=F^2_{i,j}= H^{i}(D, \Q_p(j)),\quad F^1_{i,j}/F^0_{i,j}\simeq H^{i-1}(\sg_K, \tfrac{\so(D_C)}{C}(j-1)),\\
  & F^0_{i,j}/F^{-1}_{i,j}\simeq H^i(\sg_K, \Q_p(j)).
\end{align*}
\begin{lemma}
\label{injective1}The injection 
$$\can: H^{i}(D, \Q_p(j))\hookrightarrow   H^{i}(Y, \Q_p(j))$$
is strict for the given filtrations, i.e., the induced map 
\begin{equation}
\label{newton01}
(F^{s+1}/F^s)H^{i}(D, \Q_p(j))\to (F^{s+1}/F^s)H^{i}(Y, \Q_p(j)),\quad s\geq -1,
\end{equation}
is  injective.
\end{lemma}
\begin{proof}Since $(F^{s+1}/F^s)H^{i}(D, \Q_p(j))=0$ for $s\neq -1, 0$, it suffices to check the statement of the lemma for $s=-1,0$. 
For $s=-1$, it is clear. For $s=0$, we can write the map \eqref{newton01} as:
   \begin{equation}
   \label{nie251}
  \xymatrix@R=6mm{
 H^{i-1}(\sg_K,\frac{\so(D_C)}{C}(j-1))\ar[r] & 
  H^{i-1}(\sg_K,\frac{\so(Y_C)}{C}(j-1)). 
  }
  \end{equation}
Our claim now follows from  the compatible  isomorphisms \eqref{newton15}.
\end{proof}
 
  ($\bullet$) {\em Filtration on cohomology with compact support.} There exists an ascending filtration on $H^{i}_{c}(D,\Q_p(j))$:
$$
F^2_{c,i,j}=H^{i}_{c}(D,\Q_p(j))\supset F^1_{c,i,j}\supset F^0_{c,i,j}=0,
$$
such that 
\begin{align*}
  F^2_{c,i,j}/F^1_{c,i,j}\simeq H^{i-2}(\sg_K, \Q_p(j-1)),\quad F^1_{c,i,j}/F^0_{c,i,j}\simeq H^{i-2}(\sg_K, \tfrac{\so(Y_C)}{\so(D_C)}(j-1)).
\end{align*}
We can visualize it in the following way (to simplify the notation we removed the subscripts from cohomology): 
$$
\xymatrix@R=6mm{   & 0\ar[d]\\
 F^1_{c,i,j}\ar[r]^-{\sim}\ar[d] & H^{i-2}(\sg_K, (\so(Y_C)/\so(D_C))(j-1))\ar[d] \\
 F^2_{c,i,j}:=H^i_{c}(D,\Q_p(j))\ar[r]^-{\sim}  &  H^{i-2}(\sg_K, H^{2}_{c}(D_C,\Q_p(j))) \ar[d]\\
 & H^{i-2}(\sg_K,\Q_p(j-1))\ar[d]\\
&  0
}
$$
 Hence the filtration comes basically only from the syntomic  sequence. The right column is  exact by Lemma \ref{PotB0}.  
  \begin{lemma}
 \label{surjective1}The canonical maps
 $$\partial: F^{s}H^{i-1}(Y,\Q_p(j))\to F^{s}H^{i}_{c}(D,\Q_p(j)),\quad s\geq -1,
 $$
 are  surjective.
 \end{lemma}
 \begin{proof} This is clear for $s=-1, 0$ and $s\geq 2$. For $s=1$, we use the computations done earlier, to rewrite the map in the lemma as the canonical map:
\begin{equation}
\label{dlugo1}
\partial: 
  H^{2-i}(\sg_K,\tfrac{\so(Y_C)}{C}(1-j))\to H^{2-i}(\sg_K,\tfrac{\so(Y_C)}{\so(D_C)}(1-j)).
  \end{equation}
Using   the generalized Tate's  isomorphisms   \eqref{newton15} 
we see that the map in  \eqref{dlugo1} is  surjective. 
 \end{proof}

 ($\bullet$)  {\em Pairings on the graded pieces.} From diagram \eqref{nie3}, Lemma \ref{surjective1},  and Theorem~\ref{kwakus11},  it follows that the cup product pairing
 $$
\cup:\quad H^{i}(D,\Q_p(j))\otimes^{\Box}_{\Q_p}  H^{i^{\prime}}_{c}(D,\Q_p(j^{\prime}))\to H^{i+i^{\prime}}_{c}(D,\Q_p(j+j^{\prime}))
 $$
  is compatible with the above filtrations. In particular, 
 the subgroups 
\begin{align*}
&F^0_{i,j}=H^i(\sg_K, \Q_p(j))\subset H^i(D,\Q_p(j)),\\
& F^1_{c,4-i,2-j}=H^{2-i}(\sg_K, \tfrac{\so(Y_C)}{\so(D_C)}(1-j))\subset H^{4-i}_{c}(D,\Q_p(2-j))
\end{align*}
 are orthogonal.
 That orthogonality   induces the following map of  exact sequences (all the vertical maps are induced from cup products and the trace map ${\rm Tr}_D$):
\begin{equation}
\label{DNote}
 \xymatrix@R=5mm@C=4mm{
   0\ar[r] & H^i(\sg_K, \Q_p(j))\ar[r]\ar[d]^{\alpha_{D,i}} & H^{i}(D,\Q_p(j))\ar[d]^{\gamma_{D,i}}\ar[r]&  
H^{i-1}(\sg_K, \tfrac{\so(D_C)}{C}(j-1))\ar[r]\ar[d]^{\beta_{D,i}} & 0\\
    0\ar[r] &  H^{2-i}(\sg_K,  \Q_p(1-j))^*\ar[r] &  H^{4-i}_{c}(D,\Q_p(2-j))^* \ar[r] &  
H^{2-i}(\sg_K, \tfrac{\so(Y_C)}{\so(D_C)}(1-j))^*\ar[r] &  0
 }
\end{equation}

    It suffices now to show  that  maps $\alpha_{D,i}$ and $\beta_{D,i}$ are  isomorphisms. This is easy to prove for the first map because this map
  is induced from Galois cohomology pairing
 $$
 H^i(\sg_K, \Q_p(j))\otimes^{\Box}_{\Q_p} H^{2-i}(\sg_K,  \Q_p(1-j))\stackrel{\cup}{\to} H^{2}(\sg_K,  \Q_p(1))\xrightarrow[\sim]{{\rm Tr}_{K}}\Q_p
 $$
as can be seen by comparing $\alpha_{D,i}$  with the map $\alpha_{Y,i}$ -- an analog for the ghost circle $Y$ -- and evoking Theorem~\ref{kwakus11}.

  To identify  map  $\beta_D$, consider the commutative diagram (we omitted the indices of  filtrations and subscripts from cohomology): 
  $$
  \xymatrix@R=6mm{
  (F^1/F^0) H^{i}(D,\Q_p(j))\otimes^{\Box}_{\Q_p}  (F^1_{c}/F^0_{c})H^{4-i}_{c}(D,\Q_p(2-j))\ar@<-3cm>@{^(->}[d]^{\can}\ar[r]^-{\cup} & (F^2_{c}/F^1_{c})H^{4}_{c}(D,\Q_p(2)) \ar[r]^-{{\rm Tr}_D} _-{\sim}& \Q_p\\
    (F^1/F^0 )H^{i}(Y,\Q_p(j))\otimes^{\Box}_{\Q_p}  (F^1/F^0)H^{3-i}(Y,\Q_p(2-j))\ar[r]^-{\cup}\ar@<-3cm>@{->>}[u]_{\partial} &  (F^2_{c}/F^1_{c})H^{3}(Y,\Q_p(2))\ar[u]^{\wr}_{\partial} \ar[ru]_-{{\rm Tr}_Y}^-{\sim} 
  }
  $$
 Map $\beta_{D,i}$ is induced by the top pairing in this diagram and we want to show that that this pairing is perfect. Map $\can$ is   injective by Lemma \ref{injective1} and  maps $\partial$ are  surjective by Lemma \ref{surjective1}.     Identifying the graded pieces, the above commutative diagram can be rewritten as: 
   \begin{equation}
   \label{nie25}
  \xymatrix@R=6mm{
 H^{i-1}(\sg_K,\frac{\so(D_C)}{C}(j-1))\otimes^{\Box}_{\Q_p}  H^{2-i}(\sg_K,\frac{\so(Y_C)}{\so(D_C)}(1-j))\ar@<-2cm>@{^(->}[d]^{\can}\ar[r]^-{\cup^{\proeet}} & \Q_p\\
  H^{i-1}(\sg_K,\frac{\so(Y_C)}{C}(j-1)) \otimes^{\Box}_{\Q_p}    H^{2-i}(\sg_K,\frac{\so(Y_C)}{C}(1-j))\ar[ru]_-{\cup^{\proeet}}\ar[r]^-{\cup^{G-coh}}\ar@<-2cm>@{->>}[u]^{\partial}& H^{1}(\sg_K, C)\ar[u]_-{{\rm Tr}^{\prime}_{K}}  ,
  }
  \end{equation}
  where ${\rm Tr}^{\prime}_{K}$ is the composition $$H^{1}(\sg_K, C) \xleftarrow[\sim]{\frac{\log \chi}{\log \chi(\gamma_K)}}K\lomapr{{\rm Tr}_{K/\Q_p}} \Q_p.$$ 
  The right triangle commutes by  Theorem~\ref{kwakus11}.
 Going back to the definition of cohomology with compact support,  it is easy to check that the vertical maps ($\can$ and $\partial$) are induced from the canonical  coherent maps.  Moreover, using the generalized Tate's   isomorphisms \eqref{newton15}, we can rewrite the nontrivial cases of the above commutative diagram further as:
 \begin{equation}
   \label{nie255}
  \xymatrix@R=6mm{
 \frac{\so(D)}{K}\otimes^{\Box}_{\Q_p}   \frac{\so(Y)}{\so(D)}\ar@<-0.8cm>@{^(->}[d]^{\can}\ar[r]^-{\cup^{\proeet}}& \Q_p\\
\frac{\so(Y)}{K}\otimes^{\Box}_{\Q_p}    \frac{\so(Y)}{K}\ar[ru]_-{\cup^{\rm coh}}\ar[r]^-{\cup}\ar@<-0.8cm>@{->>}[u]^{\partial}& K \ar[u]_-{{\rm Tr}_{K}} .
  }
  \end{equation}
It is now clear that the top product in this diagram is  the  coherent product. Since the latter 
  is perfect, it follows that so is  the top product in diagram \eqref{nie25},    as wanted.

  The argument for the map 
\begin{align*}
  \gamma^c_{D,i}: \quad H^i_{c}(D,\Q_p(j))\stackrel{\sim}{\to} H^{4-i}(D,\Q_p(2-j))^.
\end{align*}
 is analogous. 
 \end{proof}
\subsection{The case of  an open annulus} \label{open-annulus-kolo}We will now prove Theorem \ref{main-arithmetic} for an open annulus. 
Let $A$ be an   open annulus over $K$.
\begin{proposition}\label{main-arithmeticA}{\rm (Arithmetic duality for an open annulus)}  Theorem \eqref{main-arithmetic} holds for $A$. 
\end{proposition}
\begin{proof}  Let $Y:=\partial A$ be the boundary of $A$, a disjoint union of two ghost circles $Y_a$, $Y_b$ over $K$. 

 (i)  The geometric and arithmetic trace maps are  defined as follows:
\begin{align}
&{\rm Tr}_{A_C}: H^2_{c}(A_C,\Q_p(1))\to \Q_p,\\
&{\rm Tr}_{A}: H^4_{c}(A,\Q_p(2))\stackrel{\sim}{\to} H^2(\sg_K,H^2_{c}(A_C,\Q_p(2)))\verylomapr{H^2(\sg_K,{\rm Tr}_{A_C})} \Q_p,\notag
\end{align}
where ${\rm Tr}_{A_C}$ is the map coming from \eqref{Vvk1}.  The map ${\rm Tr}_{A}$ is an isomorphism by Lemma \ref{mist5} and the vanishing of  $H^2(\sg_K,\so(\partial A_C)/(\so(A_C)\oplus C)(1))$ (see \eqref{Vvk1}).

    Alternatively, we have an   exact sequence
  \begin{equation}
  \label{ogrod1}
   0\to H^1(A_C,\Q_p(1))\to H^1(Y_{a,C},\Q_p(1))\oplus H^1(Y_{b,C},\Q_p(1))\stackrel{\partial}{\to}  H^2_{c}(A_C,\Q_p(1))\to 0.
  \end{equation}
  The trace map $ {\rm Tr}_{A_C}$  is induced from the trace map $$
  {\rm Tr}_{Y_{a,C}}+{\rm Tr}_{Y_{b,C}}: H^1(Y_{a,C},\Q_p(1))\oplus H^1(Y_{b,C},\Q_p(1))\to \Q_p.
  $$ This works because $H^1(A_C,\Q_p(1))\stackrel{\sim}{\to} \Q_p$ compatibly with the maps ${\rm Tr}_{Y_{a,C}}$ and ${\rm Tr}_{Y_{b,C}}$.
Applying $H^2(\sg_K,-)$  and ${\rm Tr}_K$ to the exact sequence \eqref{ogrod1}, we obtain that the arithmetic trace ${\rm Tr}_A$ can be defined via the maps
 $$
{\rm Tr}_A: \, H^4_{c}(A,\Q_p(2))\stackrel{\,\partial}{\twoheadleftarrow } H^3(Y_a,\Q_p(2))\oplus H^3(Y_b,\Q_p(2))\verylomapr{{\rm Tr}_{Y_a}+{\rm Tr}_{Y_b}} \Q_p. 
$$
We used here the fact that  the composition
$$ H^3(A,\Q_p(2))\to H^3(\partial A,\Q_p(2))\verylomapr{{\rm Tr}_{Y_a}+{\rm Tr}_{Y_b}}\Q_p 
$$
is $0$. 
 
  (ii) We will first show that 
the  pairings
\begin{equation}
\label{pairing1A}
H^i(A,\Q_p(j))\otimes^{\Box}_{\Q_p}  H^{4-i}_{c}(A,\Q_p(2-j))\stackrel{\cup}{\to} H^4_{c}(A,\Q_p(2))\xrightarrow[\sim]{{\rm Tr}_A}\Q_p
\end{equation}
  induce   isomorphisms 
\begin{align*}
 \gamma_{A,i}: \quad H^i(A,\Q_p(j))\stackrel{\sim}{\to} H^{4-i}_{c}(A,\Q_p(2-j))^*.
\end{align*}

 ($\bullet$) {\em Compatibility of pairings.} The pairing \eqref{pairing1A} is compatible with the pairing \eqref{pairing1}, i.e., the diagram 
 \begin{equation}
 \label{nie3A}
 \xymatrix@R=6mm{
 H^i(A,\Q_p(j))\otimes^{\Box}_{\Q_p}  H^{i^{\prime}}_{c}(A,\Q_p(j^{\prime}))\ar@<-15mm>@{^(->}[d]^{\can}\ar[r]^-{\cup} &  H^{i+i^{\prime}}_{c}(A,\Q_p(j+j^{\prime})) \\
  H^i(Y_a\sqcup Y_b,\Q_p(j))\otimes^{\Box}_{\Q_p}  H^{i^{\prime}-1}(Y_a\sqcup Y_b,\Q_p(j^{\prime}))\ar[r]^-{\cup} \ar@<-15mm>@{->>}[u]_{\partial}&  H^{i+i^{\prime}-1}(Y_a\sqcup Y_b,\Q_p(j+j^{\prime})) \ar@{->>}[u]^{\partial}
 }
 \end{equation}
 commutes (up to a sign). That is, for $a\in  H^i(A,\Q_p(j))(S)$ and $b\in  H^{i^{\prime}-1}(Y_a\sqcup Y_b,\Q_p(j^{\prime}))(S)$, where $S$ is an extremally disconnected set, we have 
 $$
\partial(\can(a)\cup b)=(-1)^ia\cup \partial(b).
 $$
 This follows easily from the formulas in Section \ref{prod}. The injectivity and surjectivity of the vertical maps in diagram \eqref{nie3A} follows from the computations in Section \ref{compact-11}. 

  ($\bullet$) {\em  Filtration on cohomology.} By Section \ref{filtration1} and Lemma \ref{mist3}, there exists an ascending filtration on $H^{i}(A,\Q_p(j))$:
$$
F^2_{i,j}=H^{i}(A,\Q_p(j))\supset F^1_{i,j}\supset F^1_{i,j}\supset F^{-1}_{i,j}=0,
$$
such that 
\begin{align*}
 & F^2_{i,j}/F^1_{i,j}\simeq  H^{i-1}(\sg_K, \Q_p(j-1)),\quad F^1_{i,j}/F^0_{i,j}\simeq H^{i-1}(\sg_K, \tfrac{\so(A_C)}{C}(j-1)),\\
  & F^0_{i,j}/F^{-1}_{i,j}\simeq H^i(\sg_K, \Q_p(j)).
\end{align*}
 \begin{lemma} \label{injectiveA} The canonical injection
   $$\can:  H^i(A,\Q_p(j))\hookrightarrow   H^i(Y_a\sqcup Y_b,\Q_p(j))$$ is strict for the filtrations, i.e., the induced map
   \begin{equation}
   \label{bitter1}
   (F^{s+1}/F^s)H^i(A,\Q_p(j))\to     (F^{s+1}/F^s) H^i(Y_a\sqcup Y_b,\Q_p(j)),\quad s\geq -1,
   \end{equation}
   is  injective.
   \end{lemma}
   \begin{proof}This is clear for $s=-1$ and $s\geq 1$.  For $s=0$, the argument is analogous to the one used in the proof of Lemma \ref{injective1}.
   \end{proof}

    ($\bullet$) {\em  Filtration on cohomology with compact support.} Similarly, there exists an ascending filtration on $H^{i}_{c}(A,\Q_p(j))$:
$$
F^2_{c,i,j}=H^{i}_{c}(A,\Q_p(j))\supset F^1_{c,i,j}\supset F^0_{c,i,j}\supset F^{-1}_{c,i,j}=0,
$$
such that 
\begin{align*}
  & F^2_{c,i,j}/F^1_{c,i,j}\simeq H^{i-2}(\sg_K, \Q_p(j-1)),\quad F^1_{c,i,j}/F^0_{c,i,j}\simeq H^{i-2}(\sg_K, \tfrac{\so(\partial A_C)}{\so(A_C)\oplus C}(j-1)),\\
  &  F^0_{c,i,j}/F^{-1}_{c,i,j}\simeq H^{i-1}(\sg_K, H^{1}_{c}(A_C,\Q_p(j)))\simeq H^{i-1}(\sg_K, \Q_p(j)).
\end{align*}
We will visualize this filtration in the following way (to simplify the notation we removed the subscripts from cohomology):
\begin{equation}
\label{nie2A}
\xymatrix@R=6mm@C=10pt{ & &0\ar[d]  & 0\ar[d]\\
0\ar[r] &  F^0_{c,i,j}:=H^{i-1}(\sg_K, H^{1}_{c}(A_C,\Q_p(j)))\ar[r] \ar@{=}[d] & F^1_{c,i,j} \ar[r] \ar[d] & H^{i-2}(\sg_K, \frac{\so(\partial A_C)}{\so(A_C)\oplus C}(j-1))\ar[d] \ar[r] & 0\\
0\ar[r] &  H^{i-1}(\sg_K, H^{1}_{c}(A_C,\Q_p(j)))\ar[r] & F^2_{c,i,j}:=H^i_{c}(A,\Q_p(j))\ar[r] \ar[d] &  H^{i-2}(\sg_K, H^{2}_{c}(A_C,\Q_p(j)))\ar[r] \ar[d]& 0 \\
& &   H^{i-2}(\sg_K,\Q_p(j-1) )\ar@{=}[r] \ar[d] &  H^{i-2}(\sg_K,\Q_p(j-1))\ar[d]\\
& & 0 & 0
}
\end{equation}
The diagram is a map of  exact sequences with   exact columns. The middle exact row comes from the filtration induced by the Hochschild-Serre spectral sequence (see Lemma \ref{mist5}). 
The right exact column is induced by the syntomic filtration  from \eqref{Vvk1}. The term $F^1_{i,j}$ is defined as the  pullback of the top right square.  
\begin{lemma} \label{surjectiveA} The map
$$\partial: F^sH^{i-1}(\partial A,\Q_p(j))\to F^sH^{i}_{c}(A,\Q_p(j)),\quad s\geq -1,
$$
is surjective.
   \end{lemma}
   \begin{proof} This is clear for $s=-1$ and $s\geq 2$. For $s=0$, we need to check strict surjectivity of the canonical map
   $$
   H^{i-1}(\sg_K, H^0(\partial A_C,\Q_p(j)))\to    H^{i-1}(\sg_K, H^1_{c}(A_C,\Q_p(j))).
   $$
   But, as follows from \eqref{Vvk1}, the canonical map
   $$
   H^0(\partial A_C,\Q_p(j))\to    H^1_{c}(A_C,\Q_p(j))
   $$
   is  surjective with a Galois equivariant section. 
   
     For $s=1$, having done the case of $s=0$, it suffices to show that the map
     $$
     \partial: \quad (F^2/F^1)H^{i-1}(\partial A,\Q_p(j))\to (F^2/F^1)H^{i}_{c}(A,\Q_p(j))
     $$
     is  surjective. This amounts to showing that the canonical map
     $$
       H^{i-2}(\sg_K,\tfrac{\so(\partial A_C)}{C}(j-1))\to  H^{i-2}(\sg_K,\tfrac{\so(\partial A_C)}{\so(A_C)\oplus C}(j-1))
     $$
     is  surjective. But, by formula \eqref{newton15} and its suitable analog, or both the domain and the target of this map are trivial or this map is isomorphic to the canonical map
     $$
     \tfrac{\so(\partial A)}{K}\to \tfrac{\so(\partial A)}{\so(A)\oplus K}
     $$
     whose strict surjectivity is clear.
   \end{proof}

  ($\bullet$) {\em  Pairings on the graded pieces.}  From diagram \eqref{nie3A}, Lemma \ref{surjectiveA}, and Theorem~\ref{kwakus11},  it follows that the cup product pairing
 $$
\cup:\quad H^{i}(A,\Q_p(j))\otimes^{\Box}_{\Q_p}   H^{i^{\prime}}_{c}(A,\Q_p(j^{\prime}))\to H^{i+i^{\prime}}_{c}(A,\Q_p(j+j^{\prime}))
 $$
  is compatible with the above filtrations.    In particular, 
 the subgroups 
\begin{align*}
&F^0_{i,j}=H^i(\sg_K, \Q_p(j))\subset H^i(A,\Q_p(j)),\\
& F^1_{c,4-i,2-j}\subset H^{4-i}_{c}(A,\Q_p(2-j))
\end{align*}
 are orthogonal and so are 
$F^0_c$ and $F^1$. Hence
we obtain  the following commutative diagram with  exact rows (all the vertical maps are induced from cup products and the trace map ${\rm Tr}_A$)
\begin{equation}
\xymatrix@R=6mm@C=4mm{
0\ar[r] & H^{i-1}(\sg_K, \tfrac{\so(A_C)}{C}(j-1))\ar[r] \ar[d]^{\gamma_{A,i}} & H^{i-1}(\sg_K,H^1(A_C,\Q_p(j)))\ar[d]^{\beta_{A,i}}  \ar[r]  &H^{i-1}(\sg_K, \Q_p(j-1))\ar[r]  \ar[d]^{\alpha_{A,i}}& 0\\
0\ar[r] & H^{2-i}(\sg_K,\tfrac{\so(\partial A_C)}{\so(A_C)\oplus C}(1-j))^*\ar[r] & (F^2_{c,4-i,2-j})^* \ar[r]  & H^{3-i}(\sg_K,  \Q_p(2-j))^*
}
\end{equation}

We claim  that the maps $\alpha_{A,i}$ and $\gamma_{A,i}$ are  isomorphisms. 
Using Theorem~\ref{kwakus11} and arguing as in the case of diagram \eqref{DNote} for the open disc, we easily check that the map  $\alpha_{A,i}$ is induced by the Galois pairing; hence it is an isomorphism. 
    
     To identify  map  $\gamma_{A,i}$, consider the commutative diagram (we omitted the indices of  filtrations and the subscripts of cohomology): 
  $$
  \xymatrix@R=6mm{
  (F^1/F^0) H^{i}(A,\Q_p(j))\otimes^{\Box}_{\Q_p}   (F^1_{c}/F^0_{c})H^{4-i}_{c}(A,\Q_p(2-j))\ar@<-3cm>@{^(->}[d]^{\can}\ar[r]^-{\cup} & (F^2_{c}/F^1_{c})H^{4}_{c}(A,\Q_p(2)) \ar[r]^-{{\rm Tr}_A} _-{\sim}& \Q_p\\
    (F^1/F^0 )H^{i}(\partial A,\Q_p(j))\otimes^{\Box}_{\Q_p}   (F^1/F^0)H^{3-i}(\partial A,\Q_p(2-j))\ar[r]^-{\cup}\ar@<-3cm>@{->>}[u]_{\partial} &  (F^2_{c}/F^1_{c})H^{3}(\partial A,\Q_p(2))\ar[u]_{\partial} 
    \ar[r]_-{{\rm Tr}_{\partial A}}^-{\sim} & \Q_p^{\oplus 2}\ar[u]_{+}
  }
  $$
 Map $\gamma_{A,i}$ is induced by the top pairing in this diagram and we want to show that that this pairing is perfect.  Map $\can$ is   injective by Lemma \ref{injectiveA} and  maps $\partial$ are  surjective by Lemma \ref{surjectiveA}.     Identifying the graded pieces, the above commutative diagram can be rewritten as: 
   \begin{equation} 
   \label{nie25A}
  \xymatrix@R=6mm{
 H^{i-1}(\sg_K,\frac{\so(A_C)}{C}(j-1))\otimes^{\Box}_{\Q_p}   H^{2-i}(\sg_K,\frac{\so(\partial A_C)}{\so(A_C)\oplus C}(1-j))\ar@<-2cm>@{^(->}[d]^{\can}\ar[r]^-{\cup} & H^{1}(\sg_K, C) \ar[r]^-{{\rm Tr}_{K}} & \Q_p\\
  H^{i-1}(\sg_K,\frac{\so(\partial A_C)}{C}(j-1)) \otimes^{\Box}_{\Q_p}    H^{2-i}(\sg_K,\frac{\so(\partial A_C)}{C}(1-j))\ar[r]^-{\cup}\ar@<-2cm>@{->>}[u]_{\partial} & H^{1}(\sg_K, C)^{\oplus 2} \ar[u]_{+}\ar[r]^-{\oplus {\rm Tr}_{K}} & \Q_p^{\oplus 2}\ar[u]_{+},
  }
  \end{equation}
  where, for now,  the cup products are the ones induces from the pro-\'etale products. 
 Going back to the definition of cohomology with compact support,  it is easy to check that the vertical maps ($\can$ and $\partial$) are induced from the canonical  coherent maps.  Moreover, using the generalized Tate's isomorphisms \eqref{newton15}, we can rewrite the nontrivial cases of the above commutative diagram further as: 
 \begin{equation}
   \label{nie255A}
  \xymatrix@R=6mm{
 \frac{\so(A)}{K}\otimes^{\Box}_{\Q_p}   \frac{\so(\partial A)}{\so(A)\oplus K}\ar@<-0.8cm>@{^(->}[d]^{\can}\ar[r]^-{\cup} & K\ar[r]^-{{\rm Tr}_{K}} & \Q_p\\
\frac{\so(\partial A)}{K}\otimes^{\Box}_{\Q_p}    \frac{\so(\partial A)}{K}\ar[r]^-{\cup}\ar@<-0.8cm>@{->>}[u]_{\partial}& K^{\oplus 2}\ar[r]^-{\oplus {\rm Tr}_{K}} \ar[u]_{+}& \Q_p^{\oplus 2}\ar[u]_{+},
  }
  \end{equation}
  where, again, the products are still the ones induced from the pro-\'etale products.  Now, 
  the identification of the bottom product follows from Theorem~\ref{kwakus11}: it is just the coherent product.  
It is now clear that the top product in this diagram is also the  coherent product. Since the latter 
  is perfect, it follows that so is  the top product in diagram \eqref{nie25A},    as wanted.

     We have shown  that we have an isomorphism
$$
\beta_{A,i}: \quad (F^2/F^0)H^{i}(A,\Q_p(j))\stackrel{\sim}{\to} (F^1 H^{4-i}_{c}(A,\Q_p(2-j)))^*.
$$
Similarly as above, we obtain the following commutative diagram with  exact rows (again, all the vertical maps are induced from cup products and the trace map ${\rm Tr}_A$)
\begin{equation}
\label{DNote-ghostA}
 \xymatrix@R=5mm@C=4mm{ 
   0\ar[r] & H^i(\sg_K, \Q_p(j))\ar[r]\ar[d]^{\tilde{\alpha}_{A,i}}  & H^{i}(A,\Q_p(j))\ar[d]^{\tilde{\gamma}_{A,i}}\ar[r] &  
H^{i-1}(\sg_K, H^1(A_C,\Q_p(j)))\ar[r]\ar[d]^{\beta_{A,i}}_{\wr}& 0\\
    0\ar[r] & H^{2-i}(\sg_K, \Q_p(1-j))^* \ar[r] &  H^{4-i}_{c}(A,\Q_p(2-j))^* \ar[r] &  
(F^2_{c,4-i, 2-j})^* \ar[r] &  0
 }
\end{equation}
 Using Theorem~\ref{kwakus11}, we easily check that the map $\tilde{\alpha}_{A,i}$ is induced by the Galois pairing; hence it is an  isomorphism. It follows that so is  the map $\tilde{\gamma}_{A,i}$, as wanted. 
 
 The arguments for the map   $$
  \gamma^c_{A,i}: \quad H^i_{c}(A,\Q_p(j)){\to} H^{4-i}(A,\Q_p(2-j))^*
  $$
 are analogous. 
\end{proof}

\subsection{The case of  wide open curves} Now we pass to a special kind of Stein curves.
\subsubsection{Definition of wide opens} A {\em wide open}  (see \cite[Sec. III]{Col} for a brief study) is a rigid analytic space isomorphic to the complement in a  proper, geometrically connected, and smooth curve
 of finitely many closed discs. Examples of wide opens include
open discs and annuli. A {\em $K$-wide open}  is a rigid analytic space  over $K$ isomorphic to the complement in a  proper, geometrically connected, and smooth curve over $K$
 of finitely many closed discs over $K$. 

We will need the following fact:
\begin{lemma}
\label{embedding1} 
Let $X$ be a wide open over $K$.  Then one can embed $X$
into a proper, geometrically connected, and smooth curve $\overline{X}$ over $K$ such that we have an admissible covering
$$
\overline{X}=X\cup_{i=1}^{m} \{D_i\},
$$
where $D_i$'s are disjoint discs 
 over $K$ with centers $\{x_i\}_{i=1}^m$, $x_i\in\overline{X}(K)$, such that the intersections $A_i:=X\cap D_i$ are open annuli. 
\end{lemma}
\begin{proof}
By definition, we can embed $X$ into $\overline{X}$ as in the lemma with complementary disjoint closed discs $\overline{D}_j$. We embed these discs $\overline{D}_j$ into  open discs $D_j$. By shrinking if necessary, 
we may insure that the open discs are disjoint as well. It is then clear that the intersections $A_j:=X\cap D_j$ are open annuli, as wanted. 
\end{proof}

\subsubsection{Theorem \ref{main-arithmetic} for wide opens}  
\begin{proposition}{\rm ({Arithmetic duality for wide opens})}\label{tea1}
Let $X$ be  a wide open over $K$. Theorem \ref{main-arithmetic} holds for $X$.
\end{proposition}
\begin{proof}
By Lemma~\ref{embedding1}, $X$ can be embedded into a proper smooth geometrically irreducible curve $\overline{X}$ over $K$. Using the same notation as in that lemma, we write $D$ for the union of the open discs $D_i$'s and $A$ for the union of the open annuli $A_i$'s (coming from the intersections of the $D_i$'s and $X$). 

   For $j\in\Z$, we have the Mayer-Vietoris distinguished triangles:
\begin{align}\label{triangles1}
 \R\Gamma(\overline{X},\Q_p(j))   \to \R\Gamma(X,\Q_p(j))  \oplus \R\Gamma(D,\Q_p(j))  \to \R\Gamma(A,\Q_p(j)), \\
\R\Gamma_{c}(A,\Q_p(j))   \to \R\Gamma_{c}(X,\Q_p(j)) \oplus \R\Gamma_{c}(D,\Q_p(j))  \to \R\Gamma_{c}(\overline{X},\Q_p(j)).\notag
\end{align}
The first one comes from analytic descent of pro-\'etale cohomology and the second one from analytic co-descent of compactly supported pro-\'etale cohomology. 

 The following lemma will be needed later to pass from derived duality to classical duality. We will write   ${\mathbb D}(-):=\R\underline{\Hom}_{\Q_p}(-,\Q_p)$ for the duality functor. 
     \begin{lemma}\label{crutch1} Let $X$ be a geometrically connected smooth Stein curve over $K$. Let $j,s\in\Z$. We have a natural isomorphism
 \[H^j{\mathbb D}(\R\Gamma_{c}({X},\Q_p(s)))\simeq H^{-j}_{c}({X},\Q_p(s))^*,\] 
\end{lemma}
\begin{proof} We have the spectral sequence
$$
E_2^{i,j}=\underline{\Ext}^i_{\Q_p}(H^{-j}_{c}({X},\Q_p(s)),\Q_p)\Rightarrow  H^{i+j}{\mathbb D}(\R\Gamma_{c}({X},\Q_p(s))).
$$
Hence, it suffices to show that
\begin{align}\label{arg1}
\underline{\Ext}^i_{\Q_p}(H^j_{c}({X},\Q_p(s)),\Q_p)\simeq\begin{cases}  H^{j}_{c}({X},\Q_p(s))^* & \mbox{ if } i=0,\\
0 & \mbox{ if } i>0.
\end{cases}
\end{align}
 Since, by Theorem \ref{final1}, the solid $\Q_p$-vector space $V^j:=H^{j}_{c}({X},\Q_p(s))$ is of compact type, it is an LS of compact type (by \cite[Cor. 3.38]{Cam}). That is, it can be written as a countable colimit of Smith spaces with injective trace class  transition maps (see \cite[Def. 3.34]{Cam}): 
$V^j\simeq \colim_n V^{j,S}_n$, where $V^{j,S}_n$'s are Smith spaces. We compute 
\begin{align*}
\R\underline{\Hom}_{\Q_p}(V^j,\Q_p) & \simeq\R\underline{\Hom}_{\Q_p}(\colim_n V^{j,S}_n,\Q_p)\simeq \R\lim_n\R\underline{\Hom}_{\Q_p}(V^{j,S}_n,\Q_p)\\
 & \simeq \R\lim_n\underline{\Hom}_{\Q_p}(V^{j,S}_n,\Q_p)  \simeq \lim_n\underline{\Hom}_{\Q_p}(V^{j,S}_n,\Q_p)\simeq\underline{\Hom}_{\Q_p}(\colim_n V^{j,S}_n,\Q_p)\\
  & \simeq 
\underline{\Hom}_{\Q_p}(V^j,\Q_p).
\end{align*}
The third quasi-isomorphism follows from the fact that Smith spaces are projective objects,  the fourth  one from Section \ref{solid-frechet}  since the pro-system $\{\underline{\Hom}_{\Q_p}(V^{j,S}_n,\Q_p)\}_{n\in\N}$ is built from  Banach spaces with compact transition maps hence is equivalent to a pro-system of Banach spaces with dense transition maps (by \cite[discussion after Prop. 16.5]{Schn}).  This proves \eqref{arg1}, as wanted. 
\end{proof}

  (i) {\em Duality map $\gamma_{X,i}$.}  Apply now  the duality functor   ${\mathbb D}(-)$ to the second triangle in \eqref{triangles1}.  We obtain a distinguished triangle 
\begin{equation}
{\mathbb D}(\R\Gamma_{c}(\overline{X},\Q_p(j)))\to  {\mathbb D}(\R\Gamma_c(X,\Q_p(j)))\oplus{\mathbb D}(\R\Gamma_c(D,\Q_p(j)))\to  {\mathbb D}(\R\Gamma_{c}(A,\Q_p(j))).
\end{equation}
Here and below, we write  $\R\Gamma(-):= \R\Gamma$,  $\R\Gamma_c(-):=\R\Gamma_{c}(-)$.
 We have  a map of distinguished triangles:  
{\small \begin{equation}
\label{vl}
\xymatrix@R=6mm{
\R\Gamma(\overline{X},\Q_p(j))\ar[r]^-{\gamma_{\overline{X}}} \ar[d] &{\mathbb D}(\R\Gamma_{c}(\overline{X},\Q_p(2-j))[4])\ar[d] \\
\R\Gamma({X},\Q_p(j))\oplus \R\Gamma({D},\Q_p(j))\ar[r]^-{{\gamma_{{X}}}\oplus {\gamma_{{D}}}}\ar[d] & {\mathbb D}(\R\Gamma_c(X,\Q_p(2-j))[4])\oplus\ar[d] {\mathbb D}(\R\Gamma_c(D,\Q_p(2-j))[4])\\
\R\Gamma(A,\Q_p(j)) \ar[r]^-{\gamma_{{A}}} & {\mathbb D}(\R\Gamma_{c}(A,\Q_p(2-j))[4])
}
\end{equation}}
The top horizontal arrow is a quasi-isomorphism by Section~\ref{case1}.  The bottom horizontal arrow is a quasi-isomorphism by  Proposition \ref{main-arithmeticA} and Lemma \ref{crutch1}.
Since, moreover, by Proposition \ref{main-arithmeticD} and Lemma \ref{crutch1}, the arrow $\gamma_D$ in  diagram \eqref{vl} is a quasi-isomorphism, it follows that the map 
\begin{equation}
\label{derived1}
 {\gamma_{{X}}}: \R\Gamma({X},\Q_p(j)) \xrightarrow{} {\mathbb D}(\R\Gamma_c({X},\Q_p(2-j))[4]), 
 \end{equation}
is a quasi-isomorphism as well. That is, for $i\in\N$, we have an induced  isomorphism
\begin{equation}
\label{pada11}
  {\gamma_{{X},i}}: H^i({X},\Q_p(j)) \xrightarrow{\sim} H^i{\mathbb D}(\R\Gamma_{c}({X},\Q_p(2-j))[4]). 
  \end{equation}
This, in combination with Lemma \ref{crutch1}, yields the isomorphism
$$
  {\gamma_{{X},i}}: H^i({X},\Q_p(j)) \xrightarrow{\sim} H^{4-i}_{c}({X},\Q_p(2-j))^*,
$$
as wanted. 
   
    (ii) {\em Duality map $\gamma^c_{X,i}$.}   
    Write  the map $\gamma_{X,i}^c$  as  the composition
\begin{equation}
\label{deszcz1}
\gamma^c_{X,i}: H^i_{c}({X},\Q_p(j))\xrightarrow[\sim]{{\rm eval}} (H^i_{c}({X},\Q_p(j))^*)^*\xrightarrow[\sim]{\gamma^*_{X,4-i}}H^{4-i}({X},\Q_p(2-j))^*.
\end{equation}
The evaluation map is an isomorphism since $H^i_{c}({X},\Q_p(j))$ is reflexive, hence LS, hence solid reflexive by \cite[Thm. 3.40]{Cam}. This proves that $\gamma^c_{X,i}$ is an isomorphism, as wanted.
\end{proof}
       
 \subsection{The case of  general Stein curves}  Finally, we are ready to treat general Stein curves. 
 \begin{proposition}{\rm ({Arithmetic duality for Stein curves})}\label{tea2}
Let $X$ be  a smooth geometrically irreducible Stein curve over $K$. Theorem \ref{main-arithmetic} holds for $X$.
\end{proposition}
\begin{proof}

 (i)  {\em Reduction step.} Take a  Stein covering $\{X_n\}_{n\in\N}$ of $X$   by dagger affinoids with adapted naive interiors $X^0_n$  of $X_n$, $n\geq 1$. Using \cite[proof of Prop. 3.3]{Col} we may choose $X^0_n$ to be wide opens over finite extensions $L_n$ of $K$.  Since the duality map 
$$
\gamma_Y: \R\Gamma(Y,\Q_p(j))\to {\mathbb D}(\R\Gamma_{c}(Y,\Q_p(2-j))[4])
$$
satisfies \'etale descent,  we know from Proposition \ref{tea1}  that it is a quasi-isomorphism for each $X^0_n$. 

  (ii) {\em Trace map.}   We can write the trace map as 
   $$
{\rm Tr}_X:\quad H^4_{c}(X,\Q_p(2))\stackrel{\sim}{\leftarrow}\colim_n H^4_{c}(X^0_n,\Q_p(2))\xrightarrow[\sim]{\colim_n{\rm Tr}_{X^0_n}}\Q_p.
$$

 (iii) {\em Duality map $\gamma_{X,i}$.}
 The duality map
$$\gamma_X: \R\Gamma(X,\Q_p(j))\to {\mathbb D}(\R\Gamma_{c}(X,\Q_p(2-j))[4])$$
can be written as the composition (we set $s:=2-j$)
\begin{align*}
 \R\Gamma(X,\Q_p(j)) &  \stackrel{\sim}{\to} \R\lim_n  \R\Gamma(X^0_n,\Q_p(j))\xrightarrow[\sim]{\gamma_{X_n}}
\R\lim_n  {\mathbb D}(\R\Gamma_{c}(X^0_n,\Q_p(s))[4])\\
  & \simeq  {\mathbb D}(\colim_n\R\Gamma_{c}(X^0_n,\Q_p(s))[4])\stackrel{\sim}{\leftarrow}  {\mathbb D}(\R\Gamma_{c}(X,\Q_p(s))[4]).
 \end{align*}
The second quasi-isomorphism follows from (i).

On cohomology level this gives us  isomorphisms
$$
\gamma_{X,i}: H^i(X,\Q_p(j))\stackrel{\sim}{\to} H^i{\mathbb D}(\R\Gamma_{c}(X,\Q_p(2-j))[4]).
$$
By Theorem \ref{final1},  $H^i_{c}(X,\Q_p(j))$ is of compact type hence
 $H^i{\mathbb D}(\R\Gamma_{c}(X,\Q_p(s)))\simeq (H^i_{c}(X,\Q_p(s)))^*$ by  the computation above (proving \eqref{arg1}). Combination of  these two observations  shows that the duality map
 $$
\gamma_{X,i}: H^i(X,\Q_p(j))\stackrel{\sim}{\to} H^{4-i}_{c}(X,\Q_p(2-j))^*
$$
 is an isomorphism.

(iv) {\em Duality map $\gamma^c_X$.} Analogous to the argument used in the proof of Proposition \ref{tea1}.
\end{proof}

 \subsection{The case of  dagger affinoid curves}  Finally,  we will  treat dagger affinoids of dimension $1$. 
  \begin{proposition}{\rm ({Arithmetic duality for dagger affinoid curves})}\label{tea3}
Let $X$ be  a smooth geometrically irreducible dagger affinoid  curve over $K$. Theorem \ref{main-arithmetic} holds for $X$.
\end{proposition}
\begin{proof}
Take  a presentation $\{X_h\}_{h\in\N}$ of the dagger structure on $X$. Denote by $X^0_h$ a naive interior of $X_h$ adapted to $\{X_h\}$. 
 
 (i) {\em Duality map $\gamma^c_{X,i}$.} 
 The duality map
\begin{equation}
\label{deszcz3}
\gamma^c_{X,i}: H^i_{c}(X,\Q_p(j))\to H^{4-i}(X,\Q_p(2-j))^*
\end{equation}
can be written as the composition 
\begin{align*}
H^i_{c}(X,\Q_p(j)) &  \stackrel{\sim}{\to} \lim_h H^i_{c}(X^0_h,\Q_p(j))\xrightarrow[\sim]{\gamma^c_{X^0_h,i}}
\lim_h  H^{4-i}(X^0_h,\Q_p(2-j))^*\\
  & \simeq  (\colim_hH^{4-i}(X^0_h,\Q_p(2-j)))^*\simeq  (H^{4-i}(X,\Q_p(2-j)))^*.
 \end{align*}
The first isomorphism follows from the fact that $\R^1\lim_h H^{i-1}_{c}(X^0_h,\Q_p(j))=0$ by Section \ref{solid-frechet} since the pro-system $\{H^{i-1}_{c}(X^0_h,\Q_p(j))\}_{h\in\N}$ is built from compact type spaces with compact transition maps (hence it is equivalent to a pro-system of Banach spaces with dense transition maps (by \cite[by discussion after Prop. 16.5]{Schn}). The second isomorphism is induced by the analog of the  isomorphism \eqref{deszcz1}. It follows that the duality map \eqref{deszcz3} is an isomorphism, as wanted.

  (ii) {\em Duality map $\gamma_{X,i}$.}  Write  the duality map
  $$\gamma_{X,i}: H^i({X},\Q_p(j))\to H^{4-i}_{c}({X},\Q_p(2-j))^*
  $$
    as  the composition
$$
\gamma_{X,i}: H^i({X},\Q_p(j))\xrightarrow[\sim]{{\rm eval}} (H^i({X},\Q_p(j))^*)^*\xrightarrow[\sim]{\gamma^{c,*}_{X,4-i}}H^{4-i}_{c}({X},\Q_p(2-j))^*.
$$
The evaluation map is an isomorphism since $H^i({X},\Q_p(j))$ is reflexive by Proposition \ref{final3}, hence LS, hence solid reflexive by \cite[Thm. 3.30]{Cam}. This proves that $\gamma_{X,i}$ is an isomorphism, as wanted.
 \end{proof}

\end{document}